\documentclass[10pt,a4paper,nosumlimits,reqno]{amsart}
\usepackage{amssymb,amsmath,amsthm,url,lmodern}
\usepackage{xcolor,graphicx}
\usepackage{type1cm}%
\usepackage[shortlabels]{enumitem}
\setlist{font=\normalfont\upshape}
\usepackage[all,cmtip]{xy}
\theoremstyle{plain}
\newtheorem{thm}{Theorem}[section]
\newtheorem{prop}[thm]{Proposition}
\newtheorem{lem}[thm]{Lemma}
\newtheorem{cor}[thm]{Corollary}
\theoremstyle{definition}
\newtheorem{defn}[thm]{Definition}
\newtheorem{exmpl}[thm]{Example}
\newtheorem{rem}[thm]{Remark}
\newtheorem{probl}[thm]{Problem}
\newtheorem*{TC}{Telescope Conjecture}
\newtheorem*{NTC}{Nonstable Telescope Conjecture}
\newtheorem*{CSprobl}{Cuadra--Simson Problem}
\makeatletter
\def\longrightleftarrows{\mathrel{
	\mathop{\vcenter{
	\offinterlineskip
	\hbox to 0.6truecm{\rightarrowfill}%
	\hbox to 0.6truecm{\leftarrowfill}}}%
	}}
\newcommand{\reflective}[1][\empty]{
	\mathrel{\mathop{\vcenter{%
	\offinterlineskip
	\hbox to 0.6truecm{\rightarrowfill}%
	\hbox to 0.6truecm{$\vphantom\gets\smash{\longleftarrow\joinrel\rhook}$}}}%
	\limits^{#1}%
	}}
\newcommand{\coreflective}[1][\empty]{%
	\mathrel{\mathop{\vcenter{%
	\offinterlineskip
	\hbox to 0.6truecm{$\vphantom\to\smash{\lhook\joinrel\to}$}%
	\hbox to 0.6truecm{\leftarrowfill}}}\limits_{#1}%
	}}
\newcommand{\doublewidetilde}[1]{{%
 \mathpalette\double@widetilde{#1}%
}}
\newcommand{\double@widetilde}[2]{%
 \sbox\z@{$\m@th#1\widetilde{#2}$}%
 \ht\z@=.9\ht\z@
 \widetilde{\box\z@}%
}

\newcommand{\ubar}[2][2]{%
	{}\underline{#2\mkern-#1mu}\mkern#1mu}
\newcommand{\wbar}[2][3]{%
	{}\mkern#1mu\overline{\mkern-#1mu#2}}

\newcommand{\what}[2][3]{%
	{}\mkern#1mu\widehat{\mkern-#1mu#2}}

\newdimen\ex@
\def\nointerlineskip{\prevdepth-\@m\p@}
\def\@projlim{%
		\mathop{\vtop{\ialign{##\crcr
		\hfil\rm lim\hfil\crcr\noalign{\nointerlineskip}\leftarrowfill\crcr
		\noalign{\nointerlineskip\kern-\ex@}\crcr}}}
}
\def\@injlim{%
		\mathop{\vtop{\ialign{##\crcr
		\hfil\rm lim\hfil\crcr\noalign{\nointerlineskip\kern1pt}\rightarrowfill\crcr
		\noalign{\nointerlineskip\kern-\ex@}\crcr}}}
}
\def\@holim{%
		\mathop{\vtop{\ialign{##\crcr
		\hfil\rm holim\hfil\crcr\noalign{\nointerlineskip}\leftarrowfill\crcr
		\noalign{\nointerlineskip\kern-\ex@}\crcr}}}
}
\def\@hocolim{%
		\mathop{\vtop{\ialign{##\crcr
		\hfil\rm holim\hfil\crcr\noalign{\nointerlineskip}\rightarrowfill\crcr
		\noalign{\nointerlineskip\kern-\ex@}\crcr}}}
}
\def\varprojlim{\mathop{\@projlim}}
\def\varinjlim{\mathop{\@injlim}}
\def\varholim{\mathop{\@holim}}
\def\varhocolim{\mathop{\@hocolim}}

\def\subsection{\@startsection{subsection}{2}%
 \z@{.5\linespacing\@plus.7\linespacing}{-.5em}%
 {\bfseries\mathversion{bold}}}
\makeatother

\let\cal=\mathcal
\def\op{\mathrm{op}}
\def\N{\mathbb{N}}
\def\Z{\mathbb{Z}}

\def\C{\mathcal{C}}

\def\ker{\operatorname{Ker}}
\def\coker{\operatorname{Coker}}
\def\colim{\operatorname{Colim}}
\def\Im{\operatorname{Im}}
\def\Tr{\operatorname{Tr}}
\def\tr{\operatorname{tr}}
\def\Rad{\operatorname{Rad}}
\def\rad{\operatorname{rad}}
\def\rann{\operatorname{r{.}ann}}
\def\lann{\operatorname{l{.}ann}}
\def\Fix{\operatorname{Fix}}
\def\Cone{\operatorname{Cone}}
\def\Mor{\operatorname{Mor}}
\let\im=\Im\relax
\def\Spec{\operatorname{Spec}}
\def\Max{\operatorname{Max}}
\def\Supp{\operatorname{Supp}}
\def\Max{\operatorname{Max}}
\def\Hom{\operatorname{Hom}}
\def\End{\operatorname{End}}
\def\Ext{\operatorname{Ext}}
\def\Tor{\operatorname{Tor}}
\def\Obj{\operatorname{Ob}}
\def\Gen{\operatorname{Gen}}
\def\Pres{\operatorname{Pres}}
\def\Sum{\operatorname{Sum}}
\def\A{\mathcal{A}}
\def\B{\mathcal{B}}

\def\Proj{\operatorname{Proj}}
\def\proj{\operatorname{proj}}
\def\Flat{\operatorname{Flat}}
\def\rTors{\mathrm{Tors}\mathchar`-}
\def\Ab{\mathrm{Ab}}
\def\rMod{\mathrm{Mod\mathchar`-}}
\def\lMod{\mathrm{\mathchar`-Mod}}
\def\biMod{\mathrm{\mathchar`-Mod\mathchar`-}}
\def\rmod{\mathrm{mod\mathchar`-}}
\def\lmod{\mathrm{\mathchar`-mod}}

\def\rstmod{\underline{\mathrm{mod}}\mathchar`-}
\def\lstmod{\mathchar`-\underline{\mathrm{mod}}}
\def\rComod{\mathrm{Comod\mathchar`-}}

\def\fg{\operatorname{fg}}
\def\fp{\operatorname{fp}}
\def\fl{\operatorname{fl}}
\def\noeth{\operatorname{noeth}}
\def\leqdef{\mathrel{\mathrel{\mathop:}=}}
\def\reqdef{\mathrel{=\mathrel{\mathop:}}}
\def\chbull{{\textstyle\boldsymbol{\cdot}}}
\let\rho\varrho
\let\Gamma\varGamma\let\Delta\varDelta\let\Theta\varTheta
\let\Lambda\varLambda\let\Xi\varXi\let\Pi\varPi
\let\Sigma\varSigma\let\Upsilon\varUpsilon\let\Phi\varPhi
\let\Psi\varPsi\let\Omega\varOmega\relax
\def\G{\mathcal G}
\newcommand{\calI}{I}
\newcommand{\calJ}{J}

\title[L.F.P.\ Grothendieck categories with a flat generator]
{Locally finitely presented Grothendieck categories with a flat generator}
\thanks{Carlos E.~Parra was supported by project ANID+FONDECYT/REGULAR+1240253. Manuel Saor\'\i n was supported by the grant PID2020-113206GB-I00, funded by the MCIN/10.13039/501100011033, and the project 22004/PI/22 of the Fundaci\'on `S\'eneca' de Murcia. A little part of this paper was done while the third named author was visiting the University of Padova in March~2025. He thanks this institution and the colleagues of its Dipartimento di Matematica for their hospitality. Simone Virili was partially supported by MINECO (grant PID2023-147110NB-I00) and by the Generalitat
de Catalunya as part of the research group ``Laboratori d'Interaccions entre Geometria, \'Algebra i Topologia
(LIGAT)'' (grant 2021SGR01015).}

\author[L.~Martini]{Lorenzo Martini}
\address[Lorenzo Martini]{%
Departamento de Matem\'atica Aplicada, Ciencia e Ingenier\'\i a de los Materiales y Tecnolog\'\i a
Electr\'onica, Universidad Rey Juan Carlos, 28933 M\'ostoles (Madrid), Spain}
\author[C.\,E.~Parra]{Carlos E.~Parra}
\address[Carlos E.~Parra]{%
Instituto de Ciencias F\'\i sicas y Matem\'aticas\\
Edificio Emilio Pug\'\i n, Campus Isla Teja\\
Universidad Austral de Chile\\
5090000 Valdivia, Chile}
\author[M.~Saor\'\i n]{Manuel Saor\'\i n}
\address[Manuel Saor\'\i n]{%
Departamento de Matem\'aticas\\
Universidad de Murcia\\
30100 Espinardo, Murcia, Spain}
\author[S.~Virili]{Simone Virili}
\address[Simone Virili]{%
Departament de Matem\`atiques\\
Universitat Aut\`onoma de Barcelona\\
08001 Bellaterra, Barcelona, Spain}	
\email{lorenzo.martini@urjc.es}
\email{carlos.parra@uach.cl}
\email{msaorinc@um.es}
\email{virili.simone@gmail.com \text{or} simone.virili@uab.cat}

\subjclass[2020]{18E10,18E35,18G25,18C35}
\keywords{Grothendieck category, locally finitely presented, trace of a projective module, flat generator}

\begin{document}\fontsize{10.5pt}{13.5pt}\selectfont

\begin{abstract}
    A problem raised by Cuadra and Simson in 2007 asks whether any locally finitely presented Grothendieck category with enough flat objects also has enough projectives. In this paper, we start from a key observation: a locally finitely presented Grothendieck category has enough flat objects if, and only if, it has exact products. This enables several equivalent reformulations of the problem, allowing us to identify a counterexample (thus providing a negative solution to the problem), while also connecting it to a classical ring-theoretical question posed by Miller in 1975, and even to the Telescope Conjecture for compactly generated triangulated categories. Moreover, we describe several classes of Grothendieck categories where the problem can be answered affirmatively. For example, we show that a locally finitely presented Grothendieck category whose category of  finitely presented objects is Krull--Schmidt has enough flats if, and only if, it is generated by a family of finitely generated projectives.
\end{abstract}

\maketitle

\section*{Introduction}%

The goal of this paper is to study in depth the Grothendieck categories mentioned in the title. Unlike module categories, where flatness is defined via the classical tensor product, the definition of flatness of objects in general Grothendieck categories requires a suitable theory of purity that should coincide with the classical one via tensor products when we restrict ourselves to module categories. By now, this general theory of purity is well understood in the context of locally finitely presented Grothendieck (see \cite{St2,GJ,AR,CB,H97,Krause3,P}), and it can be easily generalized to arbitrary Grothendieck categories, as we show in Section~\ref{sec.purity-flatness}. In this generality, a flat object is defined as one for which any epimorphism onto it is pure. But, even within a locally finitely presented Grothendieck category, the concept of a flat object is somewhat elusive. For instance, in Algebraic Geometry, it is well known that the (Grothendieck) category $\mathrm{Qcoh}(\mathbb{X})$ of quasi-coherent sheaves over a scheme $\mathbb{X}$ is locally finitely presented whenever $\mathbb{X}$ is quasi-compact and quasi-separated (see \cite[I.6.9.12]{GD}). In that setting, the standard notion of flatness is based on the tensor product of sheaves and, when defined this way, the class of flat objects is generating. However, there are several choices of a scheme $\mathbb{X}$ for which $\mathrm{Qcoh}(\mathbb{X})$ has no nonzero flat objects in the categorical sense (see \cite{ES15}).

Even so, mathematicians are inclined to believe that within a locally finitely presented Grothendieck category with enough flats, equivalently with a flat generator, one should be able to develop homology theory in patterns very close to the case of categories with enough projectives. 

Our main motivation comes from three problems raised by Juan Cuadra and Daniel Simson (see \cite[Open Problems~2.9]{CS07}). The first of them (Problem~2.9.(a)) asks for the characterization of Grothendieck categories as in the title. Motivated by the so-called Enoch's or flat cover conjecture, whose settlement for module categories (see \cite{B-EB-E}) represented a major breakthrough in Ring Theory and later on in other areas of Mathematics (see, e.g., \cite{N2}), Problem~2.9.(b) of \cite{CS07} asked whether any locally finitely presented Grothendieck category had flat covers. This problem was answered in the affirmative by Wolfgang Rump (see \cite[Corollary~5.2]{Rum10} and also \cite{CPT}), thus gathering in one several instances of those Grothendieck categories for which the flat cover conjecture had been proved (see \cite{E-GR,E-Oy,E-Es}). Note that the existence of flat covers in such a Grothendieck category does not imply the existence of enough flat objects. The third problem \cite[Problem~2.9.(c)]{CS07} asks whether any Grothendieck category as in the title has enough projectives. This problem has remained open since then and, although there was a feeling that the answer would be negative (see \cite[Section~6, Remark]{Rum10}), it still continues to call the attention of algebraists (see \cite{CI} for a very recent example). 

The main goal of this paper is to tackle in depth this last-mentioned problem of Cuadra and Simson. We then emphasize it since we will be quoting it throughout the paper:

\begin{CSprobl}
Let $\mathcal{G}$ be a locally finitely presented Grothendieck category with enough flat objects. Does it have enough projective objects?
\end{CSprobl}

The most important consequence of our results is a couple of reformulations of the problem (see Problem~\ref{probl.CS reformulated}) that implies that such a $\mathcal{G}$ is always a bilocalizing subcategory of a module category over a ring or over a small preadditive category, and that a locally finitely presented Grothendieck category has enough flat objects if, and only if, it is $\mathrm{AB}\mathchar`-4^\ast$ (see Corollary~\ref{cor.Cuadra-Simson versus Djament}). By these reformulations, we get two surprises. First, the problem is linked to some classical and some new problems that people have been studying for a while. Secondly, the problem has a negative answer in general. More specifically, we will show that the Cuadra--Simson Problem is both a particular case of a broader question by Robert Miller, who asked when an idempotent ideal of an (always associative and unital) ring is the trace of a projective module, and also equivalent to one of the questions posted online by Aur\'elien Djament (see \cite{Djament}), who asked whether an $\mathrm{AB}\mathchar`-4^\ast$ Grothendieck category must be generated by a family of finitely generated projectives, whenever it is either locally noetherian, locally coherent or, at least, locally finitely generated. In the final section of the paper, we will also establish a connection between the Cuadra--Simson Problem, and the Telescope Conjecture for general compactly generated triangulated categories. Concretely, if $\mathcal{D}$ is such a triangulated category, $\mathcal{D}^c$ is its subcategory of compact objects and $\mathbf{t}=(\mathcal{X},\mathcal{Y})$ is a smashing semiorthogonal decomposition $\mathcal{D}$, then $\mathbf{t}$ is compactly generated exactly when the bilocalizing subcategory of $\rMod \mathcal{D}^c$ associated to the idempotent ideal $\calI_\mathbf{t}$ of $\mathcal{D}^c$ defined by Krause is locally finitely presented (see Theorem~\ref{thm.CuadraSimson-Telescope}).

On what concerns the negative answer to the problem, we show that if $R$ is a ring and $I$ is an idempotent two-sided ideal of $R$ that satisfies one of the conditions (i) and (ii) below, and it is not the trace of a projective right $R$-module then  $\mathcal{G}_I$, the bilocalizing subcategory of $\rMod R$ associated to $I$, is locally finitely presented and has enough flat objects, but not enough projective objects (see Theorem~\ref{teor.counterexample-recipe} for an extended version):
\begin{enumerate}[(i)]
\item  the multiplication map $I\otimes_RI\to I$ is an isomorphism (or, equivalently, $\Tor_1^R(I,R/I)=0$), and $I$ is finitely presented as a right ideal;
\item $I$ is pure as a left ideal.
\end{enumerate}
In fact, following a suggestion by Herbera and P\v r\'\i hoda \cite{HPV}, we will show in Proposition~\ref{exmpl.Dubrovin-Puninski} that, if $R$ belongs to the class of Dubrovin--Puninski rings (studied in \cite{DP}, and in the appendix of \cite{B-H-P-S-T*}), then it has a two-sided idempotent ideal $I$ that satisfies both conditions (i) and (ii) above and, moreover, it is not the trace of a projective right $R$-module, thus solving the Cuadra--Simson Problem in the negative. 
In fact, since $R$ is right coherent, $\mathcal{G}_I$ is even locally coherent (see \cite{H97}) so our example provides a negative answer to two of the three questions of Djament, while the last one about locally noetherian categories remains open.

But the negative solution of the general Cuadra--Simson Problem is not the end, as it is still interesting to study its restriction to suitable classes of categories. In this direction, we show that the problem has an affirmative solution when $\mathcal{G}$ is a bilocalizing subcategory of the module category over: (i) a commutative ring (Theorem~\ref{teor.CS-commutative case}); (ii) a small additive category all of whose objects have semiregular endomorphism rings (Theorem~\ref{teor.conjecture true for semiregular-quotients}). As a consequence of this last result, we get an affirmative answer when $\mathcal{G}$ is locally finite, in particular when $\mathcal{G}=\rComod C$  is the category of comodules over a $K$-coalgebra, where $K$ is a field (see Corollary~\ref{thm.Cuadra-Simson for comodules}). This was the problem that originally motivated Cuadra and Simson.

We now give a summary of the contents of the paper, in which we will point out the most important results. In Section~1 we recall a few basic facts and definitions about Grothendieck categories, torsion pairs, and t-structures, while Section~2 is devoted to recall all the preliminary results about modules over small preadditive categories. In Section~3, we define purity and (categorical) flatness in an arbitrary Grothendieck category $\mathcal{G}$, show their preservation by certain adjoint functors and prove that those concepts coincide with the ones defined via tensor product when $\mathcal{G}=\rMod \mathcal{A}$ is the module category over a small preadditive category $\mathcal{A}$. To do that, we have taken some time in defining and explaining the properties of the tensor product of modules over a small preadditive category for which we did not find an appropriate reference. In Section~4 we study the torsion torsionfree (TTF) triple $(\mathcal{C}_{\calI},\mathcal{T}_{\calI},\mathcal{F}_{\calI})$ and the corresponding recollement of abelian categories associated to an idempotent ideal $\calI$ of a small preadditive category $\mathcal{A}$. We show that, as done by Ohtake (\cite{Oh}) for module categories over rings, also here the colocalization functor is given as the tensor product by a certain $\mathcal{A}\mathchar`-\mathcal{A}$-bimodule from which we get the crucial point for our purposes that if $j^*\colon \rMod \mathcal{A}\to (\rMod \mathcal{A})/\mathcal{T}_{\calI}$ is the associated Gabriel quotient functor and $j_!$ is its fully faithful left adjoint, then $\Im(j_!)\reqdef\mathcal{X}_{\calI}$ consists of the $\mathcal{A}$-modules $X$ such that the multiplication map $X\otimes_\mathcal{A}\calI\to X$ is an isomorphism (Proposition~\ref{prop.colocalization as tensor product}). 

In Section~5 we find the desired reduction of the Cuadra--Simson Problem showing that if $\mathcal{G}$ is locally finitely presented and has enough flat objects, then it is $\mathrm{AB}\mathchar`-6$ and $\mathrm{AB}\mathchar`-4^\ast$ (Proposition~\ref{prop.Carlos result}), which implies, by an old result of Roos, that there are a small preadditive category (even a ring) $\mathcal{A}$ and an idempotent ideal $\calI$ of $\mathcal{A}$ such that $\mathcal{G}$ is equivalent to $\mathcal{X}_{\calI}$. We then study the concepts relevant for the Cuadra--Simson Problem within the category $\mathcal{X}_{\calI}$, showing that this category always has enough flat objects, has enough projectives (resp., a set of finitely projective generators) exactly when $\calI$ is the trace of a projective (resp., a set of finitely generated projective) $\mathcal{A}$-modules and is locally finitely presented exactly when $\mathcal{C}_{\calI}$ is generated by $\mathcal{X}_{\calI}\cap\rmod \mathcal{A}$, where $\rmod \mathcal{A}$ is the subcategory of finitely presented $\mathcal{A}$-modules (Theorem~\ref{thm.always enough flats}). We end the section by explicitly reformulating the conjecture (see Problem~\ref{probl.CS reformulated}) in two equivalent ways, something that we emphasize here in an equivalent way:

\begin{CSprobl}[Reformulation~1]
Let $\mathcal{A}$ be a small preadditive category (with just one object) and $\calI$ be an idempotent ideal of $\mathcal{A}$. Suppose that the Gabriel quotient $(\rMod \mathcal{A})/\mathcal{T}_I$ is locally finitely presented. Does it have enough projectives?
\end{CSprobl}

\begin{CSprobl}[Reformulation~2]
Let $\mathcal{G}$ be a locally finitely presented $\mathrm{AB}\mathchar`-4^\ast$ Grothendieck category. Does $\mathcal{G}$ have enough projectives?
\end{CSprobl}

In Section~6 we give some partial affirmative answers to the Cuadra--Simson Problem. Taking the last-mentioned reformulation 1, we show that it has an affirmative answer whenever $\mathcal{A}=R$ is a commutative ring (Theorem~\ref{teor.CS-commutative case}) or when all the objects in the additive closure $\what{\mathcal{A}}$\, have a semiregular endomorphism ring (Theorem~\ref{teor.conjecture true for semiregular-quotients}).
As a byproduct, we show that if $\mathcal{G}$ is any locally finitely presented Grothendieck category in which any finitely presented object has a semiregular endomorphism ring (e.g., if $\fp(\mathcal{G})$ is Krull-Schmidt) then $\mathcal{G}$ has enough flat objects if, and only if, it has a set of finitely generated projective generators. This applies in particular to the case of the category of comodules over a $K$-coalgebra, where $K$ is a field, thus answering in the affirmative Cuadra--Simson Problem for these categories. 

The fact that every finitely presented $\mathcal{A}$-module is the cokernel of a morphism between finite coproducts of representable $\mathcal{A}$-modules directly leads to the study of the morphisms $\alpha$ in $\what{\mathcal{A}}$ that produce an object in $\fp(\mathcal{X}_I)=\mathcal{X}_{\calI}\cap\rmod \mathcal{A}$, something we do in Section~7. Those morphisms are called fp-detecting and it is shown that if $\what{b}\in\Obj(\what{\mathcal{A}})$ and $\eta$ is an endomorphism of $\what{b}$ that lies in the ideal $\calI$, then $1_{\what{b}}-\eta$ is fp-detecting and $\mathcal{X}_{\calI}\cong (\rMod \mathcal{A})/\mathcal{T}_I$ is locally finitely presented precisely when the left annihilators $\Fix(\eta )\leqdef\lann_\mathcal{A}(1_{\what{b}}-\eta)$ of these endomorphisms generate $\calI$ (Theorem~\ref{thm.finitepresentability-on-ideal}). This implies that whenever $\what{\mathcal{A}}$ has pseudocokernels, $\mathcal{X}_{\calI}$ is locally finitely presented if, and only if, the ideal $\calI$ is generated by pseudocokernels of morphisms $1_{\what{b}}-\eta$ as indicated. This fact is crucial in Section~8 to relate the Cuadra--Simson Problem to the (Stable and Nonstable) Telescope Conjecture in arbitrary compactly generated triangulated categories. We show that within such a triangulated category $\mathcal{D}$ a t-structure $\mathbf{t}$ with definable co-aisle (e.g., any smashing semiorthogonal decomposition) is compactly generated if, and only if, the Gabriel quotient $(\rMod \mathcal{D}^c)/\mathcal{T}_{I_\mathbf{t}}$ is locally finitely presented.

\medskip\smallskip \noindent
{\bf Acknowledgment.} The Authors wish to express their gratitude to Dolors Herbera and Pavel P\v r\'\i hoda for suggesting the use of Dubrovin--Puninski rings and for several useful discussions about their properties. The Authors are also grateful to Juan Cuadra for drawing their attention to the particular case of the problem concerning comodules. Moreover, they are indebted to Rosie Laking,  Frederik Marks, and Jorge Vit\'oria; and to Manolo Cort\'es, for sharing \cite{LMV}, and a first draft of \cite{CI}, respectively.

\section{Preliminaries}%

All through the paper $\mathcal{A}$ will be a fixed small preadditive category and $\mathcal{G}$ will denote an abelian category. All subcategories will be strictly full. For any preadditive category $\mathcal{C}$ and any object $C\in\Obj(\mathcal{C})$, we will denote by $\mathcal{C}(X,-)\colon\mathcal{C}\to\Ab$ and $\mathcal{C}(-,X)\colon\mathcal{C}^\textup{op}\to\Ab$ the associated Hom functors. 

\subsection{Grothendieck categories}%

Recall that $\mathcal{G}$ is $\mathrm{AB}\mathchar`-3$ (resp., $\mathrm{AB}\mathchar`-3^\ast$) when it is cocomplete (resp., complete), is $\mathrm{AB}\mathchar`-4$ (resp., $\mathrm{AB}\mathchar`-4^\ast$) when coproducts (resp., products) are exact and $\mathcal{G}$ is $\mathrm{AB}\mathchar`-5$ when it is $\mathrm{AB}\mathchar`-3$ and direct limits are exact. If $\mathcal{G}$ is $\mathrm{AB}\mathchar`-5$ and has a set of generators (equivalently, a generator), then $\mathcal{G}$ is called a \emph{Grothendieck category}. When $\mathcal{G}$ is $\mathrm{AB}\mathchar`-3$ and $\mathcal{S}$ is any class of objects, we shall denote by $\Sum_\mathcal{G}(\mathcal{S})$, $\Gen_\mathcal{G}(\mathcal{S})$ and $\Pres_\mathcal{G}(\mathcal{S})$ the subcategories of $\mathcal{G}$ that consist of the objects isomorphic to (set-indexed) coproducts, epimorphic images of coproducts and cokernels of morphisms between coproducts of objects of $\mathcal{S}$, respectively. When $\mathcal{G}$ is clear from the context, we will omit it as a subscript (e.g., $\Gen(\mathcal{S})$ instead of $\Gen_\mathcal{G}(\mathcal{S})$). 

When $\mathcal{C}$ is any preadditive category and $\mathcal{S}\subseteq\Obj(\mathcal{C})$ is any class of objects, we will put $\mathcal{S}^\bot\leqdef\{C\in\Obj(\mathcal{C})\mid\mathcal{C}(S,C)=0, \ \hbox{for all $S\in\mathcal{S}$}\}$. When, in addition, $\mathcal{C}=\mathcal{G}$ is abelian, we will put
\begin{gather*}
	\mathcal{S}^{\bot_1}\leqdef\{X\in\Obj(\mathcal{G})\mid\Ext_\mathcal{G}^1(S,X)=0,
		\ \hbox{for all $S\in\mathcal{S}$}\} \cr
	\noalign{\hbox{and}}
	\mathcal{S}^{\bot_{0,1}}\leqdef\mathcal{S}^\bot\cap\mathcal{S}^{\bot_1}.
\end{gather*}
By duality, we get obvious subcategories ${}^{\perp}\mathcal{S}$, ${}^{\perp_1}\mathcal{S}$ and ${}^{\perp_{0,1}}\mathcal{S}$. An \emph{orthogonal pair} in $\mathcal{C}$ is a pair $(\mathcal{X},\mathcal{Y})$ of subcategories such that $\mathcal{Y}=\mathcal{X}^\perp$ and $\mathcal{X}={}^\perp\mathcal{Y}$.

When $\mathcal{G}$ is a Grothendieck category, an object $X$ of $\mathcal{G}$ is called \emph{finitely presented} (resp., \emph{finitely generated}) when $\mathcal{G}(X,-)\colon\mathcal{G}\to\Ab$ preserves direct limits (resp., direct unions). We denote by $\fp(\mathcal{G})$ (resp., $\fg(\mathcal{G})$) the subcategory of finitely presented (resp., finitely generated) objects, which is skeletally small (see \cite[Proposition~9.2]{CI-Cr-Sa}). The category $\mathcal{G}$ is \emph{locally finitely presented} (resp., \emph{locally finitely generated}) when it has a set of finitely presented (resp., finitely generated) generators, something which is equivalent to say that each of object of $\mathcal{G}$ is a direct limit of finitely presented (resp., finitely generated) objects. When $\mathcal{G}$ is locally finitely presented and $\fp(\mathcal{G})$ is an abelian subcategory, equivalently closed under kernels, we say that $\mathcal{G}$ is \emph{locally coherent}.

The proof of the following result is left to the reader.

\begin{lem}
Let $L:\mathcal{A}\rightleftarrows\mathcal{B}:\varGamma$ be an adjunction, where $\A$ and $\cal B$ are Abelian categories. Then, the following assertions hold true:%
\label{l:CategoryTheory}%
\begin{enumerate}[(i)]
\item If $\varGamma$ is (right) exact, then $L$ preserves projective objects. The converse holds whenever $\mathcal{A}$ has enough projectives. 
\item If $\mathcal{A}$ and $\mathcal{B}$ are both Grothendieck, with $\cal A$  locally finitely presented, then $\varGamma$ commutes with direct limits if, and only if, $L$ sends objects in $\fp(\A)$ to objects in $\fp(\cal B)$.
\end{enumerate}
\end{lem}

\subsection{Torsion pairs and localizations}%

A \emph{torsion pair} $\mathbf{t}=(\mathcal{T},\mathcal{F})$ in a Grothendieck category $\mathcal{G}$ is just an orthogonal pair. In that case $\mathcal{T}$ is closed under coproducts, extensions and quotients and, dually, $\mathcal{F}$ is closed under subobjects, extensions and products. $\mathcal{T}$ is called the \emph{torsion class} and $\mathcal{F}$ the \emph{torsionfree class} and $\mathbf{t}$ is called a \emph{hereditary torsion pair} when $\mathcal{T}$ is closed under taking subobjects. Any subcategory $\mathcal{T}$ appearing as first component of such a pair is called a \emph{hereditary torsion class} in $\mathcal{G}$. Note that $\mathcal{T}$ is called \emph{localizing subcategory} in \cite{G}.

A \emph{torsion torsionfree class} (TTF class in the sequel) in $\mathcal{G}$ is a subcategory which is both a torsion and a torsionfree class, equivalently a hereditary torsion class closed under taking products. In that case, putting $\mathcal{C}={}^{\perp}\mathcal{T}$ and $\mathcal{F}\leqdef\mathcal{T}^\perp$, we get what is called a \emph{TTF triple} $(\mathcal{C},\mathcal{T},\mathcal{F})$ in $\mathcal{G}$. 

Given a Grothendieck category $\mathcal{G}$ and a hereditary torsion class $\mathcal T$ in $\mathcal{G}$, there is a new Grothendieck category $\mathcal{G}/\mathcal T$, called the \emph{(Gabriel) quotient  of $\mathcal{G}$ by $\mathcal T$}, and a \emph{quotient or localization functor} $q\colon \mathcal{G}\to\mathcal{G}/\mathcal{T}$ satisfying the following two properties (see \cite{G}):
\begin{enumerate}[(i)]
\item $q$ is exact and vanishes on $\mathcal T$;
\item $q$ is universal with respect to property~(i). That is, if $F\colon\mathcal{G}\to\mathcal{A}$ is an exact functor to any abelian category $\mathcal A$ which vanishes on $\mathcal T$, then $F$ factors through $q$ in a unique way (up to natural isomorphism).
\end{enumerate}
The functor $q$ has a fully faithful right adjoint $j\colon\mathcal{G}/\mathcal T\to\mathcal{G}$, called the \emph{section functor}, whose essential image is $\mathcal{G}=\mathcal{T}^{\perp_{0,1}}$, that is called the \emph{Giraud subcategory} associated to $\mathcal{T}$. The counit of the adjoint pair $(q,j)$ is then an isomorphism and we will denote by $\mu\colon 1_{\mathcal{G}}\Rightarrow j\circ q$ its unit. Note that $\ker(\mu_M)$ and $\coker(\mu_M)$ are in $\mathcal{T}$, for each object $M$ of $\mathcal{G}$.

\subsection{t-Structures in triangulated categories with coproducts}%
\label{subs.triangulcats}%
A \emph{triangulated category} is a triple $(\mathcal{D},\Delta,[1])$, where $\mathcal{D}$ is an additive category, $[1]\leqdef(-)[1]\colon\mathcal{D}\to\mathcal{D}$ is an autoequivalence, called the \emph{suspension functor} and $\Delta$ is class of sequences of morphisms $X\mathrel{\smash[t]{\buildrel f\over\to}} Y\mathrel{\smash[t]{\buildrel g\over\to}} Z\mathrel{\smash[t]{\buildrel h\over\to}} X[1]$, called \emph{(distinguished) triangles}, that satisfy certain axioms (we refer the reader to \cite{N} for the details). We put $[n]=[1]^n$, for all $n\in\mathbb{Z}$. The axioms guarantee that any morphism $f$ appears in such a triangle, in which case the object $Z$ is uniquely determined by $f$, up to (non-unique) isomorphism. We put $Z=\Cone(f)$ and call it the \emph{cone of $f$}.

An example of triangulated category that will appear later is  the \emph{homotopy category} $\mathcal{K}(\mathcal{B})$ of an additive category $\mathcal{B}$. Its objects are the (cochain) complexes \[
	B^\chbull\colon \cdots\longrightarrow
		B^{n}\buildrel d^{n}\over\longrightarrow
		B^n\buildrel d^{n+1}\over\longrightarrow B^{n+1}\longrightarrow\cdots
\]
of objects of $\mathcal{B}$ and the morphisms are the homotopy classes of cochain maps. In this case $[1]\colon\mathcal{K}(\mathcal{B})\to\mathcal{K}(\mathcal{B})$ is the functor that moves a complex one unit to the left and the distinguished triangles are those isomorphic in $\mathcal{K}(\mathcal{B})$ to pointwise split short exact sequences in the category of cochain complexes (see \cite[Section~10.1]{Wei}) for the details. When $\mathcal{B}$ is an abelian category, the object $H^n(B^\chbull)=\ker(d^n)/\im(d^{n-1})$ is called the \emph{$n$-th cohomology object of $B^\chbull$} and the assignment $B^\chbull\mapsto H^n(B^\chbull)$ is the definition on objects of an additive functor $H^n\colon\mathcal{K}(\mathcal{B})\to\mathcal{B}$, called the \emph{$n$-th cohomology functor}.

A \emph{torsion pair} in a triangulated category $(\mathcal{D},\Delta,[1])$ is an orthogonal pair $\mathbf{t}=(\mathcal{X},\mathcal{Y})$ of subcategories such that each $D\in\Obj(\mathcal{D})$ fits into a triangle
\begin{equation}
	X\longrightarrow D\longrightarrow Y\longrightarrow X[1]
	\tag{$\triangle$}\label{eq:triangle}
\end{equation}
where $X\in\mathcal{X}$ and $Y\in \mathcal{Y}$. Such a torsion pair is called a \emph{t-structure} when in addition $\mathcal{X}[1]\subseteq\mathcal{X}$ (equivalently, $\mathcal{Y}[-1]\subseteq\mathcal{Y}$), and is called a \emph{semiorthogonal decomposition} or a \emph{stable t-structure} when $\mathcal{X}=\mathcal{X}[1]$ (equivalently, $\mathcal{Y}=\mathcal{Y}[1]$). When $\mathbf{t}$ is a t-structure, in particular when it is a semiorthogonal decomposition, the objects $X$ and $Y$ of the triangle \eqref{eq:triangle} depend functorially on $D$ and give rise to functors $\tau_\mathbf{t}^{\le0}\colon\mathcal{D}\to\mathcal{X}$ and $\tau_\mathbf{t}^{>0}\colon\mathcal{D}\to\mathcal{Y}$, which are right and left adjoint to the respective inclusion functors. They are called the \emph{left} and \emph{right truncation functors} with respect to $\mathbf{t}$, and the triangle \eqref{eq:triangle}, that becomes
\[
	\tau_\mathbf{t}^{\le0}D\longrightarrow D\longrightarrow
		\tau_\mathbf{t}^{>0}D\longrightarrow (\tau_\mathbf{t}^{\leq 0}D)[1],
\]
is called the \emph{truncation triangle} with respect to $\mathbf{t}$. Note that our definition of t-structure differs from the one in \cite{BBD}, but the assignment $(\mathcal{X},\mathcal{Y})\mapsto(\mathcal{X},\mathcal{Y}[1])$ gives a bijection between our t-structures and the ones in \cite{BBD}. We shall extend to an arbitrary torsion pair  $\mathbf{t}=(\mathcal{X},\mathcal{Y})$ the following terminology commonly used for t-structures: we will call $\mathcal{X}$ the \emph{aisle} and $\mathcal{Y}$ the \emph{co-aisle} of $\mathbf{t}$. 

Let $\mathcal{D}$ have coproducts. A \emph{compact object} in $\mathcal{D}$ is an object $X$ such that the functor $\mathcal{D}(X,-)\colon\mathcal{D}\to\Ab$ preserves coproducts. The corresponding subcategory is denoted by $\mathcal{D}^c$. A torsion pair in $\mathcal{D}$ is called \emph{compactly generated} when its co-aisle is of the form $\mathcal{C}^\perp$, for some set $\mathcal{C}\subset\mathcal{D}^c$. We say that $\mathcal{D}$ is a \emph{compactly generated triangulated category} when the trivial torsion pair $(\mathcal{D},0)$ is compactly generated. A torsion pair  is called \emph{smashing} if its co-aisle is closed under coproducts. 

Suppose that $\mathcal{D}$ is a compactly generated triangulated category in this paragraph. A \emph{definable subcategory} of $\mathcal{D}$ (see \cite{Krause4,Laking}) is a full subcategory $\mathcal{Z}$ of $\mathcal{D}$ for which there is a set of morphisms $\mathcal{S}\subseteq\Mor(\mathcal{D}^c)$ such that, given $Z\in\mathcal{D}$, we have that $Z\in\mathcal{Z}$ if, and only if, $\mathcal{D}(s,Z)\colon\mathcal{D}(C',Z)\to\mathcal{D}(C,Z)$ is the zero morphism, for all $s\colon C\to C'$ in $\mathcal{S}$. It is clear that any definable subcategory is closed under coproducts. 
Moreover, any compactly generated torsion pair
$\mathbf{t}=(\mathcal{X},\mathcal{Y})$ has a definable co-aisle, in
particular it is smashing, because if
$\mathcal{C}\subseteq\mathcal{D}^c$ is such that
$\mathcal{C}^\perp=\mathcal{Y}$ then one can choose
$\mathcal{S}:=\{1_C\mid C\in\mathcal{C}\}$.

\section{Modules over small preadditive categories}%

\subsection{The ($2$-)category of preadditive categories}   \label{subs_preadd}
A category $\A$ is said to be preadditive if it is enriched in the category $\Ab$ of abelian groups. Together with additive functors and natural transformations, preadditive categories form a $2$-category. Whenever our $\A$ is (skeletally) small, and $\B$ is any preadditive category, we denote by $[\A,\B]$ be the obvious (preadditive) category of additive functors and natural transformations. 

Let $\A$, $\B$, and $\C$ be preadditive categories, and let $\A\times \B$ be the usual product (which is still preadditive). A functor $F\colon \A\times \B\to \C$ is called {\em biadditive} if: 
\begin{itemize}
    \item $F(-,b)\colon \A\to \C$ is additive, for all $b\in \Obj(\B)$;
    \item $F(a,-)\colon \B\to \C$ is additive, for all $a\in \Obj(\A)$.
\end{itemize}
Assuming $\A$ and $\B$ small, we denote by $\text{Bi}[\A\times \B, \C]$ the category of biadditive functors and natural transformations. There are two standard equivalences:
\begin{equation}\label{pretenshomad_eq}
[\B,[\A,\C]]\cong\text{Bi}[\A\times \B, \C]\cong [\A,[\B,\C]],
\end{equation}
reflecting the fact that a biadditive functor $F\colon \A\times \B\to \C$ is equivalent both to the additive functor $b\mapsto [F(-,b)\colon \A\to \C]$, and also to $a\mapsto [F(a,-)\colon \B\to \C]$. 

The \emph{tensor product} $\mathcal{A}\otimes\mathcal{B}$ of $\A$ and $\B$ (not necessarily small) is defined as:
\begin{itemize}
\item $\Obj(\mathcal{A}\otimes\mathcal{B})\leqdef\Obj(\mathcal{A})\times\Obj(\mathcal{B})$;
\item $(\mathcal{A}\otimes\mathcal{B})((a,b),(a',b'))\leqdef\mathcal{A}(a,a')\otimes_\Z\mathcal{B}(b,b')$ (the usual tensor product of abelian groups),
for all $(a,b),(a',b')\in\Obj(\mathcal{A}\otimes\mathcal{B})$;
\item $(f'\otimes g')\circ (f\otimes g)\leqdef(f'\circ f)\otimes (g'\circ g)$, if both $f'\circ f$ and $g'\circ g$ make sense (in $\A$ and in $\B$, respectively); now extend this rule bilinearly.
\end{itemize}
The obvious functor $T\colon\A\times\B\to \A\otimes\B$ is the universal biadditive functor out of $\A\times \B$. For $\A$ and $\B$ small, this means that $T$ induces the following equivalence:
\[
-\circ T\colon [\A\otimes \B, \C]\overset \cong\longrightarrow \text{Bi}[\A\times \B, \C].
\]
Observe that, by \eqref{pretenshomad_eq}, this just means that $-\otimes \B$ and $\A\otimes-$ are left adjoint to $[\B,-]$ and $[\A,-]$, respectively.

\smallskip
We will denote by $\what{\mathcal{A}}$ the \emph{additive closure} of $\mathcal{A}$. This is the universal additive category associated with $\mathcal{A}$, and it can be described as follows:
\begin{itemize}
    \item $\Obj(\what{\A})\leqdef\{\what{a}=(a_1,\dots,a_n)\mid a_i\in \Obj(\A),\ \forall i=1,\dots,n\}$;
    \item $\what{\A}(\what{a},\what{b})\leqdef\{(\alpha_{ij})\mid  \alpha_{ij}\in\A(a_j,b_i)\}$, for all $\what a,\,\what b\in \Obj(\what{\A})$.
\end{itemize}
So morphisms in $\what{\A}$ are represented by matrices, and can be composed with the usual matrix multiplication. There is a canonical, fully faithful additive functor $\A\to \what{\A}$, identifying any  $a\in \Obj(\A)$ with $(a)\in \Obj(\what{\A})$, and morphisms in $\A$ with $1\times 1$ matrices. Observe that, under this identification, the coproduct in $\widehat{\A}$ of a family $a_1,\dots,a_n\in \Obj(\A)$, is simply $\what{a}\leqdef(a_1,\dots,a_n)$, i.e., $\coprod_{i=1}^na_i=\what{a}$.

To describe the additive closure of $\A^\op$, observe that $\Obj(\what{\A^\op})=\Obj(\what{\A})$, but 
\[
\what{\A^\op}(\,\what{b},\what{a}\,)\leqdef\{(\alpha_{ij})^t\mid (\alpha_{ij})\in \what{\A}(\what{a},\what{b})\},
\]
where $(\alpha_{ij})^t\leqdef(\alpha_{ji})$ denotes the transpose matrix. 
We refer to \cite[Section~1]{PSV} for a deeper discussion of additive closures and their properties.

\subsection{Modules over small preadditive categories}
Let $\A$ be a small preadditive category. The categories of {\em left} and, respectively, {\em right $\mathcal{A}$-modules} are defined as the following additive functor categories:
\[
\mathcal{A}\lMod\leqdef[\A,\Ab]\quad\text{and}\quad\rMod\mathcal{A}\leqdef [\A^\op,\Ab].%
\] 
In what follows we will concentrate primarily on right modules, without this being restrictive, as shown by the following equivalence: $\mathcal{A}\lMod\cong\rMod\mathcal{A}^\textup{op}$. 

As is usual in functor categories, all limits and colimits are computed point-wise in the target. As a consequence, the category of modules inherits several of the nice properties of $\Ab$. In particular, $\rMod\A$ is an abelian category, it is complete and cocomplete, and both direct limits and products are exact. As we will see below, it  also has a (very nice) family of generators, so  we can summarize all these properties by saying that $\rMod\A$ is an $\mathrm{AB}\mathchar`-4^\ast$ Grothendieck category.

\begin{lem}[Additive Yoneda Lemma]
Let $\A$ be a small preadditive category, and consider the following functor, called the {\em Yoneda Embedding}:
\[
\mathbf{y}=\mathbf{y}_\mathcal{A}\colon\mathcal{A}\to\rMod\mathcal{A}\quad \text{such that}\quad a\mapsto\mathbf{y}(a)=\mathcal{A}(-,a).
\]
Then, $\mathbf{y}$ is additive and fully faithful. Moreover, there is a natural isomorphism $(\rMod\A)(\mathbf{y}(a),M)\cong M(a)$, for all $M\in \rMod\A$, and all $a\in \Obj(\A)$.
\end{lem}
Let $\mathbf{y}(a)\reqdef H_a$ be the \emph{representable right $\mathcal{A}$-module} associated to $a$, for all $a\in\Obj(\mathcal{A})$. Dually, $H'_a\leqdef\mathcal{A}(a,-)=\mathbf{y}_{\mathcal{A}^\textup{op}}(a)$, for all $a\in\Obj(\mathcal{A})$. Observe that, as a direct consequence of the Additive Yoneda Lemma, $\{H_a\mid a\in\Obj(\mathcal{A})\}$ is a set of finitely generated projective generators of $\rMod\mathcal{A}$. Therefore, the (finitely generated) projective right $\mathcal{A}$-modules are precisely the direct summands of (finite) coproducts of representable $\mathcal{A}$-modules. 
We shall denote by $\Proj(\mathcal{A})$ (resp., $\proj(\mathcal{A})$) the subcategory of $\rMod\mathcal{A}$ consisting of the (finitely generated) projective $\mathcal{A}$-modules.

Observe that, for any $M\in \rMod\A$, since $M$ is an additive functor, there is a unique $\what{M}\in \rMod{\what{\A}}$ (up to a unique isomorphism) such that $\what{M}_{\restriction\A^\op}=M$, and a similar argument applies to morphisms in $\rMod\A$. In this way, it is possible to show that there is a canonical equivalence of categories $\rMod\A\cong \rMod\what{\A}$ or, in other words, that $\A$ and $\what{\A}$ are Morita equivalent. This allows us to identify $\rMod\A$ and $\rMod\what{\A}$ (resp., $\A\lMod$ and $\what{\A}\lMod$). In particular, given $\what{a}=\coprod_{j=1}^na_j\in \Obj(\what{\A})$, we will refer to the finitely generated projective modules the form $H_{\what{a}}\leqdef\coprod_{j=1}^nH_{a_j}\in \rMod\A$ (resp., $H'_{\what{a}}\leqdef\coprod_{j=1}^nH'_{a_j}\in \A\lMod$) as representable.

With these conventions in mind, a finitely generated object in $\rMod\A$, i.e.,  a \emph{finitely generated right $\mathcal{A}$-module}, is precisely an epimorphic image of a representable $H_{\what{a}}$, for some  $\what{a}\in\Obj(\what{\mathcal{A}})$. Similarly, a finitely presented object in $\rMod\mathcal{A}$, i.e., a \emph{finitely presented right $\mathcal{A}$-module}, can always be written in the form $\coker(H_{\alpha})$, for some  $\alpha=(\alpha_{ij})_{i,j}\colon\what{a}\to\what{b}$ in $\what{\mathcal{A}}$. We put $\rmod\mathcal{A}\leqdef\fp(\rMod\mathcal{A})$. Note that $\rMod\mathcal{A}$ is locally finitely presented.

\subsection{Gabriel--Popescu theorem}%

In this subsection, we just remind the reader of the refinement of the classical Gabriel--Popescu Theorem using the vision of Mitchell of small preadditive categories as rings with several objects (see \cite{Mitc}). A proof may be found in \cite[Theorem~1.1]{P2} (see also \cite[Theorem~1.2]{Lowen}).

\begin{thm}[Gabriel--Popescu--Mitchell]
\label{teor.Gabriel-Popescu-Mitchell}
Let $\mathcal{G}$ be a Grothendieck category generated by a small subcategory $\mathcal{A}$, and consider the following functor, called the {\em Restricted Yoneda Embedding} (relative to $\A$):
\[
\mathbf{y}\colon\mathcal{G}\longrightarrow\rMod\mathcal{A},\quad\text{such that}\quad Y\longmapsto\mathcal{G}(-,Y)_{\lvert\mathcal{A}}, \ \ \forall Y\in \Obj(\cal G).
\]
Then, the following assertions hold true.
\begin{enumerate}[(i)]
\item $\mathbf{y}$ is fully faithful, and it has an exact left adjoint $q\colon\rMod\mathcal{A}\to\mathcal{G}$. 
\item $\mathcal{T}\leqdef\ker(q)$ is a hereditary torsion class in $\rMod\mathcal{A}$, and $q$ induces an equivalence of categories $(\rMod\mathcal{A})/\mathcal{T}\cong\mathcal{G}$.
\end{enumerate}
\end{thm}

Later on we will need a more concrete description of the $\mathcal{A}$-modules in $\mathcal{T}$; this is taken care of in the following result.

\begin{prop}
\label{lem.Mitchell}%
In the setting of Theorem~\ref{teor.Gabriel-Popescu-Mitchell}, suppose that the objects of $\mathcal{A}$ are all finitely generated in $\cal G$. The following are equivalent for  $T\in \rMod \A$:
\begin{enumerate}[(a)]
\item $T\in\mathcal{T}$;
\item there is an epimorphism $f\colon Y\twoheadrightarrow Y'$ in $\mathcal{G}$ inducing an exact sequence
\[
	\mathbf{y}(Y)\buildrel\mathbf{y}(f)\over\longrightarrow
		\mathbf{y}(Y')\longrightarrow T\longrightarrow 0\qquad\text{in $\rMod\mathcal{A}$;}
\]
\item as in assertion~\textup{(b)}, but with $Y,\,Y'\in\Sum_\mathcal{G}(\mathcal{A})$.
\end{enumerate}
\end{prop}
\begin{proof}
Since $\mathbf{y}$ is fully faithful it follows that the counit $q\circ\mathbf{y}\Rightarrow 1_\mathcal{G}$ is a natural isomorphism (apply the dual of \cite[Proposition~7.5]{HS}). Moreover, the fact that all objects of $\mathcal{A}$ are finitely generated in $\mathcal{G}$ implies that $\mathbf{y}$ preserves coproducts. 

\smallskip \noindent``$\textup{(a)}\Leftrightarrow\textup{(c)}$'' If $H_A=\mathcal{A}(-,A)$ denotes the representable right $\mathcal{A}$-module associated to $A\in\Obj(\mathcal{A})$, one has $q(H_A)=A$ and, moreover, $\mathbf{y}$ and $q$ restrict to an equivalence $\Sum_\mathcal{G}(\mathcal{A})\rightleftarrows\mathrm{Free}\mathchar`-\mathcal{A}$, where $\mathrm{Free}\mathchar`-\mathcal{A}$ is the subcategory of $\rMod\mathcal{A}$ with objects the free $\mathcal{A}$-modules, i.e., the coproducts of representable $\mathcal{A}$-modules. Consider now any free presentation of $T$ in $\rMod\mathcal{A}$,
\[
	\coprod_{j\in J}H_{A_j}\buildrel\varphi\over\longrightarrow
		\coprod_{i\in I}H_{B_i}\buildrel\pi\over\longrightarrow T\longrightarrow 0,
\]
where $A_j,B_i\in\mathcal{A}$ for all $i\in I$ and $j\in J$. By the Additive Yoneda Lemma, $\varphi$ may be identified with a column-finite matrix $(H_{\alpha_{ij}}\colon H_{A_j}\to H_{B_i})_{i\in I,j\in J}$, where $\alpha_{ij}\in\mathcal{A}(A_j,B_i)=\mathcal{G}(A_j,B_i)$, for all $i\in I$ and $j\in J$. Consider the obvious morphism $\alpha\colon\coprod_{j\in J}A_j\to\coprod_{i\in I}B_i$ in $\Sum_\mathcal{G}(\mathcal{A})$ given as $\alpha =(\alpha_{ij})$: by the equivalence mentioned above, we deduce that $\mathbf{y}(\alpha )\cong\varphi$. Using the exactness of $q$, we get the following chain of double implications:
\[	T\in\mathcal{T} \iff q(T)=0 
	\iff \hbox{$q(\varphi)\cong\alpha$ is an epimorphism in $\mathcal{G}$}.
\]

 \noindent``$\textup{(c)}\Rightarrow\textup{(b)}$'' Clear.

\smallskip \noindent``$\textup{(b)}\Rightarrow\textup{(a)}$'' Let $f\colon Y\twoheadrightarrow Y'$ be an epimorphism in $\mathcal{G}$ such that $T$ is isomorphic to the cokernel of $\mathbf{y}(f)\colon\mathbf{y}(Y)\to\mathbf{y}(Y')$. The exactness of $q$, and the fact that $q\circ\mathbf{y}\cong 1_\mathcal{G}$, readily imply that $q(T)=0$. 
\end{proof}

\subsection{Modules and ideals associated to morphisms}%
\label{ss:modules-morphisms}%

Let $\mathcal{C}$ be a (not necessarily small) preadditive category. An \emph{ideal of $\mathcal{C}$} is a collection $(\calI(c,c'))_{c,c'\in\Obj(\mathcal{C})}$ such that $\calI(c,c')$ is a subgroup of $\mathcal{C}(c,c')$, for all $c,c'\in\Obj(\mathcal{C})$, and such that, for each composition $c''\mathrel{\smash[t]{\buildrel\alpha\over\to}}c\mathrel{\smash[t]{\buildrel\beta\over\to}}c'\mathrel{\smash[t]{\buildrel\gamma\over\to}}c'''$ of morphisms in $\mathcal{C}$, if $\beta\in\calI(c,c')$ then $\gamma\circ\beta\circ\alpha\in\calI(c'',c''')$.
\begin{exmpl}\label{ideal_gen_morph_ex}
For each morphism $\alpha\colon a\to b$ in $\mathcal{A}$, we denote by $\mathcal{A}\alpha\mathcal{A}$ the smallest ideal of $\mathcal{A}$ that contains $\alpha$. Recall (see \cite[Definition~2.2]{PSV}) that, for each $c$ and $d\in\Obj(\mathcal{A})$, the subgroup $(\mathcal{A}\alpha\mathcal{A})(c,d)\leq \A(c,d)$ consists of those morphisms $c\to d$ in $\A$ that can be expressed as a finite sum of compositions of the form
\(
	c\mathrel{\smash[t]{\buildrel\beta\over\to}}
	a\mathrel{\smash[t]{\buildrel\alpha\over\to}}
	b\mathrel{\smash[t]{\buildrel\gamma\over\to}} d
\) (for suitable morphisms $\beta$ and $\gamma$ in $\A$).
\end{exmpl}

For any ideal $I$ in $\C$, there is an associated {\em quotient category} $\mathcal{C}/\calI$, such that:
\begin{itemize}
    \item $\Obj(\mathcal{C}/\calI)\leqdef\Obj(\C)$;
    \item $(\mathcal{C}/\calI)(c,d)\leqdef\mathcal{C}(c,d)/\calI (c,d)$, for all $c,d\in\Obj(\mathcal{C})$;
\end{itemize}
with the obvious composition induced from that of $\mathcal{C}$. Clearly $\mathcal{C}/\calI$ is preadditive, and there is an additive projection functor $p\colon\mathcal{C}\to\mathcal{C}/\calI$, which is the identity on objects and where $p_{c,d}\colon\mathcal{C}(c,d)\to(\mathcal{C}/\calI)(c,d)=\mathcal{C}(c,d)/\calI (c,d)$ is the obvious group projection, for all $c,d\in \Obj(\C)$. 

\begin{exmpl}\label{stab_cat_mod_proj_ex}
Let $\mathcal{A}$ be a small preadditive category, and denote by  $\mathcal{P}$ the ideal of $\rMod\mathcal{A}$ (resp., of $\rmod\mathcal{A}$) of all those morphisms that factor through some object in $\Proj(\mathcal{A})$ (resp., in $\proj(\mathcal{A})$). The corresponding quotient category  $\underline{\mathrm{Mod}}\mathchar`-\mathcal{A}\leqdef(\rMod\mathcal{A})/\mathcal{P}$ (resp., $\rstmod\mathcal{A}\leqdef(\rmod\mathcal{A})/\mathcal{P}$) is the so-called \emph{stable category (modulo projectives)}.
\end{exmpl}

Observe that, given $\what{a}$ in $\what{\mathcal{A}}$, a submodule (or, equivalently, a subfunctor) $M$ of the representable left $\mathcal{A}$-module $H'_{\what{a}}\,\colon\mathcal{A}\to\Ab$
is simply a family of subgroups $\{M(b)\leq \what{\A}(\what{a},b)\mid b\in \Obj(\A)\}$ such that $\alpha\circ m\in M(d)$, for all $m\in M(c)$, $\alpha\in \A(c,d)$, and $c,d\in \Obj(\A)$. Submodules of representable right $\A$-modules can be described similarly. 

In particular, for any morphism $\alpha\colon\what{a}\to\what{b}$ in $\what{\mathcal{A}}$, the \emph{submodule} $\mathcal{A}\alpha\leq\smash{H'_{\what{a}}}$ (resp., $\alpha\mathcal{A}\leq\smash{H_{\what{b}}}$) \emph{generated by $\alpha$} can be described as follows (for all $c\in \Obj(\A)$):
\begin{itemize}
\item $(\mathcal{A}\alpha)(c)\leqdef\{\beta\circ\alpha\in\what{\mathcal{A}}(\what{a},c) \mid \beta\in\what{\mathcal{A}}(\what{b},c)\}$;
\item $(\alpha\mathcal{A})(c)\leqdef\{\alpha\circ\beta\in\what{\mathcal{A}}(c,\what{b})\mid \beta\in\what{\mathcal{A}}(c,\what{a})\}$.
\end{itemize}
Moreover, let us define the \emph{left (resp., right) annihilator of $\alpha$} as the submodule $\lann_\mathcal{A}(\alpha)\leq H'_{\what{b}}$ (resp., $\rann_\mathcal{A}(\alpha)\leq H_{\what{a}}$) such that (for all $c\in\Obj(\mathcal{A})$):
\begin{itemize}
\item $[\lann_\mathcal{A}(\alpha)](c)\leqdef\{\gamma\in\what{\mathcal{A}}(\what{b},c)\mid \gamma\circ\alpha =0\}$;
\item $[\rann_\mathcal{A}(\alpha)](c)\leqdef\{\beta\in\what{\mathcal{A}}(c,\what{a})\mid \alpha\circ\beta =0\}$.
\end{itemize}

\begin{defn}
\label{def.ideal associated to a submodule}%
Given  a submodule $M\le H'_{\what{a}}$\,, for some  $\what{a}$ in $\what{\mathcal{A}}$, the \emph{ideal of $\mathcal{A}$ generated by $M$} is the smallest ideal $M\mathcal{A}$ of $\A$ such that $M(b)\leq (M\mathcal{A})(\what{a},b)$, for all $b\in \Obj(\A)$. Equivalently, for each $c,d\in \Obj(\A)$, we have:
\[
(M\mathcal{A})(c,d)\leqdef\Bigl\{\sum_{k=1}^n\mu_k\circ\beta_k\Bigm\vert \beta_k\in\what{\mathcal{A}}(c,\what{a}),\, \mu_k\in M(d)\subseteq\what{\mathcal{A}}(\what{a},d)\Bigr\}.
\]
The ideal $\mathcal{A}N$, generated by a submodule $N\leq H_{\what{a}}$\,, is defined dually.
\end{defn}


 We can now describe the ideals of $\mathcal{A}$ generated by the submodules  generated by a morphism $\alpha\colon \what{a}\to \what{b}$, and by its left and right annihilators. The proof of the following corollary is immediate from the definitions.

\begin{cor}
\label{cor.idealgenerated-annihilator}%
Let $\alpha\leqdef (\alpha_{ij})_{i,j}\in \what{\A}(\what{a},\what{b})$, represented by a $m\times n$ matrix with coefficients $\alpha_{ij}\in \A(a_j,b_i)$, for all $i=1,\dots, m$ and $j=1,\dots, n$. Then, the following assertions hold:
\begin{enumerate}[(i)]
\item The ideal generated by $\mathcal{A}\alpha\leq H'_{\what{a}}$ (resp., $\alpha\mathcal{A}\leq H_{\what{b}}$) is (see Example~\ref{ideal_gen_morph_ex}):
\[
	\mathcal{A}\alpha\mathcal{A}\leqdef\sum_{\substack{1\leq i\leq m \cr 1\leq j\leq n}}
		\mathcal{A}\alpha_{ij}\mathcal{A}.
\]
\item The ideal $\lann_\mathcal{A}(\alpha)\mathcal{A}$ of $\mathcal{A}$ generated by $\lann_\mathcal{A}(\alpha)\leq H'_{\what{b}}$ can be described,  for all $c,d\in \Obj(\A)$, as follows:
\[
[\lann_\mathcal{A}(\alpha)\mathcal{A}](c,d)\leqdef\Bigl\{\sum_{k=1}^n\mu_k\circ\beta_k\Bigm\vert \beta_k\in\what{\mathcal{A}}(c,\what{b}), \text{ }\mu_k\in[\lann_\A(\alpha)](d)\Bigr\}.
\]
\end{enumerate}
\end{cor}
%
%


\begin{rem}\label{rem_sub_gen_id}
Suppose that $\alpha =1_{\what{a}}$\,, then clearly $\mathcal{A}1_{\what{a}}=H'_{\what{a}}$ and $1_{\what{a}}\mathcal{A}=H_{\what{a}}$\,. In particular, $\mathcal{A}1_{\what{a}}\,\mathcal{A}=\sum_{j=1}^n\mathcal{A}1_{a_j}\mathcal{A}$ is the ideal generated by $\{1_{a_j}\mid j=1,\ldots,n\}$. 
\end{rem}

We conclude this subsection with a description of  $\proj(\A^\op)\leqdef\proj(\A\lMod)$ and $\proj(\A)\leqdef\proj(\rMod\A)$, the categories of finitely generated projective left and, resp., right $\A$-modules. As a consequence of the proposition below one deduces immediately that $\proj(\A)$ (resp., $\proj(\A^\op)$) is equivalent to the idempotent completion of $\what{\A}$ (resp., $\what{\A^\op}$).

\begin{prop}
\label{prop.idempotent-versus-fgprojectives}%
Let $e\in\what{\mathcal{A}}(\what{a},\what{a})$ be an idempotent endomorphism of some $\what{a}\in\Obj(\what{\A})$. Then, the following assertions hold true:
\begin{enumerate}[(i)]
\item $e\mathcal{A}$ (resp., $\mathcal{A}e$) is a direct summand of $H_{\what{a}}$ (resp., $H'_{\what{a}}$\,), and so it is a finitely generated projective right (resp., left) $\mathcal{A}$-module.
\item Any finitely generated projective right (resp., left) $\mathcal{A}$-module is of the form $e\mathcal{A}$ (resp., $\mathcal{A}e$), for a suitable $\what{a}\in\Obj(\what{\mathcal{A}})$, and $e=e^2\in\what{\mathcal{A}}(\what{a},\what{a})$.
\end{enumerate}
\end{prop}
\begin{proof}
Since the representable right (resp., left) $\mathcal{A}$-modules form a set of finitely generated projective generators, the finitely generated projective right (resp., left) $\mathcal{A}$-modules are the direct summands of finite coproducts of representable right (resp., left) $\mathcal{A}$-modules. That is, they are the direct summands of right (resp., left) $\mathcal{A}$-modules of the form $H_{\what{a}}^{\vphantom\prime}$ (resp., $H'_{\what{a}}$\,), for some $\what{a}\in\Obj(\what{\mathcal{A}})$. 
%
\end{proof}

\subsection{Categories of bimodules}%

The category of {\em $\B$-$\A$-bimodules} is defined as: 
\[
\mathcal{B}\biMod\mathcal{A}\leqdef \text{Bi}[\A^\op\times \B, \Ab],
\]
but we will generally identify it with one of the following categories:
\begin{prop}\label{prop.interpretations of bimodules}%
The following categories are all equivalent to $\mathcal{B}\biMod\mathcal{A}$:
\begin{enumerate}[(a)]
\item $\rMod(\mathcal{B}^\textup{op}\otimes\mathcal{A})$;
\item $(\mathcal{A}^\textup{op}\otimes\mathcal{B})\lMod$;
\item $[\mathcal{B},\rMod\mathcal{A}]$;
\item $[\mathcal{A}^\textup{op},\mathcal{B}\lMod]$.
\end{enumerate}
\end{prop}
\begin{proof}
The equivalence between $\mathcal{B}\biMod\mathcal{A}$ and the categories in (b), (c), and (d), follows by the discussion in Section~\ref{subs_preadd}. For the equivalence between (b) and (a), it is enough to observe that $(\mathcal{A}^\textup{op}\otimes \mathcal{B})^\op\cong\mathcal{B}^\textup{op}\otimes\mathcal{A}$.
\end{proof}

\begin{rem}\label{rem.modules as bimodules}
Observe that $\A\otimes \Z\cong \A \cong \Z\otimes \A$, for any small preadditive category $\mathcal{A}$ (identifying $\Z$ with a one-object preadditive category). These trivial equivalences can be used to embed the theory of modules into that of bimodules. More precisely, $\rMod\mathcal{A}\cong\mathbb{Z}\biMod\mathcal{A}$, and $\mathcal{A}\lMod\cong\mathcal{A}\biMod\mathbb{Z}$.
\end{rem}

\begin{exmpl}\mbox{\label{rem.regular bimodule and ideal}}%
Let $\A$ be a small preadditive category.
\begin{enumerate}
\item Consider the \emph{regular $\mathcal{A}\mathchar`-\mathcal{A}$-bimodule} $\mathcal{A}_\textup{reg}\leqdef\mathcal{A}(-,-)\colon\mathcal{A}^\textup{op}\otimes\mathcal{A}\to\Ab$. The equivalences $(\mathcal{A}^\textup{op}\otimes\mathcal{A})\lMod\cong[\mathcal{A},\rMod\mathcal{A}]\cong[\mathcal{A}^\textup{op},\mathcal{A}\lMod]$ of  Proposition~\ref{prop.interpretations of bimodules}, allow us to identify $\mathcal{A}_\textup{reg}$ with the Yoneda Embedding $\mathbf{y}_{\A}\colon\mathcal{A}\to\rMod\mathcal{A}$ or, equivalently, with  $\mathbf{y}_{\A^\op}\colon\mathcal{A}^\textup{op}\to\mathcal{A}\lMod$.
\item An $\mathcal{A}\mathchar`-\mathcal{A}$-sub-bimodule of $\mathcal{A}_\textup{reg}$ is just an ideal of $\mathcal{A}$ (see Subsection~\ref{ss:modules-morphisms}). So, by part (1), we can identify the ideals of $\mathcal{A}$ with the subfunctors of either of the Yoneda Embeddings.
\end{enumerate}
\end{exmpl}

\begin{rem}
The equivalence between $\rMod\mathcal{A}$ and $\rMod\what{\mathcal{A}}$ induces a bijection between the ideals of $\mathcal{A}$ and $\what{\mathcal{A}}$, that is, for any ideal $\calI$ in $\A$, there is a unique ideal $\what{\calI}$ in $\what{\mathcal{A}}$, such that $\calI(a,b)=\what{\calI}(a,b)$ for all $a,b\in\Obj(\A)$. Similarly, for each $\what{a}=\coprod_{j=1}^na_j$, $\what{b}=\coprod_{i=1}^mb_i\in \Obj(\what{\A})$, we have:
\[
\what{\calI}(\what{a},\what{b})= \{(\alpha_{ij})_{i,j}\in \what{\A}(\what a, \what b)\mid  \alpha_{ij}\in \calI(a_j,b_i),\ \forall i,j\}.\] As these two ideals determine each other so easily, we will generally identify them, and write expressions like $I(\what a,\what b)$ (instead of $\what I(\what a,\what b)$).
\end{rem}

Among the ideals of $\A$, there is a particular family that will play an important role in the paper. These are the ideals that arise as traces (see \cite[Sections~1 and~2]{PSV}). Observe that, even if $\mathcal{S}\subseteq \rMod\mathcal{A}$ is a proper class, for any given $M\in\rMod\A$, the collection $\{\Im(f)\leq M\mid f\colon S\to M,\ S\in\mathcal{S}\}$ is just a set,  since $\rMod\mathcal{A}$ is well-powered. 

\begin{defn}[see {\cite[Subsection~1.3]{PSV}}]
\label{def. trace}%
For any class $\mathcal{S}\subseteq \rMod\mathcal{A}$, and any given right $\A$-module $M\in \rMod\mathcal{A}$, we define:
\begin{itemize}
    \item the {\em trace of $\mathcal{S}$ in $M$}, $\tr_\mathcal{S}(M)\leqdef\sum\{\Im(f)\mid f\colon S\to M,\ S\in\mathcal{S}\}\leq M$;
    \item the {\em trace ideal $\tr_{\mathcal{S}}(\A)$ of $\mathcal{S}$ in $\A$}, which is the ideal of $\A$ such that
    \[
    [\tr_{\mathcal{S}}(\A)](a,b)\leqdef[\tr_{\mathcal{S}}(H_b)](a),\quad\forall a,b\in \Obj(\A).
    \]
\end{itemize}
\end{defn}
Observe that, by Example~\ref{rem.regular bimodule and ideal}\,(2), $\tr_{\mathcal{S}}(\A)$ can also be identified with a subfunctor of the Yoneda Embedding $\bf y_\A\colon \A\to \rMod\A$ such that $a\mapsto \tr_{\cal S}(H_a)$.

\smallskip
A crucial property of trace ideals is the following (see \cite[Proposition~2.7]{PSV}).

\begin{prop}
\label{prop.traces-versus-idempotents}%
Let $\mathcal{A}$ be a small preadditive category.
\begin{enumerate}[(i)]
\item $\tr_{e\mathcal{A}}(\A_{\rm reg})=\A e \A$, for all $\what{a}\in\Obj(\what{\mathcal{A}})$ and $e=e^2\in\what{\mathcal{A}}(\what{a},\what{a})$.
\item An ideal $\calI$ of $\mathcal{A}$ is the trace of a set of finitely generated projective right (resp., left) $\mathcal{A}$-modules if, and only if, it is generated by idempotent endomorphisms in $\what{\mathcal{A}}$. That is, if and only if, there is a family $(\what{a}_\lambda)_{\lambda\in\Lambda}$ of objects in $\what{\mathcal{A}}$ and, for each $\lambda$, an idempotent  $e_\lambda\in\what{\mathcal{A}}(\what{a}_\lambda,\what{a}_\lambda)$, such that \[\calI=\sum_{\Lambda}\mathcal{A}e_\lambda\mathcal{A}.\]
\end{enumerate}
\end{prop}
\begin{proof}
Assertion~(i) for the right version is an easy generalization of \cite[Proposition~2.6]{PSV}, and the left version follows by duality. Assertion~(ii) is then a direct consequence of assertion~(i) and Proposition~\ref{prop.idempotent-versus-fgprojectives}.
\end{proof}

\subsection{Radicals of modules and of small preadditive categories}%

In this subsection we recall some facts about the radical of a small preadditive category. The concept seems to have been considered first by Kelly \cite{Kelly}. Here we follow closely \cite[Section~2]{Krause2}, where the development is done for additive categories. 

\smallskip
In what follows, given $M\in \rMod\A$, we will say that a proper submodule $N\lneq M$ is {\em maximal} if $N\leq N'\lneq M$ implies that $N=N'$.

\begin{defn}
Given $M\in \rMod\A$, the \emph{radical of $M$} is, by definition:
\[
\Rad(M)\leqdef \begin{cases}
    \bigcap\{N\mid N\lneq M\text{ maximal}\}&\text{if this family is $\neq\emptyset$;}\\
    M&\text{otherwise.}
\end{cases}
\]
Then, the {\em radical $\rad_{\mathcal{A}}$ of $\A$} is defined as the ideal of $\A$ such that:
\[
\rad_{\mathcal{A}}(b,a)\leqdef [\Rad(H_a)](b), \quad\forall a,b\in \Obj(\A).
\]
\end{defn}

\begin{lem}
\label{lem.proper submodule}%
The following are equivalent for $M\leq H_a$ ($a\in \Obj(\A)$):
\begin{enumerate}[(a)]
\item $M\lneq H_a$, i.e., $M$ is a proper submodule of $H_a$;
\item $M(a)\lneq H_a(a)=\mathcal{A}(a,a)$, i.e., $M(a)$ is a proper subgroup (or, equivalently, a proper right ideal) of $\A(a,a)$;
\item $1_a\not\in M(a)$, where $1_a\colon a\to a$ is the identity morphism.
\end{enumerate}
Moreover, if $M\lneq H_a$ is maximal, then $M(a)$ is a maximal right ideal of $\A(a,a)$.
\end{lem}
\begin{proof}
The implications ``(c)$\Rightarrow$(b)$\Rightarrow$(a)'' are trivial, while  ``(a)$\Rightarrow$(c)'' follows from the first part of Remark~\ref{rem_sub_gen_id}.

\smallskip
Finally, if $M\lneq H_a$ is maximal, but $M(a)$ is not a maximal ideal in $\mathcal{A}(a,a)$, there has to be some $\alpha\in\mathcal{A}(a,a)$ such that $M(a)\lneq M(a)+\alpha\mathcal{A}(a,a)\lneq\mathcal{A}(a,a)$. But then $M\lneq M+\alpha\mathcal{A}\lneq H_a$, contradicting the maximality of $M$ in $H_a$.
\end{proof}


The next result is standard for additive categories (see \cite[Corollary~2.10]{Krause2}, although Krause's definition of the radical is different). A consequence of the next theorem is then that our definition of the radical is equivalent to the one in \cite{Krause2}, at least for additive categories. In what follows, we will denote by $J(R)$ the Jacobson radical of $R$, for any (associative and unitary) ring $R$.

\begin{thm}\label{thm.radical morphisms}%
Let $\mathcal{A}$ be a small preadditive category, and let $\alpha \colon b\to a$ be a morphism in $\mathcal{A}$. Then, the following assertions are equivalent:
\begin{enumerate}[(a)]
\item $\alpha\in\rad_\mathcal{A}(b,a)$;
\item $\alpha\circ\beta\in J(\mathcal{A}(a,a))$, for all $\beta\in\mathcal{A}(a,b)$;
\item $1_a-\alpha\circ\beta$ is an automorphism of $a$, for all $\beta\in\mathcal{A}(a,b)$;
\item $\alpha\in\rad_{\mathcal{A}^\textup{op}}(a,b)$;
\item $\beta\circ\alpha\in J(\mathcal{A}(b,b))$, for all $\beta\in\mathcal{A}(a,b)$;
\item $1_b-\beta\circ\alpha$ is an automorphism of $b$, for all $\beta\in\mathcal{A}(a,b)$.
\end{enumerate}
\end{thm}
\begin{proof}
Observe that it is enough to prove the equivalences in the case of additive categories. In fact, if $\mathcal{A}$ is preadditive, then $\A(a,a)= \what{\A}(a,a)$, so that $J(\A(a,a))=J( \what{\A}(a,a))$, and $\rad_{\what{\mathcal{A}}}(b,a)=\rad_{\mathcal{A}\vphantom{\what{\mathcal{A}}}}(b,a)$, for all $a,b\in\Obj(\A)$. In particular, the equivalence of the above conditions for $\A$ is a consequence of the same result for $\what{\A}$. So, in the rest of the proof, we assume that $\mathcal{A}$ is additive. 

By \cite[Proposition~2.9, Corollary~2.10]{Krause2}, we just need to prove that our definition of $\rad_\mathcal{A}$ gives  $\rad_\mathcal{A}(a,a)=J(\mathcal{A}(a,a))$, for all $a\in\Obj(\mathcal{A})$, and this is a direct consequence of Lemma~\ref{lem.proper submodule}. Indeed,
let $\mathbf{m}$ be a maximal right ideal of $\mathcal{A}(a,a)$, so that $\mathbf{m}\mathcal{A}=\sum_{\mu\in\mathbf{m}}\mu\mathcal{A}$ is a submodule of $H_a$ such that $(\mathbf{m}\mathcal{A})(a)=\mathbf{m}$, in particular,  $\mathbf{m}\mathcal{A}\lneq H_a$. Since $H_a$ is finitely generated, there exists a maximal submodule $M\leq H_a$ such that $\mathbf{m}\mathcal{A}\leq M$. Evaluating at $a$, we see that $\mathbf{m}=M(a)$ and, therefore:
\[
\rad_\mathcal{A}(a,a)=\bigcap_{M\in\Max(H_a)}M(a)=
	\bigcap_{\mathbf{m}\in\Max(\mathcal{A}(a,a))}\mathbf{m}=
	J(\mathcal{A}(a,a)),
\]
where we have used $\Max(-)$ to denote both the set of maximal submodules of $H_a$, and also the set of maximal right ideals of $\mathcal{A}(a,a)$. 
\end{proof}

Bearing in mind that nonzero finitely generated $\mathcal{A}$-modules have a maximal submodule, we also get the following proposition (see the second paragraph of Section~\ref{sect.idempotent-TTF-recollements} for the definition of $M\calI$, where $\calI$ is an ideal of $\mathcal{A}$, and $M\in \rMod\A$).

\begin{prop}\label{prop.NAK}%
Let $\mathcal{A}$ be a small preadditive category and let $M\in\rMod\A$. Then, the following assertions hold true:
\begin{enumerate}[(i)]
\item $M\rad_\mathcal{A}\subseteq\Rad(M)$.
\item (Nakayama Lemma over preadditive categories) If $M$ is finitely generated, and $M\rad_\mathcal{A}=M$, then $M=0$.
\item If $N\leq M$ is a submodule such that $M/N$ is finitely generated, then $N=M$ if, and only if, $M=N+M\rad_\mathcal{A}$.
\end{enumerate}
\end{prop}

\section{Purity and flatness in Grothendieck categories}%
\label{sec.purity-flatness}

Even if, for our purposes, it would be enough to use the well-known purity theory in locally finitely presented Grothendieck categories (see \cite{CB}), we extend that theory to arbitrary Grothendieck categories.

\begin{defn}
Let $\mathcal{G}$ be a Grothendieck category.%
\label{def.purity-flatness}
\begin{itemize}
\item A morphism $p\colon Y\to Z$ in $\cal G$ is called a \emph{pure epimorphism} when it is a direct limit of retractions.
\item A short exact sequence $0\to X\to Y\mathrel{\smash[t]{\buildrel p\over\to}} Z\to 0$ in $\cal G$ is called \emph{pure exact} when $p$ is a pure epimorphism or, equivalently, when it is a direct limit of split short exact sequences.
\item An object $F$ in $\cal G$ is called \emph{(categorically) flat} when every epimorphism of the form $X\twoheadrightarrow F$ in $\mathcal{G}$ is pure. 
\end{itemize}
\end{defn}

\label{exmpl.coproduct-directlimit}%
\begin{exmpl}
Let $((Y_i)_{i\in I},(u_{ij}\colon Y_i\to Y_j)_{i\leq j})$ be a direct system in $\mathcal{G}$. Then, the canonical epimorphism $\pi\colon\coprod_IY_i\to\varinjlim_I Y_i$ is pure. 
\end{exmpl}

First of all, let us show that our definition of purity agrees with the usual one when the ambient category is locally finitely presented (see also \cite[Section~3]{CB}).

\begin{prop}
Let $\mathcal{G}$ be a locally finitely presented Grothendieck category. A morphism $p\colon Y\to Z$ in $\cal G$ is a pure epimorphism if, and only if, the induced map $\mathcal{G}(X,p)\colon\mathcal{G}(X,Y)\to\mathcal{G}(X,Z)$ is an epimorphism in $\Ab$, for any $X\in\fp(\mathcal{G})$.%
\label{p:PureEpimorphisms}%
\end{prop}
\begin{proof}
\smallskip \noindent``$\Leftarrow$'' Let $p\colon Y\to Z$ be a morphism such that $\mathcal{G}(X,p)$ is an epimorphism, for all $X\in\fp(\mathcal{G})$. Write $Z=\varinjlim_I Z_i$, where $(Z_i)_{i\in I}\subseteq\fp(\mathcal{G})$ is a suitable direct system. Since, for each $j\in I$, the canonical map $\lambda_j\colon Z_j\to Z$ factors through $p$, the pullback of $p$ along $\lambda_j$ yields a retraction $p_j\colon Y_j\twoheadrightarrow Z_j$. Hence, we obtain a direct system of retractions $(p_i\colon Y_i\to Z_i)_{i\in I}$ whose direct limit is $p$.

\smallskip \noindent``$\Rightarrow$'' Let $(p_i\colon Y_i\to Z_i)_{i\in I}$ be a direct system of retractions whose direct limit is $p$. In particular, $p\circ\lambda_j^Y=\lambda_j^Z\circ p_j^{\vphantom Z}$, for all $j\in I$, where the $\lambda_j^?$ are the canonical maps to the respective direct limits. Observe that, given $f\colon X\to Z$, with $X\in\fp(\mathcal{G})$, there is a factorization $f=\lambda_j^Z\circ f_j\colon X\to Z_j\to Z$, for some $j\in I$.
To conclude, choose a section $s_j\colon Z_j\to Y_j$ of $p_j$; we obtain that:
\[
	p\circ\lambda_j^Y\circ s_j^{\vphantom Y}\circ f_j^{\vphantom Y}=
		\lambda_j^Z\circ p_j^{\vphantom Z}\circ s_j^{\vphantom Z}\circ f_j^{\vphantom Z}=
		\lambda_j^Z\circ 1_{Z_j}\circ f_j^{\vphantom Z}=
		\lambda_j^Z\circ f_j^{\vphantom Z}=f,
\]
showing that $f$ factors through $p$. 
\end{proof}

\begin{lem}
Let $\cal G$ be a Grothendieck category, and consider two morphisms $f\colon X\to Y$ and $g\colon Y\to Z$ in $\cal G$. Then, if $g\circ f$ is a pure epimorphism,  $g$ is a pure epimorphism.%
\label{l:PureEpimorphisms}%
\end{lem}
\begin{proof}
Write $g\circ f=\varinjlim_I h_i$, where $(h_i\colon X_i\to Z_i)_{i\in I}$ is a suitable direct system of retractions, and let $\lambda_j^X\colon X_j^{\vphantom X}\to X$, and $\lambda_j^Z\colon Z_j^{\vphantom Z}\to Z$, be the canonical maps to the respective direct limit, so that $g\circ f\circ\lambda_j^X=\lambda_j^Z\circ h_j^{\vphantom Z}$, for all $j\in I$. Taking the pullback of $g$ along $\lambda_j^Z$, we obtain the following commutative diagram:
\[
\xymatrix@R=12pt@C=40pt{%
	X_j\ar@/_1.5pc/[dddr]|-{f\circ\lambda_j^X}\ar@{.>}[dr]|-{\exists!f_j}\ar@/^1.5pc/[drr]|-{h_j} \\
    & Y_j\ar@{}[ddr]|{\text{P.B.}} \ar[dd]_-{\lambda_j^Y}\ar[r]|-{g_j} &	Z_j \ar[dd]^-{\lambda_j^Z} \\
    \\
	  & Y \ar[r]|-{g} & Z.
}
\]
The uniqueness of all the arrows involved clearly gives direct systems of morphisms $(f_i\colon X_i\to Y_i)_{i\in I}$ and $(g_i\colon Y_i\to Z_i)_{i\in I}$ such that $g_j\circ f_j=h_j$, for all $j\in I$.
Conclude by observing that each $g_j$ is a retraction (as $g_j\circ f_j=h_j$, where $h_j$ is a retraction) and that, clearly, $g\cong\varinjlim_I g_i$.
\end{proof}

Our next result deals with the preservation of purity and flatness by suitable adjoint functors.

\begin{lem}
Let $L:\mathcal{G}\rightleftarrows\mathcal{G}':\varGamma$ be an adjunction between Grothendieck categories. Then, the following assertions hold true.%
\label{l:Flatness}%
\begin{enumerate}[(i)]
\item $L$ preserves pure epimorphisms.
\item Suppose that $L$ is fully faithful, then:
\begin{enumerate}[(1)]
\item If $\varGamma$ is (right) exact, then $L$ preserves flat objects.
\item If $\varGamma$ preserves direct limits, then $L$ reflects flat objects.
\end{enumerate}
\end{enumerate}
\end{lem}
\begin{proof}
Assertion~(i) is clear since $L$ preserves retractions and direct limits.

\smallskip \noindent(ii.1) Let $F\in\Obj(\mathcal{G})$ be a flat object, and let $p\colon X'\twoheadrightarrow L(F)$ be an epimorphism. Observe that the unit $\lambda\colon 1_\mathcal{G}\Rightarrow \varGamma\circ L$ is a natural isomorphism (see \cite[Proposition~7.5]{HS}), and denote by $\epsilon$ the counit $L\circ \varGamma\Rightarrow 1_{\mathcal{G}'}$. By hypothesis, the map $\varGamma(p)\colon \varGamma(X')\to (\varGamma\circ L)(F)\cong F$ is an epimorphism, whence a pure epimorphism since $F$ is flat. It then follows from assertion~(i) that $(L\circ \varGamma)(p)\colon (L\circ \varGamma)(X')\to (L\circ \varGamma\circ L)(F)\cong L(F)$ is a pure epimorphism. But this last map decomposes as $(L\circ \varGamma)(X')\mathrel{\smash[t]{\buildrel\epsilon_{\smash{X'}}\over\to}} X'\mathrel{\smash[t]{\buildrel p\over\to}} L(F)$. It now follows from Lemma~\ref{l:PureEpimorphisms} that $p$ is a pure epimorphism. Therefore, $L(F)$ is flat. 

\smallskip \noindent(ii.2) Given $F\in\Obj(\mathcal{G})$ such that $L(F)$ is flat in $\mathcal{G}'$, let $p\colon X\twoheadrightarrow F$ be an epimorphism in $\mathcal{G}$. Then, $L(p)\colon L(X)\twoheadrightarrow L(F)$ is a pure epimorphism in $\mathcal{G}'$. Fix a direct system of retractions $(q_i\colon M_i\to N_i)_{i\in I}$ in $\cal G'$ such that $L(p)\cong\varinjlim_I q_i$. Therefore, $p\cong (\Gamma\circ L)(p)\cong\varinjlim_I\Gamma (q_i)$, as  $\Gamma$ preserves direct limits, so that $p$ is a direct limit of retractions, i.e., a pure epimorphism. Thus, $F$ is flat in $\mathcal{G}$.
\end{proof}

\subsection{Tensor product of modules and bimodules}%

In this subsection we gather some facts about tensor products of (bi)modules over small preadditive categories that seem to be folklore, but for which the authors did not find a suitable reference. The first definition of such a tensor product seems to have appeared in \cite{OR}, but for most of the properties we will use \cite[Section~6]{K}, with the inconvenience that Keller works in a dg context, that we should ``un-dgrade''.

Let us start with a simple observation: given a small preadditive category $\A$, the categories of left and right $\A$-modules can be equivalently described as:
\[
\A\lMod\cong [\rMod\A,\Ab]_{!}\quad(\text{resp., } \rMod\A\cong [\A\lMod,\Ab]_{!}),
\]
where $[-,\Ab]_!$ denotes the category of cocontinuous additive functors to $\Ab$, that is, the functors that commute with all colimits (or, equivalently, with all coproducts and cokernels), see \cite[Formula (4.56), p.68]{Kelly_enriched}. In fact, the Yoneda Embedding identifies $\A$ with the full subcategory of representable left $\A$-modules, and this allows us to see a given $M\in \rMod\A$ as the additive functor from representable left $\A$-modules to $\Ab$, such that $M(H_a')\leqdef M(a)$, for all $a\in \Obj(\A)$. Moreover, since any $N\in \A\lMod$, can be written as a colimit of representables, 
there is a unique way to extend $M$ to a cocontinuos functor $M^\to\colon \A\lMod\to \Ab$ (see \cite{Kelly_enriched}, and Lemma~\ref{lem.checking natural isomorphism on generators} and its proof). Similarly, there is a unique extension of $N\in \A\lMod$ to a cocontinuos functor $N^\to\colon \rMod\A\to \Ab$. 

It is possible to define $M\otimes_\A N\leqdef M^\to(N)\cong N^\to (M)$, for all $M\in \rMod \A$, and $N\in \A\lMod$. In other words, there are natural isomorphisms of functors $(M\otimes_\A -)\leqdef M^\to$, and $(-\otimes_\A N)\leqdef N^\to$. With these definitions, the usual tensor-hom adjunctions become (see \cite[Theorem~1]{St2}):
\[
\Ab(M\otimes_\A N,G)\cong (\rMod \A)(M,\Ab(N,G))\cong (\A\lMod)(N,\Ab(M,G)),
\]
where, for all $G\in \Ab$, $\Ab(N,G)\leqdef \Ab(-,G)\circ N^\op\colon \A^\op\to \Ab^\op\to \Ab$, and $\Ab(M,G)\leqdef \Ab(-,G)^\op\circ M^\op\colon \A\to \Ab^\op\to \Ab$.

It is also possible to define $M\otimes_\A N$  more explicitly, as a cokernel (in $\Ab$) of 
\[
\Psi\colon\bigoplus_{a,a'\in \Obj(\A)}M({a'})\otimes_\Z \A(a,a')\otimes_\Z N(a)\longrightarrow \bigoplus_{a\in \Obj(\A)}M(a)\otimes_\Z N(a),
\]
where $\Psi(y\otimes f\otimes x)\leqdef y\otimes N(f)(x)-M(f)(y)\otimes x$, for all $y\in M(a')$, $f\in \A(a,a')$, and $x\in N(a)$. To see that this definition is, in fact, equivalent to the previous one, it is enough to check that also in this way we obtain a functor which is cocontinuous in each variable, and such that there are natural isomorphisms $M\otimes_\A H'_a\cong M(a)$, and $H_a\otimes_\A N\cong N(a)$, for all $a\in \Obj(\A)$ (see \cite{K}). 

\smallskip
Now, given three small preadditive categories $\mathcal{A}$, $\mathcal{B}$ and $\mathcal{C}$, a $\mathcal{B}\mathchar`-\mathcal{A}$-bimodule $X$, and an $\mathcal{A}\mathchar`-\mathcal{C}$-bimodule $Y$, we define the \emph{tensor product of $X$ and $Y$ (over $\mathcal{A}$)} as the $\mathcal{B}\mathchar`-\mathcal{C}$-bimodule $X\otimes_\mathcal{A}Y\colon\mathcal{C}^\textup{op}\otimes\mathcal{B}\to\Ab$ defined on objects as follows: $[X\otimes_\mathcal{A}Y](c,b)\leqdef X(-,c)\otimes_\A Y(b,-)$, 
for all  $b\in\Obj(\mathcal{B})$, $c\in\Obj(\mathcal{C})$. 

As one may expect, we have the following (see \cite[Section~6]{K}).

\begin{prop}\label{prop.tensor product functor}%
Let $\mathcal{A}$, $\mathcal{B}$ and $\mathcal{C}$ be small preadditive categories, and let $X$ be a $\mathcal{B}\mathchar`-\mathcal{A}$-bimodule. Then, the following assertions hold true:
\begin{enumerate}[(i)]
\item The assignment $Y\mapsto X\otimes_\mathcal{A}Y$ can be naturally extended to a coproduct-preserving right exact functor $\mathcal{A}\biMod\mathcal{C}\to\mathcal{B}\biMod\mathcal{C}$;
\item The assignment $Z\mapsto Z\otimes_\mathcal{B}X$ can be naturally extended to a coproduct-preserving right exact functor $\mathcal{C}\biMod\mathcal{B}\to\mathcal{C}\biMod\mathcal{A}$.
\end{enumerate}
\end{prop}

In the following result we collect some useful properties of the tensor product of (bi)modules whose verification is left to the reader.

\begin{cor}\label{cor.properties tensor product}%
Let $\mathcal{A}$, $\mathcal{B}$ and $\mathcal{C}$ be small preadditive categories, let $X$ be a $\mathcal{B}\mathchar`-\mathcal{A}$-bimodule, and let $Y$ be an $\mathcal{A}\mathchar`-\mathcal{C}$-bimodule. The following assertions hold:
\begin{enumerate}[(i)]
\item $X\otimes_\mathcal{A} H'_a\cong X(a,-)\colon\mathcal{B}\to\Ab$, for all $a\in\Obj(\mathcal{A})$.
\item $H_b\otimes_\mathcal{B}X\cong X(-,b)\colon\mathcal{A}^\textup{op}\to\Ab$, for all $b\in\Obj(\mathcal{B})$.
\item $(X\otimes_\mathcal{A}Y)(c,b)\cong X(-,b)\otimes_\mathcal{A}Y(c,-)$, for all $(c,b)\in \Obj(\mathcal{C})\times\Obj(\mathcal{B})$.
\item $-\otimes_\mathcal{A}\mathcal{A}_\textup{reg}\colon\mathcal{B}\biMod\mathcal{A}\to\mathcal{B}\biMod\mathcal{A}$ is naturally isomorphic to the identity endofunctor of $\mathcal{B}\biMod\mathcal{A}$. Similarly, $\mathcal{A}_\textup{reg}\otimes_\mathcal{A}-\cong 1_{\mathcal{A}\biMod\mathcal{B}}$.
\end{enumerate}
\end{cor}

\begin{rem}
\label{rem.morphism by multiplication}%
Given a $\mathcal{B}\mathchar`-\mathcal{A}$-bimodule $X$, and an ideal $\calI$ in $\mathcal{A}$, we denote by $\rho_X$ the following composition:
\[
	\rho_X\colon X\otimes_\mathcal{A}\calI\buildrel 1_X\otimes\iota\over{\relbar\joinrel\longrightarrow} X\otimes_\mathcal{A}\mathcal{A}_\textup{reg}\buildrel\cong\over{\relbar\joinrel\longrightarrow} X,
\]
where $1_X\otimes\iota \leqdef(X\otimes_\mathcal{A}-)(\iota)$ (with $\iota\colon\calI\hookrightarrow\mathcal{A}_\textup{reg}$ the obvious inclusion of $\A$-$\A$-bimodules), and where $X\otimes_\mathcal{A}\mathcal{A}_\textup{reg}\mathrel{\smash[t]{\buildrel\cong\over\to}}X$ is the natural isomorphism given by assertion~(iv) of the above corollary. 
Explicit computations show that $(\rho_X)_{(c,b)}\colon (X\otimes_\mathcal{A}\calI)(c,b)\to X(c,b)$ takes the element $x\otimes\eta$ (with $x\in X(a,b)$ and $\eta\in\calI(c,a)$, for some $a\in\Obj(\mathcal{A})$) to $x\eta\leqdef X(\eta,b)(x)\in X(c,b)$. It is not hard to see that the $\rho_X$'s are the components of a natural transformation $\rho\colon-\otimes_\mathcal{A}\calI\Rightarrow 1_{\mathcal{B}\biMod\mathcal{A}}$ of functors $\mathcal{B}\biMod\mathcal{A}\to\mathcal{B}\biMod\mathcal{A}$. We will say that $\rho$ (or any of its components $\rho_X$) is \emph{induced by multiplication}.
\end{rem}

\begin{prop}\label{prop.flat modules}%
Let $\mathcal{A}$ and $\mathcal{B}$ be small preadditive categories, and let $X$ be a $\mathcal{B}\mathchar`-\mathcal{A}$-bimodule. Then, the following assertions are equivalent:
\begin{enumerate}[(a)]
\item $X\otimes_\mathcal{A}-\colon\mathcal{A}\biMod\mathcal{C}\to\mathcal{B}\biMod\mathcal{C}$ is an exact functor, for any small preadditive category $\mathcal{C}$;
\item $X\otimes_\mathcal{A}-\colon\mathcal{A}\lMod\to\mathcal{B}\lMod$ is an exact functor;
\item $X(-,b)\otimes_\mathcal{A}-\colon\mathcal{A}\lMod\to\Ab$ is an exact functor, for all $b\in\Obj(\mathcal{B})$;
\item For each $\what{a}\in\Obj(\what{\mathcal{A}})$ and each finitely generated submodule $L\leq H'_{\what{a}}$\,, the morphism $1_X\otimes\iota\leqdef (X\otimes_\mathcal{A}-)(\iota)\colon X\otimes_\mathcal{A}L\to X\otimes_\mathcal{A}H'_{\what{a}}$\, is a monomorphism in $\mathcal{B}\lMod$, where $\iota \colon L\hookrightarrow H'_{\what{a}}$\, is the inclusion.
\end{enumerate}
\end{prop}
\begin{proof}[Proof (sketched)]
\smallskip \noindent``$\textup{(a)}\Rightarrow\textup{(b)}$'' Take $\mathcal{C}=\mathbb{Z}$.

\smallskip \noindent``$\textup{(c)}\Leftrightarrow\textup{(b)}\Rightarrow\textup{(a)}$''  Use Corollary~\ref{cor.properties tensor product}~(iii).

\smallskip \noindent``$\textup{(b)}\Rightarrow\textup{(d)}$'' Clear.

\smallskip \noindent``$\textup{(d)}\Rightarrow\textup{(b)}$'' We need to prove that, if $u\colon N\hookrightarrow M$ is an inclusion in $\mathcal{A}\lMod$, then $1_X\otimes u\colon X\otimes_\mathcal{A}N\to X\otimes_\mathcal{A}M$ is a monomorphism. Taking the pullback of $u$ along an epimorphism onto $M$ from a coproduct of representables, one easily reduces the proof to the case when $M$ itself is a coproduct of representable $\mathcal{A}$-modules. But since $N$ is the direct union of its finitely generated submodules, it is easily seen that the inclusion $N\hookrightarrow M=\coprod_{I}H'_{a_i}$ is the direct limit of a direct system of inclusions $\iota_\lambda\colon N_\lambda\hookrightarrow\coprod_{i\in I_\lambda}H'_{a_i}$, where $N_\lambda$ is a finitely generated submodule of $N$ and $I_\lambda\subseteq I$ is a finite subset, for all $\lambda\in\Lambda$. Since the functor $X\otimes_\mathcal{A}-$ preserves direct limits, and each  $1_X\otimes\iota_\lambda \colon X\otimes_\mathcal{A}N_\lambda\to X\otimes_\mathcal{A}\coprod_{I_\lambda}H'_{a_i}$ is a monomorphisms by (d), we conclude that $1_X\otimes u\colon X\otimes_\mathcal{A}N\to X\otimes_\mathcal{A} \coprod_{I}H'_{a_i}$ is a monomorphism, by the $\mathrm{AB}\mathchar`-5$ condition in $\Ab$.
\end{proof}

\begin{defn}\label{def.flat module}
Let $\mathcal{A}$ and $\mathcal{B}$ be small preadditive categories.
\begin{enumerate}[(i)]
\item A $\mathcal{B}\mathchar`-\mathcal{A}$-bimodule $X$ is said to be \emph{flat on the right} when it satisfies the equivalent conditions of Proposition~\ref{prop.flat modules}. 
\item A right $\mathcal{A}$-module $F$ is called \emph{flat} if it is flat on the right when we see it as a $\Z\mathchar`-\mathcal{A}$-bimodule. 
\end{enumerate}
The dual definitions of flat left $\mathcal{B}$-module and flat on the left bimodule are clear. 
\end{defn}

As in module categories over rings, we immediately get the following:

\begin{cor}\label{cor.properties of flat modules}%
The class of flat right (resp., left) $\mathcal{A}$-modules is closed under taking direct summands, coproducts and direct limits in $\rMod\mathcal{A}$ (resp., $\mathcal{A}\lMod$), and it contains all the projective $\mathcal{A}$-modules. 
\end{cor}

\subsection{Purity and flatness in categories of modules}%

We next see that, when $\mathcal{G}=\rMod\mathcal{A}$, purity and flatness can be characterized in the familiar ways. Our proof will be based on the following definition due to Auslander and Bridger (see \cite[Definition~2.5]{AB}): for any given $M\in\rmod\mathcal{A}$, fix a projective presentation 
$
H_{\what b}\to H_{\what a}\to M\to 0,
$
induced by a suitable $\alpha\leqdef (\alpha_{ij})_{j,i}\in \what{\A}(\what{b},\what{a}\,)$. Taking the transpose $\alpha^t\leqdef (\alpha_{ji})_{j,i}\in \what{\A^\op}(\what{a},\what{b})$, we obtain a morphism $H'_{\alpha^t}\colon H'_{\what{a}}\to H'_{\what b}$ in $\mathcal{A}\lMod$, and we define $\Tr(M)\leqdef\coker(H'_{\alpha^t})$. This finitely presented left $\mathcal{A}$-module is uniquely determined, up to isomorphism in the stable category modulo projectives $\mathcal{A}\lstmod$ (see Example~\ref{stab_cat_mod_proj_ex}), and its isomorphism class (in $\mathcal{A}\lstmod$) is independent of the choice of projective presentation for $M$. We will refer to $\Tr(M)$ as \emph{the transpose of $M$}. Observe that, if $P\in \proj\text{-}\A$, then $\Tr(P)\in \A\text{-}\proj$ (i.e., $\Tr(P)=0$ in $\mathcal{A}\lstmod$). In particular, we obtain a well-defined functor $\Tr\colon\rstmod\mathcal{A}\to\mathcal{A}\lstmod$, which is an equivalence of categories with inverse $\Tr\colon\mathcal{A}\lstmod=\rstmod\mathcal{A}^\textup{op}\to \mathcal{A}^\op\lstmod=\rstmod\mathcal{A}$.

\begin{prop}\label{prop.tensor product via transpose}
Let $\A$ be a small preadditive category, let $M\in \rmod\A$, choose a suitable $\alpha\in \what{\A}(\what{b},\what{a}\,)$ for which $M\cong \coker(H_\alpha)$, and take the transpose  $\Tr(M)\leqdef\coker(H'_{\alpha^t})\in \A\lmod$, as above.  Then, the following assertions hold:
\begin{enumerate}[(i)]
\item For each $a\in\Obj(\A)$, there are natural isomorphisms of functors \[
H_{a}\otimes_\mathcal{A}-\cong (\mathcal{A}\lMod)(H'_{a},-),\ \ \text{ and }\ \ -\otimes_\mathcal{A}H'_{a}\cong (\rMod\mathcal{A})(H_a,-).
\]
\item There is an exact sequence of functors $\mathcal{A}\lMod\to\Ab$,
\[
0\longrightarrow (\Tr(M),-) (H'_{\what{b}}\,,-)\buildrel(H'_{\alpha^t},-)\over{\relbar\joinrel\relbar\joinrel\longrightarrow} (H'_{\what{a}}\,,-)\longrightarrow M\otimes_\mathcal{A}-\longrightarrow 0,
\]
where $(X,-)\leqdef(\mathcal{A}\lMod)(X,-)$, for all $X\in \mathcal{A}\lMod$.
\item There is an exact sequence of functors $\rMod\A\to\Ab$,
\[
0\longrightarrow (M,-)'\longrightarrow -\otimes_\A H'_{\what{a}}\buildrel 1_{-}\otimes H'_{\alpha^t}\over{\relbar\joinrel\relbar\joinrel\longrightarrow} -\otimes_\A H'_{\what{b}}\longrightarrow -\otimes_\mathcal{A}\Tr(M)\longrightarrow0, 
\]
where $(M,-)'\leqdef(\rMod\A)(M,-)$.
\end{enumerate}
\end{prop}
\begin{proof}
\smallskip \noindent(i) Let $a\in\Obj(\mathcal{A})$; by Corollary~\ref{cor.properties tensor product}(i) and (ii) there are natural  isomorphisms $H_a\otimes_\mathcal{A}X\cong X(a)$, and $Y\otimes_\mathcal{A}H'_a\cong Y(a)$, for all $Y\in \A\lMod$ and all $X\in \A\lMod$. Similarly, the Additive Yoneda Lemma (for $\A$ and $\A^\op$) gives natural isomorphisms $X(a)\cong (\mathcal{A}\lMod)(H'_a,X)$, and $Y(a)\cong (\rMod\A)(H_a,Y)$. 

\smallskip \noindent(ii) The exactness of the sequence $H'_{\what{a}}\to H'_{\what{b}}\to\Tr(M)\to 0$ implies that $\ker((H'_{\alpha^t},-))\cong (\Tr(M),-)$, while the exactness of $H_{\what{b}}\to H_{\what{a}}\to M\to 0$, together with assertion~(i), shows that $\coker((H'_{\alpha^t},-))\cong M\otimes_\A -$.

\smallskip \noindent(iii) Follows similarly to assertion~(ii).
\end{proof}

Although we have stated the above proposition with $M\in \rmod\A$, it is clear that a similar result also holds for any $N\in \A\lmod=\rmod{\A^\op}$. As a consequence, we obtain the following.

\begin{cor}\label{cor.pure exact sequences modules}
A short exact sequence $0\to \smash[t]{X\buildrel u\over\longrightarrow Y\buildrel v\over\longrightarrow Z}\to 0$ in $\mathcal{A}\lMod$ is pure if, and only if, for any (finitely presented) $M\in\rMod\A$, the following is a short exact sequence in $\Ab$:
\[
	0\longrightarrow M\otimes_\mathcal{A}X\buildrel 1_M\otimes u\over{\relbar\joinrel\longrightarrow}
		M\otimes_\mathcal{A}Y\buildrel 1_M\otimes v\over{\relbar\joinrel\longrightarrow} M\otimes_\mathcal{A}Z\longrightarrow 0.
\]
\end{cor}
\begin{proof}
As in the case of modules over rings, it is enough to consider the case $M\in \rmod\A$. Suppose that the sequence is pure, so that $(\mathcal{A}\lMod)(U,-)$ keeps it exact, for all $U\in\mathcal{A}\lmod$. Consider the exact sequence of functors in Proposition~\ref{prop.tensor product via transpose}(ii). The three left most functors there are exact on our sequence, for each $M\in\rmod\mathcal{A}$. It then follows that $M\otimes_\mathcal{A}-$ keeps the sequence exact, for each $M\in\rmod\mathcal{A}$. 

Conversely, suppose that our sequence is kept exact by $V\otimes_\mathcal{A}-$, for all $V\in\rmod\mathcal{A}$. Using the obvious analog for left $\A$-modules of the previous proposition, the three right most functors in its assertion~(iii) leave the sequence exact, for each $N\in\mathcal{A}\lmod$. It then follows that $(\A\lMod)(N,-)$ keeps the sequence exact, for each $N\in\mathcal{A}\lmod$. 
\end{proof}

To end the subsection, we see that, in categories of modules, the categorical definition of flatness coincides with the one based on tensor products.

\begin{prop}\label{prop.tensor-categorical-flats-the-same}%
Let $\mathcal{A}$ be a small preadditive category. For a right $\mathcal{A}$-module $F$, the following assertions are equivalent:
\begin{enumerate}[(a)]
\item $F$ is flat, i.e.\ $F\otimes_\mathcal{A}-\colon\mathcal{A}\lMod\to\Ab$ is an exact functor;
\item $F$ is categorically flat (see Definition~\ref{def.purity-flatness});
\item $F$ is a direct limit of finite coproducts of representable right $\mathcal{A}$-modules.
\end{enumerate}
\end{prop}
\begin{proof}
Using the version of Corollary~\ref{cor.pure exact sequences modules} for $\mathcal{A}^\textup{op}$, the proof of ``$\textup{(a)}\Leftrightarrow\textup{(b)}$'' is essentially identical to the one for module categories over rings (see \cite[Proposition~I.11.1]{St}). Similarly, ``$\textup{(a)}\Leftrightarrow\textup{(c)}$'' is included in \cite[Theorem~3.2]{OR}.
\end{proof}

\section{Idempotent ideals, TTF triples, and recollements}%
\label{sect.idempotent-TTF-recollements}%

Recall that if $\calI$ and $\calJ$ are two ideals of $\mathcal{A}$, then their product $\calI\calJ$ is an ideal of $\A$ that can be defined as follows: for each $(a,b)\in\Obj(\mathcal{A})\times\Obj(\mathcal{A})$,
\[
(\calI\calJ)(a,b)\leqdef\Bigl\{\sum_{i=1}^n\gamma_i\circ\beta_i\Bigm\vert \beta_i\in\calI(a,c_i),\ \gamma_i\in\calJ(c_i,b), \text{ and }c_i\in\Obj(\mathcal{A})\Bigr\}.
\]
We write $\calI^2\leqdef \calI\calI$ and say that $\calI$ is an \emph{idempotent ideal} of $\mathcal{A}$ when $\calI^2=\calI$.

If $M$ is a right $\mathcal{A}$-module and $\calI$ is an ideal of $\mathcal{A}$, we define the $\mathcal{A}$-submodules $M\calI$ and $\rann_M\calI$ of $M$ such that, for all $a\in \Obj(\mathcal{A})$ (see \cite[Section~2]{PSV}):
\begin{itemize}
\item $(M\calI)(a)\leqdef\sum\{\im(M(\alpha))\mid b\in\Obj(\mathcal{A}), \, \alpha\in\calI(a,b)\}
$.
\item $
	(\rann_M\calI)(a)\leqdef\bigcap\{\ker(M(\alpha))\mid b\in\Obj(\mathcal{A}),\,\alpha\in\calI(b,a)\}$.
\end{itemize}

The following is a generalization of the classical Jans' Correspondence (\cite{Jans}).

\begin{prop}[{\cite[Theorem~4.5]{PSV}}]
\label{prop.TTF triple associated to idemp.ideal}%
Let $\calI$ be an idempotent ideal of $\mathcal{A}$ and consider the following subcategories of $\rMod\mathcal{A}$:
\begin{itemize}
    \item $\mathcal{C}_I\leqdef \{C\in\rMod\mathcal{A}\mid C\calI=C\}$;
    \item $\mathcal{T}_I\leqdef\{T\in\rMod\mathcal{A}\mid T\calI=0\}=\{T\in\rMod\mathcal{A}\mid T(\alpha)=0, \textup{for all $\alpha\in I$} \}$;
    \item $\mathcal{F}_I \leqdef\{F\in\rMod\mathcal{A}\mid \rann_F(\calI)=0\}$.
\end{itemize}
Then, $\mathcal{T}_I$ is a TTF class in $\rMod\mathcal{A}$ and $(\mathcal{C}_I,\mathcal{T}_I,\mathcal{F}_I)$ is the associated TTF triple. Moreover, the assignment $\calI\mapsto (\mathcal{C}_I,\mathcal{T}_I,\mathcal{F}_I)$ yields a bijection between the set of idempotent ideals of $\mathcal{A}$ and the one of TTF triples in $\rMod\mathcal{A}$. 
\end{prop}

In the rest of the section $\calI$ will be a fixed idempotent ideal of the small preadditive category $\mathcal{A}$, $(\mathcal{C}_I,\mathcal{T}_I,\mathcal{F}_I)$ will be the associated TTF triple in $\rMod\mathcal{A}$, define $\mathcal{G}_I\leqdef \mathcal{T}_I{}^{\bot_{0,1}}$, which is the Giraud subcategory associated with the hereditary torsion pair $(\mathcal T_I,\mathcal F_I)$. We also consider the following recollement (we refer to  \cite[Theorem~4.3, Corollary~4.4]{PV14} and \cite{PV} for more details about abelian recollements):
\[
\rMod\mathcal{A}/\calI=\empty\!\xymatrix@C=6em{%
	**[l]\mathcal{T}_I \ar[r]|(.5){i_\ast} &
	\rMod\mathcal{A} \ar[r]|(.5){j^\ast} \ar@/_1.5pc/[l]_(.5){i^\ast} \ar@/^1.5pc/[l]^(0.5){i^!} &
	**[r] \mathcal{G}_I \ar@/_1.5pc/[l]_(.5){j_!} \ar@/^1.5pc/[l]^(.5){j_\ast}
}
\]
where $i_\ast\colon\mathcal{T}_I\hookrightarrow\rMod\mathcal{A}$ and $j_\ast\colon\mathcal{G}_I\hookrightarrow\rMod\mathcal{A}$ are the (fully faithful) inclusions, while $j^\ast\colon\rMod\mathcal{A}\to\mathcal{G}_I\cong(\rMod\mathcal{A})/\mathcal{T}_I$ is  the (exact) Gabriel quotient functor. 

On the other hand, recall that we can recover the original TTF triple from the above recollement via the equality  $(\mathcal{C}_I,\mathcal{T}_I,\mathcal{F}_I)=(\ker(i^{\ast}),\im(i_\ast),\ker(i^{!}))$. Furthermore, $\im(i_\ast)=\mathcal{T}_I=\ker(j^\ast)$, which implies that $j^\ast\circ i_\ast=0$ and hence, by adjunction, also $i^\ast\circ j_!=0$ and $i^!\circ j_\ast=0$. Recall also that $i_\ast$, $j_!$ and $j_\ast$ are fully faithful, which implies that the units of the adjunctions $(i_\ast,i^!)$ and $(j_!,j^\ast)$, and the counits of $(i^\ast,i_\ast)$ and $(j^\ast,j_\ast)$, are all natural isomorphisms. We denote by $\omega\colon i_\ast\circ i^!\Rightarrow 1_{\rMod\mathcal{A}}$ and $\nu \colon j_!\circ j^\ast\Rightarrow1_{\rMod\mathcal{A}}$ the corresponding counits, and by $\gamma \colon 1_{\rMod\mathcal{A}}\Rightarrow i_\ast\circ i^\ast$ and $\mu\colon 1_{\rMod\mathcal{A}}\Rightarrow j_\ast\circ j^\ast$ the respective units. There are two natural exact sequences of functors:
\begin{gather*}
	0\Longrightarrow i_\ast\circ i^!\buildrel\omega\over\Longrightarrow
		1_{\rMod\mathcal{A}}\buildrel\mu\over\Longrightarrow j_\ast\circ j^\ast, \\
	j_!\circ j^\ast\buildrel\nu\over\Longrightarrow 1_{\rMod\mathcal{A}}
		\buildrel\gamma\over\Longrightarrow	i_\ast\circ i^\ast\Longrightarrow 0,
\end{gather*}
where $\ker(\nu_M)$ and $\coker(\mu_M)$ are in $\mathcal{T}_I=\ker(i_\ast)$, for all $M\in\rMod\A$. We will call $\mu_M$ and $\nu_M$ the \emph{localization} and \emph{colocalization morphisms}, respectively. Similarly, the functors $j_\ast\circ j^\ast$ and $j_!\circ j^\ast$ will be called the \emph{localization} and \emph{colocalization functors}, respectively. 

The proposition below contains the basic properties of recollements that will be needed in our treatment.

\begin{prop}\label{prop.relevantfacts-recollement}%
Let $\mathcal{A}$ be a small preadditive category, let $\calI=\calI^2$ be an  ideal in $\mathcal{A}$, and consider the associated TTF triple and recollement, as above. Then:
\begin{enumerate}[(i)]
\item The following statements are equivalent for $M\in \rMod\A$:
\begin{enumerate}[(a)]
\item $M\in\mathcal{C}_I$;
\item $\nu_M\colon(j_!\circ j^\ast)(M)\to M$ is an epimorphism;
\item $M$ is a quotient of some object in $\im(j_!)$.
\end{enumerate}
\item $\im(\nu_M)=M\calI$, for all $M\in\rMod\A$.
\item  $T\otimes_\mathcal{A}\calI=0$, for all $T\in\mathcal{T}_I$ and, more generally, \mbox{$T\otimes_\mathcal{A}C'=0$,} for all   $C'\in\mathcal{C}'_I\leqdef \{X\in\mathcal{A}\lMod\mid\calI X=X\}$.
\item $\im(j_!)={}^{\bot_{0,1}}\mathcal{T}_I$.
\item  $j_!$ induces an equivalence  $\mathcal{G}_{\calI}\mathrel\to\mathcal{X}_I\leqdef \im(j_!)$. In particular, the Grothendieck category $\mathcal{X}_I$ is coreflective in $\rMod\mathcal{A}$.
\end{enumerate}
\end{prop}
\begin{proof}
(i) If $M\in\mathcal{C}_I$, then $\gamma_M=0$ because $i_\ast i^\ast M\in\im(i_\ast)=\mathcal{T}_I$. It follows that $\nu_M$ is an epimorphism, thus proving the implication ``$\textup{(a)}\Rightarrow\textup{(b)}$''.
The implication ``$\textup{(b)}\Rightarrow\textup{(c)}$'' is clear, and the implication ``$\textup{(c)}\Rightarrow\textup{(a)}$'' holds since $\im(j_!)\subseteq\mathcal{C}_I$, and $\mathcal{C}_I$ is closed under taking quotients.

\smallskip \noindent(ii) This is well known (see \cite[Theorem~4.3]{PV14} and its proof).

\smallskip \noindent(iii) By Corollary~\ref{cor.properties tensor product}, for each $T\in\mathcal T_I$ and each $a\in\Obj(\mathcal{A})$, we have that $(T\otimes_\mathcal{A}\calI)(a)=T\otimes_\mathcal{A}\calI(a,-)\in \Ab$, and this group is generated by elements of the form $x\otimes \eta$, for suitable  $b\in \Obj(\A)$, $x\in T(b)$, and $\eta\in I(a,b)$. On the other hand, since $I^2=I$, each $\eta\in I(a,b)$, can be written as  $\eta=\sum_{i=1}^s\beta_i\alpha_i$, with $\alpha_i\in\calI(a,b_i)$ and $\beta_i\in\calI(b_i,b')$, for all $i=1,\ldots,s$. As a consequence:
\[
x\otimes \eta=x\otimes \sum_{i=1}^s\beta_i\alpha_i=\sum_{i=1}^sx\otimes \beta_i\alpha_i=\sum_{i=1}^s[T(\beta_i)](x)\otimes\alpha_i=0,
\]
since $T(\beta_i)=0$, by Proposition~\ref{prop.TTF triple associated to idemp.ideal}. This shows that $T\otimes_\mathcal{A}\calI(a,-)=0$, for all $a\in\Obj(\mathcal{A})$, which implies that $T\otimes_\mathcal{A}\calI=0$.

For the final part, note that $\mathcal{C}'_I$ is the left constituent of the TTF triple in $\mathcal{A}\lMod=\rMod\mathcal{A}^\textup{op}$ associated to the ideal $\calI$ of $\mathcal{A}^\textup{op}$. Hence, we have that $\mathcal{C}'_I=\Gen(\calI(a,-)\mid a\in\Obj(\mathcal{A}))$ (see \cite[Lemma~2.4]{PSV}), and we have seen above that $T\otimes_\mathcal{A}\calI(a,-)=0$, for all $a\in\Obj(\mathcal{A})$ and all $T\in\mathcal{T}_{\calI}$. We then conclude by the right exactness of $T\otimes_\mathcal{A}-\colon\mathcal{A}\lMod\to\Ab$. 

\smallskip \noindent(iv) It is well known (see \cite[Theorem~2.5]{PV14}).

\smallskip \noindent(v) The equivalence of categories follows directly from the full faithfulness of $j_!$. As for the coreflectivity condition, let $\iota\colon\mathcal{X}_I\hookrightarrow\rMod\mathcal{A}$ be the inclusion, and denote by $u\colon\mathcal{G}_I\to\mathcal{X}_I=\im(j_!)$ the equivalence  induced by $j_!$. Since being left or right adjoint to an equivalence just means being its quasi-inverse, and since the adjoint of a composition is the composition of the adjoints in the opposite order, the equality $j_!=\iota\circ u$ implies that $\iota\cong j_!\circ u^{-1}$ is left adjoint to $u\circ j^\ast$.
\end{proof}

Let us fix the notation $\mathcal{X}_I\leqdef \im(j_!)={}^{\bot_{0,1}}\mathcal{T}_I$, for the rest of the section.

\begin{lem}\label{lem.checking natural isomorphism on generators}%
Let $F,\, G\colon\rMod\mathcal{A}\to\mathcal{V}$ be additive functors that preserve all colimits, where $\mathcal{V}$ is a cocomplete abelian category, and let $\mathbf{y}\colon \A\to \rMod\mathcal{A}$ be the Yoneda Embedding. If the two compositions \mbox{$F\circ\mathbf{y},\, G\circ\mathbf{y}\colon\mathcal{A}\to\rMod\mathcal{A}\to\mathcal{V}$} are naturally isomorphic, then $F$ and $G$ are naturally isomorphic.
\end{lem}
\begin{proof}
Given $M\in \rMod\A$, let $\mathbf y/M$ be its category of points, that is:
\begin{itemize}
    \item $\Obj(\mathbf y/M)\leqdef\{(a,x)\mid a\in \Obj(\A),\ x\in M(a)\}$;
    \item $(\mathbf y/M)((a,x),(b,y))\leqdef\{\alpha\in \A(a,b)\mid [M(\alpha)](y)=x\}.$
\end{itemize}
Then, $M\cong \colim_{(a,x)\in \Obj(\mathbf y/M)}H_a$ (see \cite[Proposition~1.3.8]{BorceuxII}). In particular, we have that $F(M)\cong G(M)$, since $F$ and $G$ both commute with colimits, and $F(H_a)\cong G(H_a)$, for all $a\in \Obj(\A)$.
%
\end{proof}

The following is an extension of \cite{Oh} to categories of modules over small preadditive categories.

\begin{prop}\label{prop.colocalization as tensor product}%
Consider the composition $j_!\circ j^\ast\circ {\bf y}\colon\mathcal{A}\to\rMod\mathcal{A}$, where ${\bf y}\colon \A\to \rMod\A$ is the Yoneda Embedding, and let $\Lambda \colon\mathcal{A}^\textup{op}\times\mathcal{A}\to\Ab$ be the associated $\mathcal{A}\mathchar`-\mathcal{A}$-bimodule (see Proposition~\ref{prop.interpretations of bimodules}). The following assertions hold:
\begin{enumerate}[(i)]
\item The counit $\nu \colon j_!\circ j^\ast\Rightarrow 1_{\rMod\mathcal{A}}$ induces pointwise a natural transformation $\nu*{\bf y}\colon j_!\circ j^\ast\circ {\bf y}\Rightarrow {\bf y}$ in $[\A,\rMod\A]$,  that corresponds to a unique  morphism $\widetilde{\nu}\colon \Lambda\to\mathcal{A}_\textup{reg}$ in $\mathcal{A}\biMod\mathcal{A}$.  Then, $\Im(\widetilde\nu)=\calI\leq \mathcal{A}_\textup{reg}$.
\item The functors $j_!\circ j^\ast$ and $-\otimes_\mathcal{A}\Lambda$ 
are naturally isomorphic, and the counit $\nu \colon j_!\circ j^\ast\Rightarrow 1_{\rMod\mathcal{A}}$ is isomorphic to the following composition
\[
	-\otimes_\mathcal{A}\Lambda \buildrel 1_{\relbar}\otimes\widetilde\nu\over{\Relbar\joinrel\Longrightarrow}
		-\otimes_\mathcal{A}\mathcal{A}_\textup{reg}
		\buildrel\cong\over{\Relbar\joinrel\Longrightarrow} 1_{\rMod\mathcal{A}}.
\]
\item The morphism $\rho_\calI \colon\calI\otimes_\mathcal{A}\calI\to\calI$ in $\mathcal{A}\biMod\mathcal{A}$ induced by multiplication, factors as a composition
\(
	\smash[t]{\calI\otimes_\mathcal{A}\calI\buildrel\pi\over\to
		\Lambda\buildrel\overline{\nu}\over\to\calI}
\),
where $\pi$ is an epimorhism of $\mathcal{A}\mathchar`-\mathcal{A}$-bimodules, and $\overline{\nu}$ is the corestriction of  $\widetilde\nu\colon\Lambda\to\mathcal{A}_\textup{reg}$ to its image.
\item $\mathcal{X}_I$ is the full subcategory of all $X\in\rMod \A$ for which the morphism $\rho_X\colon X\otimes_\mathcal{A}\calI\to X$, induced by multiplication, is an isomorphism.
\end{enumerate}
\end{prop}
\begin{proof}
 \noindent(i) Observe that, by construction, $\widetilde\nu (b,a)\leqdef \nu_{H_a}(b)$, for all  $a,b\in\Obj(\mathcal{A})$. Hence, 
 $\im(\widetilde\nu (b,a))=\im(\nu_{H_a})(b)=(H_aI)(b)=\calI(b,a)$,  by Proposition~\ref{prop.relevantfacts-recollement}\,(ii).

\smallskip \noindent(ii) By Lemma~\ref{lem.checking natural isomorphism on generators}, we just need to check that $j_!\circ j^\ast\circ\mathbf{y}$ and $(-\otimes_\mathcal{A}\Lambda)\circ\mathbf{y}$ are naturally isomorphic in $[\mathcal{A},\rMod\mathcal{A}]_!$. But this is clear since, for each $a\in\Obj(\mathcal{A})$, there is a sequence of natural isomorphisms:
\[
(j_!\circ j^\ast\circ\mathbf{y})(a)=(j_!\circ j^\ast)(H_a)=\Lambda (-,a)\cong H_a\otimes_\mathcal{A}\Lambda= [(-\otimes_\mathcal{A}\Lambda)\circ\mathbf{y}](a).
\]

\noindent(iii) Apply the right exact functor $-\otimes_\mathcal{A}\calI\colon\mathcal{A}\biMod\mathcal{A}\to\mathcal{A}\biMod\mathcal{A}$ to the short exact sequence $\smash[t]{0\rightarrow\ker(\widetilde\nu)\buildrel\iota\over\hookrightarrow\Lambda\buildrel\overline{\nu}\over\to\calI\rightarrow 0}$ (see assertion~(i)). By Corollary~\ref{cor.properties tensor product},  $[\ker(\widetilde\nu)\otimes_\mathcal{A}\calI](b,a)\cong [\ker(\widetilde\nu)](-,a)\otimes_{\mathcal{A}}\calI(b,-)=\ker(\widetilde\nu (-,a))\otimes_\mathcal{A}\calI(b,-)$.
By construction, the map $\wbar{\nu}(-,a)\colon\Lambda (-,a)\to\calI(-,a)$ is just the corestriction of  $\nu_{H_a}\colon j_!j^\ast(H_a)\to\im(\nu_{H_a})=H_a\calI=\calI(-,a)$, and so  $\ker(\tilde\nu (-,a))\in\mathcal{T}_I$. Hence, $[\ker(\tilde\nu)\otimes_\mathcal{A}\calI] (b,a)=0$, for all $b,a\in\Obj(\mathcal{A})$, by Proposition~\ref{prop.relevantfacts-recollement}\,(iii). In particular, $\ker(\tilde\nu)\otimes_\mathcal{A}\calI=0$, and so $\overline{\nu}\otimes 1_\calI\colon\Lambda\otimes_\mathcal{A}\calI\to\calI\otimes_\mathcal{A}\calI$ is an isomorphism. 
Consider the following commutative square, where the vertical arrows are induced by multiplication, and $\overline{\nu}\otimes1_\calI$ is invertible by the previous discussion:
\[
\xymatrix@R=18pt@C=25pt{%
	\Lambda\otimes_\mathcal{A}\calI \ar[r]^-{\overline{\nu}\otimes1_\calI}\ar[d]_-{\rho_\Lambda} &
		\calI\otimes_\mathcal{A}\calI \ar[d]^-{\rho_\calI} \cr
	\Lambda \ar[r]^{\overline{\nu}} & \calI,
}
\]
where the commutativity follows because $\overline{\nu}$ is a morphism of $\mathcal{A}\mathchar`-\mathcal{A}$-bimodules. But $\rho_\calI\circ (\overline{\nu}\otimes 1_\calI)=\overline{\nu}\circ\rho_\Lambda$ implies that  $\rho_\calI=\overline{\nu}\circ (\rho_\Lambda\circ (\overline{\nu}\otimes 1_\calI)^{-1})$, so that $\pi\leqdef \rho_\Lambda\circ (\overline{\nu}\otimes 1_\calI)^{-1}$ is a morphism $\calI\otimes_\mathcal{A}\calI\to\Lambda$ in $\mathcal{A}\biMod\mathcal{A}$ such that $\overline{\nu}\circ\pi =\rho_\calI$. To conclude, observe that $(\rho_\Lambda)_{(-,a)}\colon(\Lambda\otimes_\mathcal{A}\calI)(-,a)=\Lambda (-,a)\otimes_\mathcal{A}\calI\to\Lambda (-,a)$ is an epimorphism (for all $a\in \Obj(\A)$), by Proposition~\ref{prop.relevantfacts-recollement}(i), as it can be identified with the multiplication map $j_!j^\ast(H_a)\otimes_\mathcal{A}\calI\to j_!j^\ast(H_a)\in \im(j_!)$. 

\smallskip \noindent(iv) If $X\in\mathcal{X}_I$, then $1_X\otimes\widetilde\nu \colon X\otimes_\mathcal{A}\Lambda\to X\otimes_\mathcal{A}\mathcal{A}_\textup{reg}\cong X$ is an isomorphism, by assertion~(ii). Furthermore, $1_X\otimes\widetilde\nu$ has a canonical  factorization:
\[
	1_X\otimes\widetilde\nu=\rho_X\circ(1_X\otimes\overline{\nu})\colon \xymatrix@C=25pt{X\otimes_\mathcal{A}\Lambda\ar[r]^-{1_X\otimes\overline{\nu}}&
		X\otimes_\mathcal{A}\calI\ar[r]^-{\rho_X}& X.}
\]
where $1_X\otimes\overline{\nu}$ is an epimorphism, since so is $\overline{\nu}\colon\Lambda\to\calI$. As a consequence, $\rho_X\colon X\otimes_\A \calI\to X$ is an isomorphism.

Conversely, suppose that $\rho_X\colon X\otimes_\mathcal{A}\calI\to X$ is an isomorphism. Then, also $\rho_X\otimes 1_\calI\colon X\otimes_\mathcal{A}\calI\otimes_\mathcal{A}\calI\to X\otimes_\mathcal{A}\calI$ is an isomorphism. On the other hand,  $\rho_X\otimes 1_\calI=1_X\otimes\rho_\calI$, since the former takes the standard generator  $x\otimes\eta\otimes\eta'$ to $(x\eta)\otimes\eta'$, while the latter takes it to $x\otimes (\eta\eta')$; but $ (x\eta)\otimes\eta'=x\otimes (\eta\eta')$ in $X\otimes_\mathcal{A}\calI$. Considering now the factorization $\rho_\calI=\overline{\nu}\circ\pi$ of  assertion~(iii), we deduce that $1_X\otimes\overline{\nu}\colon X\otimes_\mathcal{A}\Lambda\to X\otimes_{\mathcal{A}}\calI$ is an isomorphism which, in turn, implies that also the composition $\rho_X\circ(1_X\otimes\overline{\nu})$ is  an isomorphism. But $\rho_X\circ(1_X\otimes\overline{\nu})$ is isomorphic to the colocalization map $j_!j^\ast X\cong X\otimes_\mathcal{A}\Lambda\to X$, and so $X\in\mathcal{X}_I=\im(j_!)$.
\end{proof}

\section{Approaching the Cuadra--Simson Problem via colocalization}%
\label{sec.Cuadra-Simson via colocal.}
\subsection{Reduction of the problem to colocalizations}%

Recall the following facts and definitions.

\begin{rem}
\label{r:Sources}%
Let $\mathcal{G}$ be a Grothendieck category.
\begin{enumerate}
\item (\cite[Corollary~1.4]{R06}) $\mathcal{G}$ is $\mathrm{AB}\mathchar`-4^\ast$ if, and only if, each $M\in\Obj(\mathcal{G})$ admits a \emph{projective effacement\/}, that is, an epimorphism $p_M$ ending at $M$ such that the canonical morphism $\Ext^1_{\mathcal{G}}(p_M,N)$ is zero, for all $N\in\Obj(\mathcal{G})$.

\item $\mathcal{G}$ is $\mathrm{AB}\mathchar`-6$ if, for any  $M\in\Obj(\mathcal{G})$, and any family  $\{(M_i(\lambda))_{i\in I_\lambda}\mid \lambda\in \Lambda\}$ of direct systems of subobjects of $M$, the following equality holds:
\[
	\bigcap_{\lambda\in\Lambda}\Bigl(\sum_{i\in I_\lambda}M_i(\lambda)\Bigr)
		= \sum_{(i_\lambda)_{\lambda\in\Lambda}\in\prod_{\lambda\in\Lambda}I_\lambda}
			\Bigl(\bigcap_{\lambda\in\Lambda}M_{i_\lambda}(\lambda)\Bigr) .
\]
Equivalently (\cite[Theorem~1]{R67}), if any object $M\in\Obj(\mathcal{G})$ is the sum of its \emph{finite type subobjects\/}, where a given $L\le M$ is of finite type if, for any direct system $(M_i)_{i\in I}$ of subobjects of $M$ such that $\sum_IM_i=M$, we have that $L\cap M_i=L$, for some $i\in I$. In particular, any finitely generated subobject is of finite type, and so any Grothendieck category that is locally finitely generated is also $\mathrm{AB}\mathchar`-6$.
\end{enumerate}
\end{rem}

The following is a slight generalization of \cite[Theorem~1]{R65}.

\begin{prop}
\label{prop.Roos result}%
The following  are equivalent, for a Grothendieck category $\mathcal{G}$:
\begin{enumerate}[(a)]
\item $\mathcal{G}$ is $\mathrm{AB}\mathchar`-4^\ast$ and $\mathrm{AB}\mathchar`-6$;
\item There is a small preadditive category $\mathcal{A}$ (with just one object) and an idempotent ideal $\calI$ of $\mathcal{A}$ such that $\mathcal{G}$ is equivalent to the associated Giraud subcategory $\mathcal{G}_I$ of $\rMod\mathcal{A}$;
\item There is a small preadditive category $\mathcal{A}$ (with just one object) and an idempotent ideal $\calI$ of $\mathcal{A}$ such that $\mathcal{G}$ is equivalent to $\mathcal{X}_I$.
\end{enumerate}
\end{prop}
\begin{proof}
\smallskip \noindent``$\textup{(b)}\Leftrightarrow\textup{(c)}$'' follows from Proposition~\ref{prop.relevantfacts-recollement}\,(v).

\smallskip \noindent``$\textup{(a)}\Rightarrow\textup{(b)}$'' It is included in \cite[Theorem~1]{R65}.

\smallskip \noindent``$\textup{(b)}\Rightarrow\textup{(a)}$'' The case when $\mathcal{A}=R$ is a ring is included in \cite[Theorem~1]{R65}. We prove the case when $\mathcal{A}$ is any small preadditive category. From \cite[Proposition~3.5]{PV} we obtain that $\mathcal{G}$ is $\mathrm{AB}\mathchar`-4^\ast$. We claim that, if  $X\in\mathcal{X}_{\calI}=\im(j_!)$ and $L\leq X$ is a finitely generated $\mathcal{A}$-submodule, then $j^\ast(L)\leq j^\ast(X)$ is a subobject of finite type. Observe that, for any $Y\in\Obj(\mathcal{G})_I$, we can write $j_!(Y)=\bigcup_{\Lambda} L_\lambda$ as the direct union of its finitely generated submodules, and so, if our claim is true, we can write $Y\cong (j^\ast\circ j_!)(Y)=j^\ast(\bigcup_{\Lambda}L_\lambda)=\bigcup_{\Lambda}j^*(L_\lambda)$  as a direct union of a family of subobjects of finite type, showing that $\mathcal G$ is $\mathrm{AB}\mathchar`-6$, as desired. 

Let us verify our claim. If we write $j^\ast(X)=\bigcup_{\Lambda}Y_\lambda$ as a direct union of a direct system of subobjects, where $u_{\lambda\mu}\colon Y_\lambda\hookrightarrow Y_\mu$ and $u_\lambda\colon Y_\lambda\hookrightarrow j^\ast(X)$ in $\mathcal{G}$ are the corresponding inclusions, for all $\lambda\leq \mu\in \Lambda$, then $(j_!(Y_\lambda)_{\lambda\in \Lambda},(j_!(u_{\lambda\mu}))_{\lambda\leq \mu})$ is a direct system in $\rMod\mathcal{A}$ such that $\varinjlim_\Lambda j_!(Y_\lambda)=j_!(j^\ast(X))\cong X$. But then $X\cong (j_!\circ j^\ast)(X)=\bigcup_{\Lambda}\im(j_!(u_\lambda))$ is a direct union of submodules. Identifying $X$ with $(j_!\circ j^\ast)(X)$, there is some $\mu\in \Lambda$ such that $L\subseteq \im(j_!(u_\mu))$ (i.e., $L=L\cap\im(j_!(u_\mu))$) since $L$ is finitely generated. Furthermore,  the canonical map $j_!(u_\mu)\colon j_!(Y_\mu)\to (j_!\circ j^*)(X)$ has kernel in $\ker(j^\ast)=\mathcal{T}_I$, since $j^\ast$ is exact and $(j^\ast\circ j_!)(u_\mu)\cong u_\mu$ is a monomorphism. It follows that $j^\ast$ induces an isomorphism $Y_\mu\cong (j^\ast\circ j_!)(Y_\mu)\mathrel{\smash[t]{\buildrel\cong\over\to}}j^\ast(\im(j_!(u_\mu))$, so that, with the obvious abuse of notation, $j^\ast(L)\leq Y_\mu$.
\end{proof}

The key point to tackle the problem via colocalization is the following.

\begin{prop}\label{prop.Carlos result}%
Any locally finitely presented Grothendieck category $\mathcal{G}$ with enough flat objects is $\mathrm{AB}\mathchar`-6$ and $\mathrm{AB}\mathchar`-4^\ast$. In particular, there are a small preadditive category $\mathcal{A}$, and an idempotent ideal $\calI$ of $\A$, such that $\mathcal{G}\cong\mathcal{G}_I\cong\mathcal{X}_I$.
\end{prop}
\begin{proof}
By Remark~\ref{r:Sources}, any locally finitely presented Grothendieck category is $\mathrm{AB}\mathchar`-6$. Note that if in a composition of epimorphisms $q\circ p$ in $\mathcal{G}$ one of the two is a projective effacement, then so is $q\circ p$ by the functoriality of $\Ext^1$. Since, by hypothesis, each object is epimorphic image of a flat object, in order to prove that all objects of $\mathcal{G}$ admit a projective effacement, equivalently that $\mathcal{G}$ is $\textup{AB}\mathchar`-4^\ast$ (see Remark~\ref{r:Sources}), we just need to check that all flat objects have a projective effacement. Let then $F$ be a flat object and let $(M_\lambda)_{\lambda\in \Lambda}$ be a direct system in $\fp(\mathcal{G})$ such that $F=\varinjlim_\Lambda M_\lambda$. Given $X\in\Obj(\mathcal{G})$, pick any element in $\Ext^1_{\mathcal{G}}(F,X)$, represented by a (pure) short exact sequence $\epsilon$. We can construct a commutative diagram with pure exact rows, where $\epsilon$ is the lower one:
\[
\xymatrix{%
	\hphantom{\epsilon:{}}0 \ar[r] & K \ar[r]\ar@{.>}[d]_-\alpha &
		\coprod\nolimits_\Lambda M_\lambda \ar[r]^-{p}\ar@{.>}[d]^-\beta & F \ar[r]\ar@{=}[d] & 0 \cr
	\epsilon : 0 \ar[r] & X \ar[r]^-f & Y \ar[r]^-g & F \ar[r] & 0
}
\]
Here the morphism $\beta$ is a lifting of $p$, which exists since each $M_\lambda$ is finitely presented, and the morphism $\alpha$ is given by the universal property of kernels. The connecting homomorphism $\Hom_{\mathcal{G}}(K,X)\to\Ext^1_{\mathcal{G}}(F,X)$ sends $\alpha\mapsto[\epsilon]$, and so it is surjective, by the arbitrariness of $\epsilon$. In particular, $\Ext^1_{\mathcal{G}}(p,X)=0$, for all $X\in\Obj(\mathcal{G})$, so that $p\colon\coprod_\Lambda M_\lambda\to F$ is a projective effacement.
\end{proof}

\subsection{On the colocalizing subcategory $\mathcal{X}_\calI$}%
Recall that a functor between small categories $F\colon \mathcal C'\to \mathcal C$ is called {\em final} if, for all functors $D\colon \mathcal C\to \mathcal D$ ($\mathcal D$ any category) the colimit of $D$ exists if and only if the colimit of $D\circ F$ exists and, in that case, $\colim(D\circ F)\cong \colim D$.
\begin{lem}
\label{lem.smaller-directed-set}%
Let $\Lambda$ be a nonempty set with two partial orders $\preceq$ and $\leq$ such that, for all $\lambda,\mu\in \Lambda$, the following properties hold:
\begin{enumerate}[(1)]
\item if $\lambda\preceq \mu$ then $\lambda\leq \mu$;
\item if $\lambda\leq \mu$, then $\lambda\preceq \nu$ and $\mu\preceq\nu$, for some $\nu\in \Lambda$.
\end{enumerate}
Then, $(\Lambda,\preceq)$ and $(\Lambda,\mathord\leq)$ are both directed sets. Furthermore, the constant map $U\leqdef1_\Lambda\colon(\Lambda,\preceq)\to(\Lambda,\leq)$ is a final functor.
\end{lem}
\begin{proof}
By \cite[Lemma~8.3.4]{Riehl}, $U$ is final if and only if, for each $\lambda\in \Lambda$, the comma category $(\lambda\downarrow \Lambda)$ is not empty and connected (i.e., any two objects are connected by a finite zig-zag of morphisms). The objects of $(\lambda\downarrow \Lambda)$ are of the form $(\lambda, \lambda\leq U(\mu)=\mu)$, so that $(\lambda\downarrow \Lambda)$ is not empty as it contains, at least, the object $(\lambda,\lambda\leq \lambda)$. Moreover, there is exactly one morphism in $(\lambda\downarrow \Lambda)$ from the object $(\lambda,\lambda\leq \mu_1)$ to the object $(\lambda,\lambda\leq \mu_2)$ if $\mu_1\preceq \mu_2$, and no morphism otherwise. Now, by condition (2), we can find $\nu$ such that $\mu_1\preceq \nu$ and $\mu_2\preceq \nu$ and, using (1), it is also clear that $\lambda\leq \nu$. In particular, our two objects are connected by the zig-zag $(\lambda,\lambda\leq \mu_1)\to (\lambda,\lambda\leq \nu)\leftarrow (\lambda,\lambda\leq \mu_2)$. 
\end{proof}


We are now ready to identify the objects of $\mathcal{X}_{\calI}$ which are relevant to tackle the Cuadra--Simson Problem.

\begin{prop}
\label{prop.Proposition 4.4 of Lorenzo-notes}%
Let $\mathcal{A}$ be a small preadditive category and let $\calI$ be an idempotent ideal of $\mathcal{A}$. Then, the following assertions hold true:
\begin{enumerate}[(i)]
\item $\fp(\mathcal{X}_I)=\mathcal{X}_I\cap\rmod\mathcal{A}$;
\item $\Proj(\mathcal{X}_I)=\mathcal{X}_I\cap\Proj(\mathcal{A})=\mathcal{C}_I\cap\Proj(\mathcal{A})$;
\item $\Flat(\mathcal{X}_I)=\mathcal{X}_I\cap\Flat(\mathcal{A})=\mathcal{C}_I\cap\Flat(\mathcal{A})$.
\item A right $\mathcal{A}$-module $F$ is in $\mathcal{C}_I\cap\Flat(\mathcal{A})$ if, and only if, there is a direct system $\smash[t]{(\what{a}_\lambda\buildrel\alpha_{\lambda,\mu}\over\to\what{a}_\mu)_{\lambda\leq \mu}}$ in $\what{\mathcal{A}}$ such that $\alpha_{\lambda,\mu}\in\what{\calI}(\what{a}_\lambda,\what{a}_\mu)$ whenever $\lambda<\mu$, and $F\cong\varinjlim_I (H_{\what{a}_\lambda}, (H_{\alpha_{\lambda,\mu}}))$.
\end{enumerate}
\end{prop}
\begin{proof}
\smallskip \noindent(i) We have seen in the proof of Proposition~\ref{prop.relevantfacts-recollement}\,(v) that $\iota\colon\mathcal{X}_I\hookrightarrow\rMod\mathcal{A}$ has $u\circ j^\ast$ as its right adjoint, where $j^\ast$ is itself a left adjoint, and $u\colon\mathcal{G}_I\to\mathcal{X}_I$ is an equivalence. Using  Lemma~\ref{l:CategoryTheory}\,(ii), we deduce that $\iota$ preserves finitely presented objects, and so $\fp(\mathcal{X}_I)\subseteq\mathcal{X}_I\cap\fp(\rMod\mathcal{A})=\mathcal{X}_I\cap\rmod\mathcal{A}$. But, since $\mathcal{X}_I$ is coreflective in $\rMod\mathcal{A}$, direct limits in $\mathcal{X}_I$ are computed as in $\rMod\mathcal{A}$, so that the inclusion $\mathcal{X}_I\cap\rmod\mathcal{A}\subseteq\fp(\mathcal{X}_I)$ also holds. 

\smallskip \noindent(ii) Applying Lemma~\ref{l:CategoryTheory}\,(i) to the adjunction $(\iota,u\circ j^\ast)$, we obtain that $\iota$ preserves projective objects, so that $\Proj(\mathcal{X}_I)\subseteq\mathcal{X}_I\cap\Proj(\mathcal{A})$ and, Proposition~\ref{prop.relevantfacts-recollement}, also that  $\Proj(\mathcal{X}_I)\subseteq\mathcal{C}_I\cap\Proj(\mathcal{A})$. On the other hand, if $P\in\mathcal{C}_I\cap\Proj(\mathcal{A})$, it follows from Proposition~\ref{prop.relevantfacts-recollement} that $\nu_P$ is an epimorphism and, therefore, $P$ is a direct summand of $(j_!\circ j^\ast)(P)\in\mathcal{X}_I$, which implies that $P\in\mathcal{X}_I\cap\Proj(\mathcal{A})$, as $\mathcal{X}_I$ is closed under direct summands in $\rMod\mathcal{A}$. Hence, $P\in\Proj(\mathcal{X}_I)$ by Lemma~\ref{l:CategoryTheory}\,(i), since $\iota \colon\mathcal{X}_I\hookrightarrow\rMod\mathcal{A}$ is (right) exact.

\smallskip \noindent(iii) Applying Lemma~\ref{l:Flatness}\,(ii.1) to the adjunction $(\iota,u\circ j^\ast)$, we deduce that $\iota$ preserves flat objects, that is,  $\Flat(\mathcal{X}_I)\subseteq\mathcal{X}_I\cap\Flat(\mathcal{A})$. As a consequence,  $\Flat(\mathcal{X}_I)\subseteq\mathcal{C}_I\cap\Flat(\mathcal{A})$. Conversely, let $F\in\mathcal{C}_I\cap\Flat(\mathcal{A})$. The morphism $\rho_F \colon F\otimes_\mathcal{A}\calI\to F\otimes_\mathcal{A}\mathcal{A}_\textup{reg}\cong F$, induced by multiplication, is a monomorphism since $F$ is flat, and it is an epimorphism  since $F\in\mathcal{C}_I$. It then follows from Proposition~\ref{prop.colocalization as tensor product}\,(iv) that $F\in\mathcal{X}_I\cap\Flat(\mathcal{A})$. Now apply Lemma~\ref{l:Flatness}\,(ii.b) to the adjunction $(\iota, u\circ j^\ast )$ to conclude that $F\in\Flat(\mathcal{X}_{\calI})$.

\smallskip \noindent(iv) Let $F$ be a right $\mathcal{A}$-module such that there is a direct system as indicated in the statement. Then, $F$ is flat by Proposition~\ref{prop.tensor-categorical-flats-the-same}. In order to see that $F\in\mathcal{C}_I$, i.e., $F=F\calI$, we need to prove that, for each $a\in\Obj(\mathcal{A})$, the following inclusion holds: $F(a)\subseteq (F\calI)(a)=\sum_{b\in\Obj(\mathcal{A}),\alpha\in\calI(a,b)}\im(F(\alpha))$. Fix $a\in \Obj(\A)$, and pick an element $x\in F(a)$. Since direct limits in $\rMod\mathcal{A}$ are computed pointwise, there is some $\mu\in \Lambda$,  with $u_\mu\colon H_{\what{a}_\mu}\to \varinjlim_\Lambda (H_{\what{a}_\lambda}, (H_{\alpha_{\lambda,\mu}}))=F$ the corresponding map to the direct limit, such that $x=(u_\mu)_a(\eta)$, for some $\eta\in H_{\what{a}_\mu}(a)=\what{\mathcal{A}}(a,\what{a}_\mu)$. But, whenever $\mu<\mu'$, we have that $u_\mu=u_{\mu'}\circ H_{\alpha_{\mu,\mu'}}$ and so $x=(u_{\mu'})_a(\alpha_{\mu,\mu'}\circ\eta)$, for all $\mu'>\mu$. This means that, up to replacement of $\mu$ by a $\mu'>\mu$ and $\eta$ by $\alpha_{\mu,\mu'}\circ\eta$ if necessary, we can and shall assume that $x=(u_\mu)_a(\eta)$ with $\eta\in\what{\calI}(a,\what{a}_\mu)$. Since $u_\mu\colon H_{\what{a}_\mu}\to F$ is a natural transformation of additive functors $\mathcal{A}^\textup{op}\to\Ab$, we have a commutative diagram
\[
\xymatrix@R=20pt@C.5em{%
	F(a) & F(\what a_\mu) \ar[l]_-{F(\eta)} \cr
	**[l]\what{\mathcal{A}}(a,\what a_\mu)=H_{\what a_\mu}(a) \ar[u]^-{(u_\mu)_a} &
		**[r] H_{\what a_\mu}(\what a_\mu)=\what{\mathcal{A}}(\what a_\mu,\what a_\mu)
		\ar[u]_-{(u_\mu)_{\what a_\mu}}\ar[l]_-{H_{\what a_\mu}(\eta)}
}
\]
It then follows that
\begin{align*}
	x &=(u_\mu)_a(\eta)=(u_\mu)_a (H_{\what{a}_\mu}(\eta)(1_{\what{a}_\mu})) =[(u_\mu)_a\circ H_{\what{a}_\mu}(\eta)](1_{\what{a}_\mu})\cr
    &=[F(\eta)\circ (u_\mu)_{\what{a}_\mu}](1_{\what{a}_\mu})=F(\eta) ( (u_\mu)_{\what{a}_\mu}(1_{\what{a}_\mu})).
\end{align*}
If $\what{a}_\mu\leqdef(a_{1,\mu},\ldots,a_{m,\mu})$ and $\eta\leqdef(\eta_k\colon a\to a_{k,\mu})_{k=1,\ldots,m}$, where $a_{k,\mu}\in\Obj(\A)$, for all $k=1,\ldots,m$, then $\eta_{k,\mu}\in\calI(a,a_{k,\mu})$, for all $k=1,\ldots,m$. In particular,  $x\in\im(F(\eta))=\sum_{k=1}^m\im(F(\eta_k))$, and so $x\in (F\calI)(a)$, as desired.

Suppose now that $F\in\mathcal{C}_I\cap\Flat(\mathcal{A})$. By Proposition~\ref{prop.tensor-categorical-flats-the-same} and the Additive Yoneda Lemma, there exists a direct system $\smash[t]{(\what{a}_\lambda\buildrel\alpha_{\lambda,\mu}\over\to\what{a}_\mu)_{\lambda\leq \mu}}$ in $\what{\mathcal{A}}$ such that $F\cong\varinjlim_\Lambda (H_{\what{a}_\lambda},H_{\alpha_{\lambda,\mu}})$, and we identify $F$ with this direct limit. The crucial point now is the following:

\subparagraph*{\emph{Claim}}%
If $\mu\in \Lambda$, $\what{a}\in\Obj(\what{\mathcal{A}})$ and $\beta\in\what{\mathcal{A}}(\what{a},\what{a}_\mu)$, then there exists some $\mu'\in \Lambda$ such that $\mu\leq \mu'$ and $\alpha_{\mu,\mu'}\circ\beta\in\what{\calI}(\what{a},\what{a}_{\mu'})$.

Once the claim is proved, we can apply it to $\what{a}=\what{a}_\mu$ and $\beta =1_{\what{a}_\mu}$ to deduce that, for each $\mu\in \Lambda$, there is  $\mu'>\mu$ in $\Lambda$ such that $\alpha_{\mu,\mu'}\in\what{\calI}(\what{a}_\mu,\what{a}_{\mu'})$. This allows us to consider a relation $\preceq$ in $\Lambda$ defined by the following rule: $\lambda\preceq \mu$ if, and only if, either $\lambda=\mu$ or $\lambda<\mu$ and $\alpha_{\lambda,\mu}\in\what{\calI}(\what{a}_\lambda,\what{a}_\mu)$. It is now routine to check that $\preceq$ is a partial order in $\Lambda$ and that $\preceq$ and $\leq$ satisfy the hypotheses of Lemma~\ref{lem.smaller-directed-set}. Hence, there is a direct system, based on $(\Lambda,\preceq )$, such that $\alpha_{\lambda,\mu}\in\what{\calI}(\what{a}_\lambda,\what{a}_\mu)$ whenever $\lambda\preceq \mu$ and the corresponding direct limit is isomorphic to $F$, thus ending the proof.

Let us prove our claim. It is not restrictive to assume that $\what{a}=a\in\Obj(\mathcal{A})$. But then $\beta\in H_{\what{a}_\mu}(a)$, and so $x\leqdef (u_\mu)_a(\beta)\in F(a)$, where $u_\mu\colon H_{\what{a}_\mu}\to F$ is the canonical map to the direct limit. But the equality $F=F\calI$ gives a family of morphisms $\beta_k\in\calI(a,b_k)$ and elements $x_k\in F(b_k)$, for $k=1,\ldots,m$, such that $x=\sum_{k=1}^mF(\beta_k)(x_k)$. Since $x_k\in\varinjlim_\Lambda\what{\mathcal{A}}(b_k,\what{a}_\lambda)$ we can find a large enough index $\mu'\geq \mu$ such that $x_k=(u_{\mu'})_a(\eta_k)$ for some $\eta_k\in H_{a_{\mu'}}(b_k)=\what{\mathcal{A}}(b_k,\what{a}_{\mu'})$, and for all $k=1,\ldots,m$. Letting $\beta'=\alpha_{\mu,\mu'}\circ\beta$,  the element $\beta '-\sum_{k=1}^m\eta_k\circ\beta_k$ is in the kernel of $(u_{\mu'})_a\colon H_{\what{a}_{\mu'}}(a)=\what{\mathcal{A}}(a,\what{a}_{\mu'})\to F(a)$. Using the standard properties of direct limits in $\Ab$, we can find some $\mu'<\mu''\in \Lambda$, such that $H_{\alpha_{\mu',\mu''}}(\beta '-\sum_{k=1}^m\eta_k\circ\beta_k)=0$ or, equivalently,  $\alpha_{\mu',\mu''}\circ\beta'=\sum_{k=1}^m\alpha_{\mu',\mu''}\circ\eta_k\circ\beta_k$. Hence, 
$\alpha_{\mu,\mu''}\circ\beta'\in\what{\calI}(a,\what{a}_{\mu''})$, as  $\beta_k\in\calI(a,b_k)$, for all $k=1,\ldots,m$. 
\end{proof}

Recall that the trace $\tr_{\cal S}(\A)$ of $\cal S$ in $\A$ was introduced in Definition~\ref{def. trace}. 

\begin{lem}
\label{lema.generation of C versus of X}%
The following assertions are equivalent for a subclass  $\mathcal{S}\subseteq\mathcal{X}_I$:
\begin{enumerate}[(a)]
\item $\mathcal{C}_I=\Gen(\mathcal{S})$;
\item $\tr_\mathcal{S}(\mathcal{A})=\calI$;
\item $\mathcal{S}$ is a class of generators of the Grothendieck category $\mathcal{X}_I$.
\end{enumerate}
\end{lem}
\begin{proof}
\smallskip \noindent``$\textup{(a)}\Rightarrow\textup{(b)}$'' Assertion~(a) implies an equality $\tr_\mathcal{S}(H_a)= H_aI=\calI(-,a)$ for all $a\in\Obj(\mathcal{A})$, so that $\tr_\mathcal{S}(\mathcal{A})=\calI$.

\smallskip \noindent``$\textup{(b)}\Rightarrow\textup{(a)}$'' Assertion~(b) implies that each $\calI(-,a)$ is an epimorphic image of some coproduct of objects in $\mathcal{S}$. Since $\mathcal{C}_I=\Gen(\calI(-,a)\mid a\in\Obj(\mathcal{A}))$ (see \cite[Subsection 4.2]{PSV}), we conclude that $\mathcal{C}_I=\Gen(\mathcal{S})$.

\smallskip \noindent``$\textup{(c)}\Rightarrow\textup{(a)}$'' Bearing in mind that the inclusion $\mathcal{X}_\calI\hookrightarrow\rMod\mathcal{A}$ preserves coproducts and cokernels, assertion~(c) is equivalent to the equality $\mathcal{X}_\calI=\Pres(\mathcal{S})$. The equality $\mathcal{C}_I=\Gen(\mathcal{S})$ then follows from Proposition~\ref{prop.relevantfacts-recollement}.

\smallskip \noindent``$\textup{(a)}\Rightarrow\textup{(c)}$'' Let $X\in\mathcal{X}_I\subseteq\mathcal{C}_I$ and consider a short exact sequence of the form: $0\to L\to\coprod_{\Lambda}S_\lambda\mathrel{\smash[t]{\buildrel p\over\to}} X\rightarrow 0$, where $S_\lambda\in\mathcal{S}$ for all $\lambda\in \Lambda$. Applying the functor $-\otimes_\mathcal{A}\calI\colon\rMod\mathcal{A}\to\rMod\mathcal{A}$, that ``fixes'' $\coprod_{I}S_\lambda$ and $X$ by Proposition~\ref{prop.colocalization as tensor product}, we obtain an exact sequence $L\otimes_\mathcal{A}\calI\to\coprod_{\Lambda}S_\lambda\mathrel{\smash[t]{\buildrel p\over\to}} X\to 0$. Hence, $L\in\mathcal{C}_I=\Gen(\mathcal{S})$, and so $X\in\Pres(\mathcal{S})$, showing that $\mathcal{S}$ is a class of generators of $\mathcal{X}_\calI$.
\end{proof}

We have now the following:

\begin{thm}\label{thm.always enough flats}%
Let $\mathcal{A}$ be a small preadditive category, and let $\mathcal{G}_I\subseteq\rMod\A$ be a Giraud subcategory associated with an idempotent ideal $I\leq \A$. Then:
\begin{enumerate}[(i)]
\item $\mathcal{G}_I$ has always enough flat objects or, equivalently, a flat generator.
\item $\mathcal{G}_I$ is locally finitely presented if and only if $\mathcal{C}_I=\Gen(\mathcal{X}_I\cap\rmod\mathcal{A})$.
\item $\mathcal{G}_I$ has enough projectives or, equivalently, a projective generator, if and only if $\calI$ is the trace of a projective right $\mathcal{A}$-module.
\item $\mathcal{G}_{\calI}$ has a set of finitely generated projective generators if, and only if, $\calI$ is the trace of a set of finitely generated projective right $\mathcal{A}$-modules if, and only if, $\calI$ is generated by a family of idempotents in $\what{\mathcal{A}}$.
\end{enumerate}
\end{thm}
\begin{proof}
By Proposition~\ref{prop.relevantfacts-recollement}, it is enough to prove the assertions for $\mathcal{X}_I$.

\smallskip \noindent(i) We shall prove that, for each  $\alpha\in\calI(a,b)$, the morphism $H_\alpha \colon H_a\to H_b$ factors through some countably generated module in $\mathcal{C}_I\cap\Flat(\mathcal{A})$. By \cite[Lemma~2.5]{PSV} this will imply that $\calI=\tr_F(\mathcal{A})$, where $F$ is the coproduct of a family of representatives of each isomorphism class of the countably generated $\mathcal{A}$-modules in $\mathcal{C}_I\cap\Flat(\mathcal{A})$. To do that we shall construct inductively a sequence
\(
	a=\what{a}_0\buildrel\alpha^1\over\to\what{a}_1\buildrel\alpha^2\over\to\cdots
\)
of morphisms in $\what{\mathcal{A}}$, and a second sequence of morphisms
\(
	\smash[t]{\what{a}{}^n\buildrel\beta^n\over\to b}
\)
in $\what{\mathcal{A}}$, such that $\beta^0=\alpha$, and $\beta^n\circ\alpha^n=\beta^{n-1}$, for all $n>0$. Moreover, the sequences will satisfy that $\alpha^n\in\what{\calI}(\what{a}_{n-1},\what{a}_n)$, and $\beta^n\in\what{\calI}(\what{a}_{n},b)$, for all $n>0$. We put $\what{a}^0\leqdef a$ and $\beta^0\leqdef\alpha$. Assume now that $n>0$, and that $\alpha^k\colon\what{a}{}^{\,k-1}\to\what{a}{}^{\,k}$, and $\beta^k\colon\what{a}{}^{\,k}\to b$, have been defined, for all $0\leq k<n$. Since $\beta^{n-1}\in\what{\calI}(\what{a}{}^{\,n-1},b)$, the idempotency of $\what{\calI}$ in $\what{\mathcal{A}}$ implies that there is a factorization
\[
	\beta^{n-1}=\beta^n\circ\alpha^n\colon\what{a}{}^{\,n-1}\buildrel\alpha^n\over\longrightarrow
		\what{a}{}^{\,n}\buildrel\beta^n\over\longrightarrow b,
\]
for some $\what{a}{}^{\,n}\in\Obj(\what{\mathcal{A}})$, where $\alpha^n\in\what{\calI}(\what{a}{}^{\,n-1},\what{a}{}^{\,n})$ and $\beta^n\in\what{\calI}(\what{a}{}^n,b)$. Applying the Yoneda Embedding $\mathbf{y}\colon\what{\mathcal{A}}\to\rMod\mathcal{A}$, which takes each $\what{a}=\coprod_{i=1}^ra_i\in\Obj(\what{\mathcal{A}})$ to $H_{\what{a}}\cong\coprod_{i=1}^rH_{a_i}$, we get the following commutative diagram in $\rMod\mathcal{A}$:
\[
\xymatrix{%
	H_a=H_{\what a_0} \ar[r]^-{H_{\alpha^1}}\ar@/_15pt/[drr]|-{H_\alpha^{\vphantom0}=H_{\beta^0}} &
		H_{\what a_1} \ar[r]\ar@/_8pt/[dr]|-{H_{\beta^1}} &
		\cdots \ar@{}[d]|(.6){\textstyle\cdots}\ar[r]^{H_{\alpha^n}} &
		H_{\what a_n} \ar[r]\ar@/_-8pt/[dl]|-{H_{\beta^n}} & \cdots \ar@{}[dll]^(.475){\textstyle\ \ \cdots} \cr
	&& H_{\what b}
}
\]
As a consequence, we find the following factorization of $H_\alpha$,
\[
	H_a=H_{\what{a}_0}\buildrel\what{\alpha}\over\longrightarrow
		\varinjlim_{n\in\N} (H_{\what{a}_n},H_{\alpha^n})\reqdef
		F\buildrel\what{\beta}\over\longrightarrow H_b,
\]
where $F$ is in $\mathcal{C}_I\cap\Flat(\mathcal{A})$, by Proposition~\ref{prop.Proposition 4.4 of Lorenzo-notes}\,(iv). Moreover, $F$ is countably generated as it is a quotient of the countably generated module $\coprod_{\mathbb{N}}H_{\what{a}_n}$.

\smallskip \noindent(ii) The category $\mathcal{X}_I$ is locally finitely presented if, and only, if it is generated by its (skeletally small) subcategory $\fp(\mathcal{X}_I)$. By Proposition~\ref{prop.Proposition 4.4 of Lorenzo-notes} and Lemma~\ref{lema.generation of C versus of X}, this occurs precisely when $\mathcal{C}_I=\Gen(\mathcal{X}_I\cap\rmod\mathcal{A})$.

\smallskip \noindent(iii,\,iv) Since $\mathcal{X}_{\calI}$ is a Grothendieck category, it has enough projectives if, and only if, it has a set $\mathcal{P}$ of projective generators, that for the proof of assertion~(iv), can be taken all finitely generated. By Lemma~\ref{lema.generation of C versus of X}, this happens if, and only if, $I=\tr_\mathcal{P}(\mathcal{A})$, and $\mathcal{P}\subseteq\mathcal{X}_{\calI}\cap\Proj(\mathcal{A})$, by Proposition~\ref{prop.Proposition 4.4 of Lorenzo-notes}. Moreover, the objects in $\mathcal{P}$ are finitely generated (i.e., finitely presented) in $\mathcal{X}_I$ if, and only if $\mathcal{P}\subseteq\rmod\mathcal{A}$. That is, if and only if $\mathcal{P}\subseteq\proj(\mathcal{A})$. The final part of assertion~(iv), regarding the idempotents in $\what \A$, follows by Proposition~\ref{prop.traces-versus-idempotents}. 
\end{proof}

From the above theorem and its proof we derive the following:

\begin{cor}
\label{cor.idempotent=trace of flat}%
Let $\mathcal{A}$ be a small preadditive category (e.g., a ring). Each idempotent ideal of $\mathcal{A}$ is the trace of a flat $\mathcal{A}$-module.
\end{cor}

As a consequence of Theorem~\ref{thm.always enough flats} and Propositions~\ref{prop.Roos result} and \ref{prop.Carlos result}, we get the following:

\begin{cor}
\label{cor.Cuadra-Simson versus Djament}%
Consider the following claims for a Grothendieck category $\mathcal{G}$:
\begin{enumerate}[(a)]
\item $\mathcal{G}$ is locally finitely presented and has enough flat objects;
\item $\mathcal{G}$ is locally finitely presented and $\mathrm{AB}\mathchar`-4^\ast$;
\item $\mathcal{G}$ is $\mathrm{AB}\mathchar`-6$ and $\mathrm{AB}\mathchar`-4^\ast$;
\item there is a small preadditive category (a ring) $\mathcal{A}$ with an idempotent ideal $\calI$ such that $\mathcal{G}$ is equivalent to the Giraud subcategory $\mathcal{G}_{\calI}$ of $\rMod\mathcal{A}$.
\end{enumerate}
Then, the implications $\textup{``(a)}\Leftrightarrow\textup{(b)}\Rightarrow\textup{(c)}\Leftrightarrow\textup{(d)''}$ hold true, but $\textup{``(b)}$\rotatebox[origin=c]{180}{$\not\Rightarrow$}$\textup{(c)''}$.
\end{cor}
\begin{proof}
\smallskip \noindent The implication ``$\textup{(a)}\Rightarrow\textup{(b)}$'' follows from Proposition~\ref{prop.Carlos result}, while the implication ``$\textup{(b)}\Rightarrow\textup{(a)}$'' follows from Proposition~\ref{prop.Roos result} and Theorem~\ref{thm.always enough flats}. The implication ``$\textup{(b)}\Rightarrow\textup{(c)}$'' is clear, and the equivalence ``$\textup{(c)}\Leftrightarrow\textup{(d)}$'' is Proposition~\ref{prop.Roos result}. Finally, it follows from Proposition~\ref{prop.Proposition 4.4 of Lorenzo-notes}, that the Grothendieck category described in \cite[Example~4.2]{R06} satisfies condition~(d) (and so, also (c)) but it cannot be locally finitely presented since, in that example, $I=\mathbf{m}=J(A)$ and so, by Nakayama's Lemma, $\fp(\mathcal{X}_{\calI})=\mathcal{X}_{\calI}\cap\rmod\mathcal{A}=0$.
\end{proof}

Recall that a hereditary torsion class in any Grothendieck category is also a Grothendieck category.

\begin{exmpl}
\label{ex.heredtor-not-AB4star-in-PID}
If $R$ is a (commutative) principal ideal domain that is not a field, then the category $\rTors R$ of torsion $R$-modules does not have a flat generator. Indeed, the proof of \cite[Example~2.11]{AM-P} also works in this slightly more general setting, thus proving that $\rTors R$ is not $\mathrm{AB}\mathchar`-4^\ast$ and, since it is clearly locally finitely presented, Corollary~\ref{cor.Cuadra-Simson versus Djament} tells us that it does not have enough flat objects or, equivalently, that it does not have a flat generator.
\end{exmpl}

\subsection{Reformulation of the Cuadra--Simson Problem and its refutation}%

According to Theorem~\ref{thm.always enough flats} and Corollary~\ref{cor.Cuadra-Simson versus Djament}, the Cuadra--Simson Problem can be given the following equivalent reformulations:

\begin{probl}[Cuadra--Simson reformulated]
\mbox{\label{probl.CS reformulated}}\vglue0pt%
\begin{enumerate}
\item Let $\mathcal{A}$ be a small preadditive category (or just a ring) and let $\calI=\calI^2$ be an ideal in $\A$, such that the associated Giraud subcategory $\mathcal{G}_I$ of $\rMod\mathcal{A}$ is locally finitely presented. Is $\calI$ the trace of a set of projective right $\mathcal{A}$-modules? Or, equivalently, does $\mathcal{G}_I$ have enough projectives?
\item Let $\mathcal{G}$ be an $\mathrm{AB}\mathchar`-4^\ast$ locally finitely presented Grothendieck category. Does $\mathcal{G}$ have enough projectives?
\end{enumerate}
\end{probl}

According to the reformulation \ref{probl.CS reformulated}\,(1), we see that the problem is a particular case of the following problem raised by Robert Miller in 1975 (see \cite{W}).

\begin{probl}[Miller]
Let $R$ be a ring and let $I$ be an idempotent ideal of $R$. When is $I$ the trace of a projective $R$-module?
\end{probl}

On the other hand, by the reformulation~\ref{probl.CS reformulated}\,(2), the Cuadra--Simson Problem is a generalization of the following problem posted online by Aur\'elien Djament (see \cite{Djament}):

\begin{probl}[Djament]
Let $\mathcal{G}$ be a locally noetherian (resp., locally coherent, or locally finitely presented) $\mathrm{AB}\mathchar`-4^\ast$ Grothendieck category. Does $\mathcal{G}$ have a family of (finitely generated) projective generators?
\end{probl}

Our next result shows two sources for potential counterexamples to the Cuadra--Simson Problem.

\begin{thm}
\label{teor.counterexample-recipe}%
Let $\mathcal{A}$ be a small preadditive category, and let $\calI$ be an idempotent ideal of $\mathcal{A}$. Suppose that one of the following two conditions is satisfied:
\begin{enumerate}[(i)]
\item The canonical morphism  $\rho_{\calI}\colon\calI\otimes_\mathcal{A}\calI\to\calI$ is an isomorphism, and $\calI(-,a)$ is a finitely presented right $\mathcal{A}$-module, for all $a\in\Obj(\mathcal{A})$.
\item $\calI$ is pure as a left ideal, i.e., the multiplication map $\rho_M\colon M\otimes_\mathcal{A}\calI\to M$ is a monomorphism, for all (finitely presented) right $\mathcal{A}$-modules $M$.
\end{enumerate}
Then, the associated Giraud subcategory $\mathcal{G}_I$ of $\rMod\mathcal{A}$ is locally finitely presented (and has enough flat objects). Moreover, if $\calI$ is not the trace in $\mathcal{A}$ of any projective right $\mathcal{A}$-module, then $\mathcal{G}_I$ does not have enough projective objects and, therefore, it provides a negative answer to the Cuadra--Simson Problem.
\end{thm}
\begin{proof}
The fact that $\mathcal{G}_I$ has enough flat objects is a consequence of Theorem~\ref{thm.always enough flats}, from which we also deduce the final part of the statement. 

\smallskip \noindent(i) By Corollary~\ref{cor.properties tensor product}\,(iii), the canonical morphism  $\rho_{\calI}\colon\calI \otimes_\mathcal{A}\calI\to\calI$ is an isomorphism if, and only if, the morphism $\rho_{\calI(-,a)}\colon\calI(-,a) \otimes_\mathcal{A}\calI\to\calI(-,a)$ is an isomorphism, for all $a\in\Obj(\mathcal{A})$. By Proposition~\ref{prop.colocalization as tensor product}\,(iv), this is equivalent to say that $\calI(-,a)\in \mathcal{X}_I$, for all $a\in\Obj(\mathcal{A})$. The fact that $\mathcal{C}_{\calI}=\Gen(\calI(-,a)\mid a\in\Obj(\mathcal{A}))$ (see \cite[Lemma~2.4]{PSV}) implies that $\mathcal{X}_{\calI}\;(\empty\cong\mathcal{G}_{\calI})$ is locally finitely presented (see Theorem~\ref{thm.always enough flats}\,(ii)).

\smallskip \noindent(ii) If $I$ is pure on the left then $\mathcal{C}_\calI=\mathcal{X}_\calI$, by  Proposition~\ref{prop.colocalization as tensor product}\,(iv). We claim that $\mathcal{C}_\calI$ is then closed under submodules, so that $(\mathcal{C}_{\calI},\mathcal{T}_{\calI})$ is a hereditary torsion pair in $\rMod \mathcal{A}$. Indeed, if $N$ is a submodule of $X\in\mathcal{C}_\calI$, then $X/N\in\mathcal{C}_\calI=\mathcal{X}_\calI$ and, by an argument similar to  the proof of ``$\textup{(a)}\Rightarrow\textup{(c)}$'' in Lemma~\ref{lema.generation of C versus of X},  $N\in\mathcal{C}_\calI$. 

In this paragraph we follow the terminology of \cite[Section~3]{PSV}. By \cite[Proposition~3.9]{PSV}, the  Grothendieck topology $\mathbf{G}$ associated to $(\mathcal{C}_{\calI},\mathcal{T}_{\calI})$ has a basis consisting of finitely generated right ideals. This implies that, for each $a\in\Obj(\mathcal{A})$, there is a set $\mathbf{B}_a$ of finitely generated submodules of $H_a$, such that $\mathcal{C}_I=\Gen(H_a/N\mid a\in\Obj(\mathcal{A}),\, N\in\mathbf{B}_a)$. Thus, $\{H_a/N\mid a\in\Obj(\mathcal{A}),\, N\in\mathbf{B}_a\}$ is a set of finitely presented generators of $\mathcal{X}_\calI=\mathcal{C}_\calI$ by Proposition~\ref{prop.Proposition 4.4 of Lorenzo-notes} and Theorem~\ref{thm.always enough flats}.
\end{proof}

In the following proposition we describe a class of rings that come with a canonical idempotent ideal that satisfies both conditions (i) and (ii) from Theorem~\ref{teor.counterexample-recipe}. This class of examples was suggested to us by Dolors Herbera and Pavel P\v r\'\i hoda \cite{HPV}. Furthermore, the proof of assertion~(iv) in the following proposition is due to Pavel P\v r\'\i hoda. The reader is referred to \cite{DP} for
the non-defined concepts and terminology used in the next proposition
and its proof.

\begin{prop}\label{exmpl.Dubrovin-Puninski} 
Let $R$ be a nearly simple uniserial domain, let $X=R/rR$, with $0\neq r\in D$, be its unique (up to isomorphism) cyclically presented torsion right $R$-module, denote by $S\leqdef\End_R(X_R)$ its endomorphism ring, and let $K$ be the idempotent two-sided ideal of $S$ consisting of the endomorphisms which are not surjective. Then, the following assertions hold true:
\begin{enumerate}[(i)]
\item ${}_SK$ is finitely presented;
\item the canonical map $K\otimes_SK\to K$ is an isomorphism;
\item $K$ is pure as a right ideal of $S$;
\item $K$ is not the trace of a projective left $S$-module.
\end{enumerate}
In particular,  the Giraud subcategory $\mathcal{G}_K$ of $S\lMod=\rMod{S^{\op}}$ associated with $K$ is locally finitely presented, it has enough flat objects, but not enough projectives. Therefore, the Cuadra--Simson Problem has a negative answer in general.
\end{prop}
\begin{proof}
The final statement is a direct consequence of Theorem~\ref{teor.counterexample-recipe} and assertions (i--iv), that are proved below.

\smallskip \noindent(i) By \cite[Lemma~4.4, Corollary~3.6]{DP}, the ring $S$ is left and right coherent and $K$ is principal, whence finitely presented, as a left ideal. 

\smallskip \noindent(ii) By \cite[Lemmas~5.3 and~6.1]{DP}, $S/K$ is an injective left $S$-module such that  $\mathrm{pd}_S(S/K)\leq 2$, equivalently $\mathrm{pd}_S(K)\leq 1$, where $\mathrm{pd}_S$  denotes the projective dimension as a left $S$-module. Denote by $\mathcal{T}'_K$ the TTF class in $S\lMod$ associated with $K$, which can be identified with $(S/K)\lMod$. If $T\in\mathcal{T}'_K$ and we consider a projective presentation in this latter category,
\[
	0\longrightarrow T' \longrightarrow (S/K)^{(J)}\longrightarrow
		T \longrightarrow0,
\]
where $J$ is some set, then we obtain the following exact sequence of $\Ext$-groups:
\[
	0=\Ext_S^1(K,S/K)^{(J)}\cong\Ext_S^1(K, (S/K)^{(J)})\to\Ext_S^1(K,T)\to\Ext_S^2(K,T')=0,
\]
showing that $K\in\mathcal{X}'_K\leqdef{}^{\perp_{0,1}}\mathcal{T}'_K$. We can now conclude by Proposition~\ref{prop.colocalization as tensor product}\,(iv).

\smallskip \noindent(iii) By \cite[Theorem~5.5]{DP}, it suffices to verify that the canonical map 
\[
\mu_{K,Z}\colon K\otimes_SZ\longrightarrow Z,
\]is injective, whenever $Z={}_SS$, $Z={}_S(S/K)$, or $Z={}_SK$. On the other hand, this is clear for $Z={}_SS$, as $K\otimes_SS\cong K$, so that $\mu_{K,S}$ is just the inclusion of $K$ in $S$, and also for $Z={}_S(S/K)$, as $K\otimes_S S/K\cong K/K^2=0$. Finally, the case $Z={}_SK$ is a consequence of part~(ii).

\smallskip \noindent(iv) Suppose, looking for a contradiction, that ${K}$ is the trace of a projective left $S$-module $P$. Consider an idempotent row-finite matrix $A=(a_{\lambda,\mu})\in \mathcal{M}_{\Lambda\times \Lambda}(S)$ (for some set $\Lambda$), giving a free presentation of $P$. Then, ${{K}}$ is generated by the content of $A$. Choose a non-trivial row in $A$, corresponding to some $\lambda \in\Lambda$. Since $A$ is row-finite, the family $\{\Im(a_{\lambda,\mu})\mid  \mu\in\Lambda\}$ is finite and totally ordered by inclusion (as $X$ is uniserial). Hence, there exists $\mu\in\Lambda$ such that $\Im(a_{\lambda,\mu})$ is maximal. Since  $A$ is idempotent, we can write $a_{\lambda,\mu}$ as follows:
\[
a_{\lambda,\mu} = \sum_{\nu\in\Lambda} a_{\lambda,\nu} a_{\nu,\mu}.
\]
Observe that, being $A$ row-finite, the sum on the right-hand side is finite and, since $X$ is uniserial, we obtain that $\Im(a_{\lambda,\mu})=\Im(a_{\lambda,\nu}a_{\nu,\mu})$, for some $\nu$; as we had chosen $\mu$ for which $\Im(a_{\lambda,\mu})$ is maximal, we deduce that:
\[
\Im(a_{\lambda,\nu})=\Im(a_{\lambda,\mu})=\Im(a_{\lambda,\nu}a_{\nu,\mu}).
\]
Hence, $a_{\lambda,\nu}(X)=a_{\lambda,\nu}(\Im(a_{\nu,\mu}))$, and so:
\[
X = a_{\lambda,\nu}^{-1} \left(a_{\lambda,\nu} (X)\right) = a_{\lambda,\nu}^{-1}\left(a_{\lambda,\nu}(\Im(a_{\nu,\mu}))\right) = \Im(a_{\nu,\mu}) + \ker(a_{\lambda,\nu}).
\]
Hence, since $X$ is uniserial, either $X=\ker(a_{\lambda,\nu})$, which is a contradiction since we have chosen $\lambda$ so that the corresponding row is non-trivial, and $\mu$ so that $a_{\lambda,\mu}$ has maximal image, or $X=\Im(a_{\nu,\mu})$ which is a contradiction because $a_{\nu,\mu}$ belongs to the ideal of non-epimorphisms. 
\end{proof}

\begin{rem}\mbox{}\vglue0pt%
\begin{enumerate}
\item Note that the ideal $K$ of Proposition~\ref{exmpl.Dubrovin-Puninski} is the trace of a projective right $S$-module (see \cite[Corollary~2.7]{W}), showing that the property of being the trace of a projective module is not left-right symmetric. Similarly, $K$ is pure as a right, but not as a left, ideal of $S$.

\item Several of the ideas that we have used in the proof of Proposition~\ref{exmpl.Dubrovin-Puninski} did already appear in the appendix
of \cite{B-H-P-S-T*}, although that part did not make it to the final version of the paper (see \cite{B-H-P-S-T}).

\item It follows by \cite[Theorem~A]{AM-B-P} that, in the setting of Proposition~\ref{exmpl.Dubrovin-Puninski}, for $K$ to be the trace of a projective left $S$-module, it is necessary for $S/K$ to be a silting module. In fact, under the extra assumption that all the ideals of $S$ are countably generated, it follows from \cite[Lemma~5.4\,(1), Example~5.12]{B-H-P-S-T*} that $S/K$ is not silting, hence precluding $K$ from being the trace of a projective left $S$-module. This provides an alternative proof for part (iv) of Proposition~\ref{exmpl.Dubrovin-Puninski}, in this particular case. We refer to \cite[Example~3.1, Remark~3.7]{DP} for a concrete example of a ring $S$ that fits into the setting of Proposition~\ref{exmpl.Dubrovin-Puninski} and whose ideals are all countably generated. 

\item Observe that the ring $S$ in Proposition~\ref{exmpl.Dubrovin-Puninski} is left coherent, so the Grothendieck category $\mathcal{G}_K$ is  locally coherent (see \cite[Theorem~2.16]{H97}). Hence, such a $\mathcal{G}_K$ also gives a negative answer for two of the three questions asked by Djament. 
\end{enumerate}
\end{rem}

\section{Partial affirmative answers to the Cuadra--Simson Problem}%
Although the Cuadra--Simson Problem has a negative answer in general, we shall see in this section that it may still be answered in the affirmative when we restrict it to suitable classes of Grothendieck categories, including the case of comodule categories, which were the primary focus of Cuadra and Simson.

\subsection{Bilocalizing subcategories of modules  over commutative rings}%
Our next result gives an affirmative answer to the (reformulated) Cuadra--Simson Problem~\ref{probl.CS reformulated}\,(1) in the commutative case. Recall that if $R$ is a commutative ring and $M$ is an $R$-module, then the \emph{support of $M$}, denoted by $\Supp(M)$, is the set of all primes $\mathbf{p}\in\Spec(R)$, such that $M_\mathbf{p}\neq 0$.

\begin{thm}
\label{teor.CS-commutative case}%
Let $R$ be a commutative ring, and let $I$ be an idempotent ideal of $R$ such that the associated Giraud subcategory $\mathcal{G}_I$ of $\rMod R$ is locally finitely presented. Then, $I_\mathbf{p}=R_\mathbf{p}$, for all $\mathbf{p}\in\Supp(I)$. In particular, $I$ is the trace of a projective $R$-module, and so $\mathcal{G}_I$ has enough projectives. 
\end{thm}
\begin{proof}
If $I_\mathbf{p}=R_\mathbf{p}$, for all $\mathbf{p}\in\Supp(I)$ then, for each $\mathbf{p}\in\Spec(R)$, one has that either $(R/I)_\mathbf{p}=R_\mathbf{p}$ or $(R/I)_\mathbf{p}=0$. So $(R/I)_\mathbf{p}$ is a projective $R_\mathbf{p}$-module, for all $\mathbf{p}\in\Spec(R)$, which implies that $R/I$ is flat as $R$-module (see \cite[Proposition~3.10]{AM}). This is equivalent to say that $I$ is pure in $R$. But then $I$ is the trace of a projective $R$-module (see \cite[Proposition 1.1]{JT} or  \cite[Corollary~2.13]{HP14}).

It remains to verify that $I_\mathbf{p}=R_\mathbf{p}$, for all $\mathbf{p}\in\Supp(I)$. Suppose, looking for a contradiction, that there is some $\mathbf{p}\in\Spec(R)$ such that $0\neq I_\mathbf{p}\subsetneq R_\mathbf{p}$. For all $X\in\mathcal{X}_I\cap\rmod R$ we have that $X=XI$, and so $X_\mathbf{p}=X_\mathbf{p}I_\mathbf{p}$. But $X_\mathbf{p}$ is a finitely presented $R_\mathbf{p}$-module while $I_\mathbf{p}\subseteq\mathbf{p}R_\mathbf{p}=J(R_\mathbf{p})$. Nakayama's Lemma then forces $X_\mathbf{p}=0$, for all $X\in \mathcal{X}_I\cap\rmod R$. On the other hand, by Theorem~\ref{thm.always enough flats}, we have that $\mathcal{C}_I=\Gen(\mathcal{X}_I\cap\rmod R)$, and so there exists  an epimorphism $p\colon \coprod_{\Lambda}X_\lambda\twoheadrightarrow I$ in $\rMod R$, with $X_\lambda\in \mathcal{X}_I\cap\rmod R$, for all $\lambda\in\Lambda$. Applying the localization functor $(-)_\mathbf{p}\cong-\otimes_RR_\mathbf{p}\colon\rMod R\to\rMod R_\mathbf{p}$ we get an epimorphism $0=\coprod_{\Lambda}(X_\lambda)_\mathbf{p}\cong (\coprod_{\Lambda}X_\lambda)_\mathbf{p}\twoheadrightarrow I_\mathbf{p}\neq 0$, a contradiction.
\end{proof}

\subsection{When the endomorphism rings are semiregular}%

The class of those locally finitely presented Grothendieck categories $\G$ which are {\em regular (in the sense of von Neumann)} was defined in \cite[Section~5]{St2}. In fact,  Stenstr\"om characterized the regularity of $\G$ via eight equivalent conditions, see \cite[Theorem~4]{St2}, labeled with letters going from (a) to (f\,$'$). Note that condition~(c), i.e., the equality $\fp(\G)=\proj(\G)$, implies that $\G$ is generated by $\proj(\G)$, and this occurs only in module categories. Hence, it is not restrictive to assume $\G= \rMod\A$, for a small preadditive category $\A$ (e.g., take $\A\leqdef\proj(\G))$. 

In what follows, we will say that $\A$ is  \emph{von Neumann regular (vNR)} if $\rMod\A$ is regular in the sense of \cite{St2}. The equivalence between Stenstr\"om's conditions (a), (e), and (f), implies that $\A$ is  vNR if and only if it satisfies one of the following equivalent conditions:

%
\begin{itemize}
\item for each $\alpha\colon \what b\rightarrow \what a$ in $\what{\A}$, there is $\beta\colon \what a\rightarrow \what b$ such that $\alpha=\alpha\circ\beta\circ\alpha$;
\item for each $\what a\in\Obj(\what \A)$, the ring $\what \A(\what a,\what a)$ is von Neumann regular;
\item for each $a\in\Obj(\mathcal{A})$, the ring $\A(a,a)$ is von Neumann regular.
\end{itemize}
Adapting the usual ring-theoretical terminology, we may call a morphism $\alpha$ in $\what\A$ {\em regular} if it has a {\em quasi-inverse} $\beta$, i.e., the following equality holds: $\alpha=\alpha\circ\beta\circ\alpha$. With these conventions, $\A$ is vNR precisely when each $\alpha$ in $\what\A$ is regular.


Recall from \cite{N76} that an element $x$ in a ring $R$ is called \emph{semiregular} if there exists $y\in R$ such that $y=yxy$, and $x - xyx\in J(R)$. We say that $R$ is
\emph{semiregular} when each element $x\in R$ is semiregular or, equivalently, when $R/J(R)$ is vNR, and idempotents lift modulo $J(R)$ (see \cite[Theorem~2.9]{N76}); we now extend this idea to small preadditive categories.

\begin{defn}
A small preadditive category $\A$ is called \emph{semiregular} when $\what\A/\rad_{\what\A}$ is vNR, and idempotents in $\what\A/\rad_{\what\A}$ lift modulo the radical.
\end{defn}
Observe that, as a direct consequence of the definitions, $\A$ is semiregular if, and only if, the ring $\what\A(\what a,\what a)$ is semiregular for all $\what a\in \Obj(\what\A)$.

In what follows we will use the notation $\rad_\A(-,\what a)$ for the right $\A$-submodule $\rad_{\what\A}(-,\what a)_{|\A}\leq H_{\what a}$.

\begin{lem}\label{prop.semiregular categories}
Let $\A$ be a semiregular small preadditive category. Then, for each $\alpha\in\what\A(\what b,\what a)$ there is an idempotent $e=e^2\in \what\A(\what a,\what a)$, such that:
\begin{equation*}
\tag{$\dagger$} e\A\subseteq\alpha\A\qquad\text{and}\qquad \alpha\A\subseteq e\A+\rad_{\A}(-,\what a).
\end{equation*}
As a consequence: $\A e\A\subseteq\A \alpha\A $ and $\A \alpha\A \subseteq\A e\A +\rad_{\A}$.
\end{lem}

\begin{proof}
Let $\wbar{\A}\leqdef \what{\A}/\rad_{\what{\A}}$\, and denote by $\pi \colon \what{\A} \to \wbar{\A}$ the projection functor. Given $\alpha \in \what{\A}(\what b,\what a)$, the map $\bar{\alpha}\leqdef \pi(\alpha)$ is regular in $\wbar{\A}$, so there is a quasi-inverse $\bar \beta\leqdef\pi(\beta)$ of $\bar \alpha$, for some $\beta \in \what{\A}(\what a,\what b)$; let  also  $\bar{e}\leqdef \pi(\alpha \circ \beta)=\bar\alpha \circ\bar \beta\in \wbar{\A}(\what a,\what a)$. The following equalities hold by construction:
\begin{equation*}
\bar{\alpha} = \bar{\alpha} \circ \bar\beta \circ \bar{\alpha}=\bar e\circ \bar \alpha\in\wbar{\A}(\what b,\what a)\qquad\text{and} \qquad \bar e=\bar e^2\in \wbar{\A}(\what a,\what a).
\end{equation*}
Since $\A$ is semiregular, idempotents lift modulo the radical. Furthermore, as we have observed, the ring $\what\A(\what a,\what a)$ is semiregular and so, by \cite[Proposition~2.2]{N76}, there exists an idempotent $e = e^2 \in (\alpha \circ \beta)\what{\A}(\what a,\what a)$ such that $\pi(e) = \bar{e}$.

To prove the first inclusion in ($\dagger$), we simply observe that $e \in (\alpha \circ \beta)\what{\A}(\what a,\what a)$ implies $e = \alpha \circ \beta \circ \delta$ for some $\delta \in \what{\A}(\what a,\what a)$. Thus, $e{\A} \subseteq \alpha{\A}$ (by definition).
As for the second inclusion in ($\dagger$), we apply the regularity of $\bar{\alpha}$ in the quotient and, in particular, the following equality: $\bar{\alpha} = \bar{e} \circ \bar{\alpha}$. Lifting back to $\what{\A}$, we have that $\alpha - e \circ \alpha \in \rad_{\what{\A}}(\what b,\what a)$, which implies $\alpha{\A}\subseteq e{\A}+\rad_{\A}(-,\what a)$, as desired.
\end{proof}

\begin{thm}\label{teor.conjecture true for semiregular-quotients}%
Let $\mathcal{A}$ be a semiregular small preadditive category, let $\calI$ be an idempotent ideal of $\mathcal{A}$, and suppose that the associated category $\mathcal{G}_{\calI}$ is locally finitely presented. Then, $\mathcal{G}_{\calI}$ is generated by its finitely generated projectives.
\end{thm}
\begin{proof}
 Suppose  $\calI\neq 0$; then $\mathcal{X}_\calI\cap\rmod\A\neq 0$, since $\Gen(\mathcal{X}_\calI\cap\rmod \A)=\mathcal{C}_\calI\neq 0$. By Nakayama's Lemma, we deduce that $\calI\not\subseteq\rad_\A$. Moreover, given a morphism $\alpha\in\what\calI(\what b,\what a)\setminus\rad_{\what{\A}}(\what b,\what a)$, we obtain  by Proposition~\ref{prop.semiregular categories}, that: ${\A}e{\A}\subseteq{\A}\alpha{\A}$ and ${\A}\alpha{\A}\subseteq{\A}e{\A}+\rad_{{\A}}$.

Define $\mathcal{E}\leqdef \smash[b]{\bigcup_{\what a\in\Obj(\what{\mathcal{A}})}\mathcal{E}(\what a,\what a)}$, where $\mathcal{E}(\what a,\what a)$ is the set of those idempotents in $\what{\mathcal{A}}(\what a,\what a)$ that belong to $\what{\calI}(\what a,\what a)$, for all $\what a\in\Obj(\what{\mathcal{A}})$. The ideal $\calI'=\sum_{e\in\mathcal{E}}\mathcal{A}e\mathcal{A}$ is the trace in $\mathcal{A}$ of the set $\mathcal{P}=\{e\mathcal{A}\mid e\in\mathcal{E}\}\subseteq \proj\text{-}\A$ (see Proposition~\ref{prop.traces-versus-idempotents}), and it is then an idempotent ideal of $\mathcal{A}$ such that $\calI'\subseteq\calI$. Our goal is to verify the converse inclusion, after which we will have concluded, by  Theorem~\ref{thm.always enough flats}. We have already observed that $\calI'+\rad_\mathcal{A}=\calI+\rad_\mathcal{A}$. We put now
\[
	\wbar{\calI}{}'\leqdef (\calI'+\rad_\mathcal{A})/{\rad_\mathcal{A}}
	\qquad\hbox{and}\qquad
	\wbar{\calI}\leqdef (\calI+\rad_\mathcal{A})/{\rad_\mathcal{A}},
\]
so that $\wbar{\calI}{}'=\wbar{\calI}$, as ideals of the preadditive category $\wbar{\mathcal{A}}=\mathcal{A}/{\rad_\mathcal{A}}$.

Let now $X\in\rmod\mathcal{A}$ and consider the right $\wbar{\mathcal{A}}$-module $\wbar{X}\leqdef X/{X\rad_\mathcal{A}}$. If we also have that $X\in\mathcal{C}_I$ (e.g., if $X\in\mathcal{X}_I\cap\rmod\mathcal{A}$), then $X=X\calI$ from which we readily get that $\wbar{X}=\wbar{X}\wbar{\calI}=\wbar{X}\wbar{\calI'}$. It then follows that $X=X\calI'+X\rad_\mathcal{A}$, which implies that $X=X\calI'$ by Proposition~\ref{prop.NAK}. Thus, $\mathcal{C}_I=\Gen(\mathcal{X}_I\cap\rmod\mathcal{A})\subseteq\mathcal{C}_{I'}$ and this inclusion is an equality since $\calI'\subseteq\calI$. Hence, the TTF triples in $\rMod\mathcal{A}$ associated to $\calI$ and $\calI'$ coincide, which implies that $\calI=\calI'$ by Proposition~\ref{prop.TTF triple associated to idemp.ideal}. 
\end{proof}

\begin{cor}
\label{cor.semiregular ring}%
Let $R$ be a semiregular ring (e.g., a semiperfect or a left or right artinian ring), and let $I$ be an idempotent ideal of $R$. The associated Giraud subcategory $\mathcal{G}_I$ of $\rMod R$ is locally finitely presented if, and only if, $I$ is generated by idempotent elements of $R$. 
\end{cor}
\begin{proof}
Note that if $\mathcal{A}=R$ is viewed as a preadditive category with just one object, then $\what{\mathcal{A}}$ is equivalent to $\mathrm{free}\mathchar`-R$, the subcategory of finitely generated free $R$-modules. Since the ring $\mathcal{M}_{n\times n}(R)$ is semiregular, for all $n>0$ (see \cite[Proposition~2.7]{N76}), Theorem~\ref{teor.conjecture true for semiregular-quotients} tells us that $\mathcal{G}_I$ is locally finitely presented if, and only if, $I$ is the trace of a set of finitely generated projective $R$-modules. Use then Theorem~\ref{teor.conjecture true for semiregular-quotients} and \cite[Theorems~1 and~3]{War} (see \cite[Theorem~2.9]{N76}).
\end{proof}

Recall that an idempotent-complete skeletally small additive category is \emph{Krull--Schmidt} if each object is a finite coproduct of endolocal objects. Equivalently, each object has a semiperfect endomorphism ring (see \cite[Theorem~A.1]{Chen-Ye-Zhang}).

\begin{cor}
\label{cor.Krull-Schmidt case}%
Let $\mathcal{A}$ be a Krull--Schmidt additive category and let $\calI=\calI^2$ be an ideal of $\mathcal{A}$. If the associated Giraud subcategory $\mathcal{G}_I$ of $\rMod\mathcal{A}$ is locally finitely presented, then it has a set of finitely generated projective generators. 
\end{cor}
\begin{proof}
The result is a direct consequence of last theorem since any semiperfect ring is semiregular. 
\end{proof}

\subsection{Locally finite categories and comodule categories}%

Recall that a Grothendieck category $\mathcal{G}$ is said to be \emph{locally finite} when each object is the direct union of its subobjects of finite length and the subcategory $\fl(\mathcal{G})$ of finite length objects is skeletally small. Note that $\fl(\mathcal{G})$ is a Krull--Schmidt category since it is idempotent-complete and, by Fitting's Lemma, the endomorphism ring of any finite length indecomposable object is local. On the other hand, such a $\mathcal{G}$ is locally noetherian (i.e., each of its objects is a direct union of its noetherian subobjects, and the noetherian subobjects form a skeletally small subcategory $\noeth(\mathcal{G})$). Then, $\mathcal{G}$ is locally finitely presented (see \cite[Proposition~2]{R69} and \cite{GU}), and so $\fp(\mathcal{G})=\noeth(\mathcal{G})=\fl(\mathcal{G})$.

The following result gives an affirmative answer to the Cuadra--Simson Problem for a class of Grothendieck categories that contains the locally finite ones.

\begin{thm}
\label{teor.locally finite}%
Let $\mathcal{G}$ be a locally finitely presented Grothendieck category such that $\fp \G$ is semiregular (e.g., $\mathcal{G}$ locally finite or, more generally, such that  $\fp(\mathcal{G})$ is Krull--Schmidt). Then, the following assertions are equivalent:
\begin{enumerate}[(a)]
\item $\mathcal{G}$ has enough flat objects.
\item $\mathcal{G}$ is $\mathrm{AB}\mathchar`-4^\ast$.
\item $\mathcal{G}$ has a set of finitely generated projective generators.
\end{enumerate}
\end{thm}
\begin{proof}
\smallskip \noindent``$\textup{(a)}\Leftrightarrow\textup{(b)}$'' It follows from Corollary~\ref{cor.Cuadra-Simson versus Djament}.

\smallskip \noindent``$\textup{(c)}\Rightarrow\textup{(a)}$'' It is clear.

\smallskip \noindent``$\textup{(b)}\Rightarrow\textup{(c)}$'' Let $\mathcal{A}$ be a skeleton of $\fp(\mathcal{G})$, let $\mathbf{y}\colon\mathcal{G}\to\rMod\mathcal{A}$ be the Restricted Yoneda Embedding (see Theorem \ref{teor.Gabriel-Popescu-Mitchell}), and let $q\colon\rMod\mathcal{A}\to\mathcal{G}$ be its (exact) left adjoint. We claim that $\mathcal{T}=\ker(q)$ is closed under products, and so  it is a TTF class in $\rMod\mathcal{A}$. This would imply that $\mathcal{T}=\mathcal{T}_{\calI}$, for a uniquely determined idempotent ideal $\calI$ of $\mathcal{A}$, and the implication would then follow immediately from Theorem~\ref{teor.conjecture true for semiregular-quotients}.

In order to prove our claim, consider a family $(T_\lambda )_{\lambda\in\Lambda}$ of objects in $\mathcal{T}$. Using Proposition~\ref{lem.Mitchell}, we can fix a family of epimorphisms $(f_\lambda^{\vphantom\prime}\colon Y_\lambda^{\vphantom\prime}\twoheadrightarrow Y'_\lambda)_{\lambda\in\Lambda}^{\vphantom\prime}$ in $\mathcal{G}$ such that $T_\lambda\cong\coker(\mathbf{y}(f_\lambda))$, for all $\lambda\in\Lambda$. Since $\rMod\mathcal{A}$ is $\mathrm{AB}\mathchar`-4^\ast$ and $\mathbf{y}$ preserves products, we have an induced exact sequence in $\rMod\mathcal{A}$,
\[
	\mathbf{y}\Bigl(\prod_{\lambda\in\Lambda}Y_\lambda\Bigr)
		\buildrel\mathbf{y}(\prod f_\lambda)\over\longrightarrow
		\mathbf{y}\Bigl(\prod_{\lambda\in\Lambda}Y_\lambda\Bigr)\longrightarrow
		\prod_{\lambda\in\Lambda}T_\lambda\longrightarrow 0.
\]
Applying the exact functor $q$ to this sequence, and using the natural isomorphism $q\circ\mathbf{y}\cong 1_{\mathcal{G}}$, we deduce that $q(\prod_{\Lambda}T_\lambda)$ is isomorphic to the cokernel in $\mathcal{G}$ of the morphism $\prod f_\lambda^{\vphantom\prime}\colon\prod_{\Lambda}Y_\lambda^{\vphantom\prime}\to\prod_{\Lambda}Y'_\lambda$. The latter is an epimorphism in $\mathcal{G}$, since it satisfies the $\mathrm{AB}\mathchar`-4^\ast$ condition. Thus, $q(\prod_{\Lambda}T_\lambda)=0$, and so  $\prod_{\Lambda}T_\lambda\in\mathcal{T}$.
\end{proof}

The initial motivation that led Cuadra and Simson to formulate their problem is to be found in their work in categories of comodules over  $K$-coalgebras. The reader is referred to \cite{Sweedler} and \cite{DNR} for the basics of coalgebras and their comodule categories. 

In the rest of this subsection, we assume $K$ to be a field, and $C$ a $K$-coalgebra. It is well-known that any (right) $C$-comodule is the direct union of its finite dimensional subcomodules (see \cite[Theorem~2.1.3]{Sweedler}). This immediately implies that, if $\rComod C$ is the category of $C$-comodules, which is well-known to be a Grothendieck category (see \cite[Corollary~2.2.8]{DNR}), and $\mathcal{C}_0$ is its skeletally small subcategory of finite dimensional $C$-comodules, then $\mathcal{C}_0=\fl(\rComod C)$, and $\rComod C$ is locally finite. As an immediate consequence of Theorem~\ref{teor.locally finite}, we get an affirmative answer to the Cuadra--Simson Problem for comodule categories.

\begin{cor}
\label{thm.Cuadra-Simson for comodules}%
Let $C$ be a $K$-coalgebra. Then, the following are equivalent:
\begin{enumerate}[(a)]
\item $\rComod C$ has enough flat objects;
\item $\rComod C$ is $\mathrm{AB}\mathchar`-4^\ast$;
\item $\rComod C$ has a set of finitely generated projective generators.
\end{enumerate}
\end{cor}

\begin{rem}
The equivalence ``$\textup{(b)}\Leftrightarrow\textup{(c)}$'' in Theorem~\ref{teor.locally finite} (as well as the corresponding equivalence in Corollary~\ref{thm.Cuadra-Simson for comodules}) was already known for locally finite Grothendieck categories (see \cite[Theorem~3.1 and Corollary~3.2]{CIENT04}).
\end{rem}

\section{Local finite presentability of the associated Giraud subcategories}%

This section may be seen as a first approximation to \cite[Problem~2.9.a]{CS07}. Using the results of Section~\ref{sec.Cuadra-Simson via colocal.}, the problem can be seen to be equivalent to finding conditions on a small preadditive category $\mathcal{A}$ and an idempotent ideal $\calI$ in $\A$,  for $\mathcal{X}_{\calI}$ to be locally finitely presented. We interpret any $X\in\rmod\mathcal{A}$ as the cokernel of a map $H_\alpha$, where $\alpha\in\what{\mathcal{A}}(\what{a},\what{b})$, for some $\what{a},\,\what{b}\in \what{\mathcal{A}}$.
The key point is to identify the objects in $\mathcal{X}_I\cap\rmod\mathcal{A}$. This is tantamount to a characterization of the morphisms $\alpha$ in $\what{\A}$ such that $\coker(H_\alpha)\in\mathcal{X}_I\cap\rmod\mathcal{A}$: this is the primary goal of this section. The reader is referred to Subsection~\ref{ss:modules-morphisms} for the definition of some $\mathcal{A}$-modules associated to morphisms in $\what{\mathcal{A}}$.

\begin{defn}\label{def.ideal:morphism}%
Given a morphism $\alpha\in\what{\mathcal{A}}(\what{a},\what{b})$, we denote by $(\what{\calI}(-,\what{b}):\alpha)$ the $\mathcal{A}$-submodule of $H_{\what{a}}$ such that, for each $c\in\Obj(\mathcal{A})$:
\[
(\what{\calI}(-,\what{b}):\alpha)(c)\leqdef\{\gamma\in\what{\mathcal{A}}(c,\what{a})=H_{\what{a}}(c)\mid  \alpha\circ\gamma\in\what{\calI}(c,\what{
b})\}.
\]
\end{defn}

\begin{rem}\label{rem.I versus I:alpha}%
We leave to the reader the easy verification of the fact that $(\what{\calI}(-,\what{b}):\alpha)$ is indeed an $\mathcal{A}$-submodule of $H_{\what{a}}$\,, and that it contains both $\what{\calI}(-,\what{a})$, and $\rann_\mathcal{A}(\alpha)$. In particular, there are the following inclusions in $\rMod\A$: 
\[
\what{\calI}(-,\what{a})+\rann_\mathcal{A}(\alpha )\leq(\what{\calI}(-,\what{b}):\alpha)\leq H_{\what{a}}.
\]
\end{rem}

\subsection{fp-Detecting morphisms}%

In the following definition we introduce the class of morphisms that we need to characterize in order to describe $\cal X_{\calI}\cap \rmod\A$.

\begin{defn}
\label{def.fp-detecting morphisms}%
A morphism $\alpha\colon\what{a}\to\what{b}$ in $\what{\mathcal{A}}$ is said to be \emph{fp-detecting} (with respect to $I$), when $\coker(H_\alpha)\in \mathcal{X}_I\cap\rmod\mathcal{A}$.
\end{defn}

Our following result identifies a wider class.

\begin{lem}\label{lem.fpmodules-in-C}%
Let $\what{a},\what{b}\in\Obj(\what{\mathcal{A}})$, $\alpha\in\what{\mathcal{A}}(\what{a},\what{b})$, and let  $X_\alpha\leqdef\coker(H_\alpha)$. Then, the following assertions are equivalent:
\begin{enumerate}[(a)]
\item $X_\alpha\in\mathcal{C}_I$;
\item There exists a morphism $\beta\in\what{\mathcal{A}}(\what{b},\what{a})$ such that $1_{\smash[t]{\what{b}}}-\alpha\circ\beta\in\what{\calI}(\what{b},\what{b})$, i.e., $\alpha$ is mapped onto a retraction by the quotient functor $\what{\mathcal{A}}\twoheadrightarrow\what{\mathcal{A}}/\what{\calI}$.
\end{enumerate}
\end{lem} 
\begin{proof}
Letting $X\leqdef X_\alpha=H_{\what{b}}/{\im(H_{\alpha})}$, we have that $X=X\calI$ if, and only if,
\begin{equation}\label{fpmodules-in-C_eq}
	H_{\what{b}}=H_{\what{b}}\calI +\im(H_{\alpha}).
\end{equation}
But $\smash[b]{H_{\what{b}}\calI} =\what{\calI}(-,\what{b})$ and, by Lemma~\ref{lem.proper submodule}, the equality \eqref{fpmodules-in-C_eq} holds if, and only if, $1_{\what{b}}\in\what{\calI}(\what{b},\what{b})+(\im(H_{\alpha}))(\what{b})$, which is equivalent to assertion~(b).
\end{proof}

The above lemma identifies the morphisms $\alpha$ in $\what{\mathcal{A}}$ such that the canonical morphism $\rho_{X_\alpha}\colon X_\alpha\otimes_\mathcal{A}\calI\to X_\alpha$, induced by multiplication, is an epimorphism. The following lemma identifies those $\alpha$ for which $\rho_{X_\alpha}$ is a monomorphism.

\begin{lem}
\label{lem.I-Mittag-Leffler}%
In the same setting of Lemma~\ref{lem.fpmodules-in-C}, the following assertions are equivalent:
\begin{enumerate}[(a)]
\item The canonical morphism $\rho_{X_\alpha}\colon X_\alpha\otimes_\mathcal{A}\calI\to X_\alpha$ is a monomorphism.
\item $(\what{\calI}(-,\what{b}):\alpha)=\what{\calI}(-,\what{a})+\rann_\mathcal{A}(\alpha)$.
\end{enumerate}
\end{lem}
\begin{proof}
Let $X\leqdef X_\alpha$, and consider the following commutative diagram $\rMod\mathcal{A}$:
\[
	H_{\what{a}}\otimes_\mathcal{A}\calI\cong\!%
\xymatrix@R=20pt@C=38pt{%
	\what{\calI}(-,\what{a}) \ar[r]^-{H_\alpha\otimes 1_I}\ar@{^{`}->}[d]_-{\iota_{\what{a}}} &
		\what{\calI}(-,\what{b})=H_{\what{b}}\otimes_\mathcal{A}\calI
			\ar[r]\ar@{^{`}->}[d]^-{\iota_{\what{b}}} &
		X\otimes_\mathcal{A}\calI \ar[d]^-{\rho_X} \ar[r] & 0\  \cr
	H_{\what{a}} \ar[r]^-{H_\alpha} & H_{\what{b}} \ar[r] & X \ar[r] & 0.
}
\]
By the exactness of the rows in the above diagram, one sees that $\rho_X$ is isomorphic to:
\[
	\psi\colon\frac{\what{\calI}(-,\what{b})}{H_\alpha (\what{\calI}(-,\what{a}))}\longrightarrow\frac{H_{\what{b}}}{\im(H_\alpha)},
\]
where $\psi$ is induced by the inclusion $\iota_{\what{b}}\colon\what{\calI}(-,\what{b})\hookrightarrow H_{\what{b}}$\,. It follows that $\rho_X$ is a monomorphism if, and only if, so is $\psi$. On the other hand,
\[
	\ker(\psi)=\frac{\what{\calI}(-,\what{b})\cap\im(H_\alpha)}{H_\alpha (\what{\calI}(-,\what{a}))}.
\]
Therefore, $\psi$ is a monomorphism if, and only if,
\begin{equation}
	\what{\calI}(-,\what{b})\cap\im(H_\alpha)\subseteq H_\alpha (\what{\calI}(-,\what{a})).
	\tag{$\diamond\diamond$}
	\label{eq:5}
\end{equation}
Evaluating at a given $c\in\Obj(\mathcal{A})$, we have that $[\what{\calI}(-,\what{b})\cap\im(H_\alpha)](c)$ consists of those $\gamma\in\what{\calI}(c,\what{b})$ such that $\gamma =\alpha\circ\beta$, for some $\beta\in\what{\mathcal{A}}(c,\what{a})$, while $[H_\alpha(\what{\calI}(-,\what{a}))](c)$ consists of the morphism $\gamma\in\what{\mathcal{A}}(c,\what{b})$ such that $\gamma =\alpha\circ\eta$, for some $\eta\in\what{\calI}(c,\what{a})$. Therefore, \eqref{eq:5} holds if, and only if, for each $\beta\in (\what{\calI}(-,\what{b}):\alpha)(c)$, there exists $\eta\in\what{\calI}(c,\what{a})$ such that $\alpha\circ\beta =\alpha\circ\eta$ or, equivalently, $\beta-\eta\in\rann_\mathcal{A}(\alpha)(c)$. Hence, $\psi$ is a monomorphism if, and only if, $(\what{\calI}(-,\what{b}):\alpha)\subseteq\what{\calI}(-,\what{a})+\rann_\mathcal{A}(\alpha)$ or, equivalently, $(\what{\calI}(-,\what{b}):\alpha)=\what{\calI}(-,\what{a})+\rann_\mathcal{A}(\alpha)$ (see Remark~\ref{rem.I versus I:alpha}). 
\end{proof}

It is now straightforward to deduce from the last two lemmas the desired characterization of the fp-detecting morphisms in $\what{\mathcal{A}}$:

\begin{prop}
\label{prop.fp(X)}%
Let $\alpha\colon\what{a}\to\what{b}$ be a morphism in $\what{\mathcal{A}}$, and let $p\colon\what{\mathcal{A}}\to\what{\mathcal{A}}/\what{\calI}$ be the projection. Then, the following assertions are equivalent:
\begin{enumerate}[(a)]
\item $\alpha$ is fp-detecting, i.e., $X_\alpha\leqdef\coker(H_\alpha)\in\mathcal{X}_I\cap\rmod\mathcal{A}$;
\item The following two conditions hold:
\begin{enumerate}[(i)]
\item the image $p(\alpha)$ of $\alpha$ in $\what{\mathcal{A}}/\what{\calI}$,  is a retraction.
\item $(\what{\calI}(-,\what{b}):\alpha)=\what{\calI}(-,\what{a})+\rann_\mathcal{A}(\alpha)$.
\end{enumerate}
\end{enumerate}
\end{prop}

\begin{exmpl}
\label{ex.stableunit-fpdetecting}%
Let $\what{b}\in\Obj(\what{\mathcal{A}})$ and $\eta\in\what{\calI}(\what{b},\what{b})$. Then $1_{\what{b}}-\eta$ is fp-detecting.

Indeed, it is clear that  $\alpha\leqdef 1_{\what{b}}-\eta$ satisfies condition (b.i) of Proposition~\ref{prop.fp(X)}. Let us verify also condition~(b.ii), which is equivalent to the following inclusion (see Remark~\ref{rem.I versus I:alpha}):
\begin{equation}
	(\what{\calI}(-,\what{b}):\alpha )\subseteq \what{\calI}(-,\what{b})+\rann_\mathcal{A}(\alpha ).
	\tag{\dag}\label{eq:3}
\end{equation}
Given $c\in\Obj(\mathcal{A})$, by Definition~\ref{def.ideal:morphism}, $[(\what{\calI}(-,\what{b}):\alpha )](c)$ consists of those morphisms $\gamma\in\what{\mathcal{A}}(c,\what{b})$ such that $\alpha\circ\gamma\in\what{\calI}(c,\what{b})$ or, equivalently, $\gamma-\eta\circ\gamma\in\what{\calI}(c,\what{b})$ which is equivalent to saying that $\gamma\in\what{\calI}(c,\what{b})$, because $\eta\in\what{\calI}(\what{b},\what{b})$. Hence, we have deduced that $(\what{\calI}(-,\what{b}):\alpha )=\what{\calI}(-,\what{b})$, which clearly implies \eqref{eq:3}. 
\end{exmpl}

\subsection{The main theorem}%

We can now prove the main result of this section.

\begin{thm}
\label{thm.finitepresentability-on-ideal}%
Let $\mathcal{A}$ be small preadditive category, let $\calI$ be an idempotent ideal of $\mathcal{A}$, and denote by $\ubar{\mathcal{V}}$ the set of fp-detecting morphisms in $\what{\mathcal{A}}$ (see Definition~\ref{def.fp-detecting morphisms} and Proposition~\ref{prop.fp(X)}). Then, the following assertions are equivalent:
\begin{enumerate}[(a)]
\item The associated Giraud subcategory $\mathcal{G}_I$ is locally finitely presented;
\item $\calI =\sum_{\alpha\in\ubar{\mathcal{V}}}\lann_\mathcal{A}(\alpha)\mathcal{A}$;
\item $\calI=\sum_{\what{b}\in\Obj(\what{\mathcal{A}})}\sum_{\eta\in\what{\calI}(\what{b},\what{b})}\Fix(\eta)\mathcal{A}$, where $\Fix(\eta)\leq H'_{\what{b}}$ is the $\mathcal{A}$-submodule such that $\Fix(\eta) (c)\leqdef\{\beta\in\what{\mathcal{A}}(\what{b},c)\mid\beta =\beta\circ\eta\}$, for all $c\in \Obj(\A)$.
\end{enumerate} 
\end{thm}
\begin{proof}
Recall that $\mathcal{G}_I$ and $\mathcal{X}_I$ are equivalent Grothendieck categories. 

\smallskip \noindent``$\textup{(a)}\Leftrightarrow\textup{(b)}$'' By Proposition~\ref{prop.fp(X)}, we know that the objects of $\mathcal{X}_I\cap\rmod\mathcal{A}$, i.e., the finitely presented objects of $\mathcal{X}_I$, are those of the form $X_\alpha\leqdef\coker(H_\alpha)$, for some $\alpha\in\ubar{\mathcal{V}}$. A morphism $f\colon X_\alpha\to H_d$ in $\rMod\mathcal{A}$ is then identified by a morphism $g\colon H_{\what{b}}\to H_d$ in this same category such that $g\circ H_\alpha =0$. By the Yoneda Lemma, given such $g$, there is unique morphism $\beta\in\what{\mathcal{A}}(\what{b},d)$ such that $g=H_\beta$, and $\beta\circ\alpha =0$. It is also clear that $\im(f)=\im(g)=\im(H_\beta)$ where, for each $c\in\Obj(\mathcal{A})$, $\im(H_\beta)(c)$ consists of those morphisms $\xi\in\what{\mathcal{A}}(c,d)=\mathcal{A}(c,d)$ such that $\xi =\beta\circ\gamma$, for some $\gamma\in\what{\mathcal{A}}(c,\what{b})$.

Recall that the ideal $\tr_{X_\alpha}(\mathcal{A})$ of $\mathcal{A}$ is the $\mathcal{A}\mathchar`-\mathcal{A}$-bimodule associated to the additive functor $\mathcal{A}\to\rMod\mathcal{A}$ that takes $d\mapsto\tr_{X_\alpha}(H_d)$. Using the discussion in the previous paragraph, one sees that $\tr_{X_\alpha}(\mathcal{A})(c,d)$ consists of the morphisms $\xi\in\mathcal{A}(c,d)$ that can be expressed as finite sums of the form  $\xi=\sum_{k=1}^t\beta_k\circ\gamma_k$, with $\gamma_k\in\what{\mathcal{A}}(c,\what{b})$, $\beta_k\in\what{\mathcal{A}}(\what{b},d)$, and $\beta_k\circ\alpha =0$, for all $k=1,\ldots,t$. By Corollary~\ref{cor.idealgenerated-annihilator}, we conclude that $\tr_{X_\alpha}(\mathcal{A})=\lann_\mathcal{A}(\alpha)\mathcal{A}$ for all $\alpha\in\ubar{\mathcal{V}}$.

Let now $\mathcal{X}_0\leqdef \{X_\alpha\mid \alpha\in\ubar{\mathcal{V}}\}$, which is a small subcategory equivalent to $\mathcal{X}_I\cap\rmod\mathcal{A}$. Lemma~\ref{lema.generation of C versus of X} and Theorem~\ref{thm.always enough flats} tell us that $\mathcal{X}_I$ is locally finitely presented if, and only if, $\calI =\tr_{\mathcal{X}_0}(\mathcal{A})=\sum_{\alpha\in\ubar{\mathcal{V}}}\tr_{X_\alpha }(\mathcal{A})$. By the argument in the previous paragraph, this is equivalent to the condition $\calI =\sum_{\alpha\in\ubar{\mathcal{V}}}\lann_\mathcal{A}(\alpha)\mathcal{A}$.

\smallskip\noindent
``$\textup{(b)}\Leftrightarrow\textup{(c)}$'' Observe that $\Fix(\eta)=\lann_\mathcal{A}(1_{\what{b}}-\eta)$, for all $\what{b}\in\Obj(\what{\mathcal{A}})$, and $\eta\in\what{\calI}(\what{b},\what{b})$. Putting $\calI'\leqdef\sum_{\what{b}\in\Obj(\what{\mathcal{A}})}\sum_{\eta\in\what{\calI}(\what{b},\what{b})}\Fix(\eta)\mathcal{A}$, Example~\ref{ex.stableunit-fpdetecting} tells us that $I'\subseteq \sum_{\alpha\in\ubar{\mathcal{V}}}\lann_\mathcal{A}(\alpha)\mathcal{A}$. 

On the other hand, if $\alpha\in\what{\mathcal{A}}(\what{a},\what{b})$ is fp-detecting, then $\ubar{\alpha}\leqdef p(\alpha )$ is a retraction in $\what{\mathcal{A}}/\what{\calI}$ (see Proposition~\ref{prop.fp(X)}). Choose  $\beta\in\what{\mathcal{A}}(\what{b},\what{a})$ such that $\ubar{\alpha}\circ\ubar{\beta}=\underline{1}_{\what{b}}\,$. Then we have $\alpha\circ\beta=1_{\what{b}}-\eta$, where $\eta\in\what{\calI}(\what{b},\what{b})$, so that $\lann_\mathcal{A}(\alpha\circ\beta)=\Fix(\eta)$. Since we clearly have an inclusion $\lann_\mathcal{A}(\alpha)\subseteq\lann_\mathcal{A}(\alpha\circ\beta)$, we conclude that $\sum_{\alpha\in\ubar{\mathcal{V}}}\lann_\mathcal{A}(\alpha)\mathcal{A}\subseteq I'$ and so $\sum_{\alpha\in\ubar{\mathcal{V}}}\lann_\mathcal{A}(\alpha)\mathcal{A}= I'$. It follows that assertion~(b) holds if, and only if, so does assertion~(c). 
\end{proof}

\begin{rem}
\label{rem.leftannihilator-in-I}%
Note that, if $\calI $ is an idempotent ideal of $\mathcal{A}$, and $\ubar{\mathcal{V}}$ is the corresponding set of fp-detecting morphisms in $\what{\mathcal{A}}$, then we always have the inclusion:
\[
\sum_{\what{b}\in\Obj(\what{\mathcal{A}})}\sum_{\eta\in\what{\mathcal{A}}(\what{b},\what{b})}\Fix(\eta )\mathcal{A}=\smash[b]{\sum_{\alpha\in\ubar{\mathcal{V}}}\lann_\mathcal{A}(\alpha)\mathcal{A}}\subseteq\calI.
\]
It is the converse inclusion that makes $\mathcal{X}_I\cong\mathcal{G}_I$ locally finitely presented. Indeed, given $\alpha\colon\what{a}\to\what{b}$  in $\ubar{\mathcal{V}}$, it is clear that $\smash[b]{\ubar{\beta}\circ\ubar{\alpha}}=0$ in $\what{\mathcal{A}}/\what{\calI}$, for all $\beta\colon\what{b}\to d$ such that $\beta\circ\alpha =0$. Since $\ubar{\alpha}$ is a retraction, this means that $\smash[b]{\ubar{\beta}}=0$, i.e., $\beta\in\what{\calI}(\what{b},d)$. Thus, $\lann_\mathcal{A}(\alpha )\subseteq \calI$ and, therefore, also $\lann_\mathcal{A}(\alpha)\mathcal{A}\subseteq \calI$.
\end{rem}

Recall that if $\alpha\in\what{\mathcal{A}}(\what{a},\what{b})$ then a \emph{pseudocokernel} or a \emph{weak cokernel} of $\alpha$ is a morphism $\alpha'\colon\what{b}\to\what{c}$ in $\what{\mathcal{A}}$ such that $\alpha'\circ\alpha =0$ and any $\beta\in\what{\mathcal{A}}(\what{b},\what{d})$ such that $\beta\circ\alpha =0$ satisfies that $\beta =\gamma\circ\alpha'$, for a (not necessarily unique) morphism $\gamma\in\what{\mathcal{A}}(\what{c},\what{d})$. We leave to the reader the easy verification that $\alpha'$ is a pseudocokernel of $\alpha$ if, and only if, $\lann_\mathcal{A}(\alpha)=\mathcal{A}\alpha'$. 

\begin{defn}
\label{def.coherentcategory}%
A small preadditive category $\mathcal{A}$ will be called \emph{left coherent} when $\what{\mathcal{A}}$ has pseudocokernels, equivalently, when $\mathcal{A}\lMod\cong\what{\mathcal{A}}\lMod$ is a locally coherent Grothendieck category (see \cite[Corollary~1.11]{PSV}).
\end{defn}

We then get the following immediate consequence of Theorem~\ref{thm.finitepresentability-on-ideal}.

\begin{cor}
\label{cor.I generated by pseudocokernels}%
Let $\mathcal{A}$ be a left coherent preadditive category and $I$ be an idempotent ideal of it. Then $\mathcal{G}_I$ is locally finitely presented if and only if the ideal $\calI$ is generated by the pseudocokernels of fp-detecting morphisms, and if and only if $\calI$ is generated by pseudocokernels of morphisms of the form $1_{\what{b}}-\eta$, where $\what{b}\in\Obj(\what{\mathcal{A}})$ and $\eta\in\what{\calI}(\what{b},\what{b})$.
\end{cor} 

\subsection{When $\mathcal{A}$ is a ring}%
When $\mathcal{A}=R$ is a ring, i.e.\ a preadditive category with just one object, and $I$ is an idempotent ideal of $R$, $\ubar{\mathcal{V}}$ may be identified with a set of matrices with entries in $R$. Let us call a matrix $A\in\mathcal{M}_{m\times n}(R)$ an \emph{fp-detecting matrix (with respect to $I$)} when the associated morphism in $\rMod R$
\[
A\colon R^n \longrightarrow R^m \quad\text{such that}\quad X=(x_1,\dots,x_n)^t \longmapsto AX
\]
is fp-detecting, in which case we put $A\in\ubar{\mathcal{V}}$. On the other hand, in this particular situation the equivalent of a morphism $\eta\in\what{\calI}(\what{b},\what{b})$ for an object $\what{b}\in\what{\mathcal{A}}$ is just a matrix $E\in\mathcal{M}_{n\times n}(I)$. Then we can identify $\Fix(\eta)$ with the submodule $\Fix(E)$ of the left $R$-module $R^n$ that consists of those $X\in R^n=\mathcal{M}_{1\times n}(R)$ such that $X=XE$.

\begin{lem}
\label{lem.fpdetecting-ringcase}%
Let $R$ be a ring, $I$ be an idempotent ideal of $R$ and let $\ubar{\mathcal{V}}$ be the set of fp-detecting matrices (with respect to $I$). For any matrix $A\in\mathcal{M}_{m\times n}(R)$, the following assertions hold:
\begin{enumerate}[(i)]
\item $I_A\leqdef\lann_R(A)R$ is the two-sided ideal of all those $r\in R$ that can be expressed as finite sums of matrix products $r=\sum_{k=1}^tX_kY_k$, where $X_k\in\mathcal{M}_{1\times m}(R)$, $Y_k\in\mathcal{M}_{m\times 1}(R)$ and $X_kA=0$, for $k=1,\ldots,t$.
\item $A\in\ubar{\mathcal{V}}$, i.e., $A$ is fp-detecting if, and only if, the following conditions hold:
\begin{enumerate}[(1)]
\item $A$ is right invertible modulo $I$, i.e.\ there is $B\in\mathcal{M}_{n\times m}(R)$ such that $1_m-AB$ has all its entries in $I$;
\item for each column-vector
$Y=(\begin{matrix}
		y_1 & \ldots & y_m	
	\end{matrix})^t\in I^m$, the equation $AX=Y$ is solvable in $R^n$ if and only if it has a solution in $I^n$.
\end{enumerate}
\end{enumerate}
\end{lem}
\begin{proof}
\smallskip \noindent(i) It is just an interpretation in this particular case of what the ideal $\lann_\mathcal{A}(\alpha )\mathcal{A}$ means for a morphism $\alpha$ in $\what{\mathcal{A}}$ (see Corollary~\ref{cor.idealgenerated-annihilator}).

\smallskip \noindent(ii) Here we only need to check that condition~(ii.2) is exactly the translation to this context of condition~(b.ii) in Proposition~\ref{prop.fp(X)}. Indeed, in this context $\what{a}$ and $\what{b}$ are coproducts of, say, $n$ and $m$ copies of the unique object of $R$, when we view $R$ as a preadditive category with just one object, that we shall denote by $x$. It follows that the morphism $\alpha\colon\what{a}\to\what{b}$ may be interpreted as a matrix $A\in\mathcal{M}_{m\times n}(R)$. Then $(\what{\calI}(-,\what{b}):\alpha)$ consists of the morphisms $\beta\in \what{R}(x,\what{a})\cong\Hom_R(R,R^n)$ such that $\alpha\circ\beta\in\what{\calI}(x,\what{b})$. But such a $\beta$ may then be viewed as a $B\in\mathcal{M}_{n\times 1}(R)$ such that $AB\in\mathcal{M}_{m\times 1}(I)=I^m$. 

The equality $(\what{\calI}(-,\what{b}):\alpha)=\what{\calI}(-,\what{a})+\rann_R(\alpha)$ says that $\beta =\eta +\gamma$, where $\eta\in\what{\calI}(x,\what{a})\cong I^n=\mathcal{M}_{n\times 1}(I)$ and $\gamma\in\what{\calI}(x,\what{a})\cong\Hom_R(R,R^n)\cong\mathcal{M}_{n\times 1}(R)$ satisfies that $\alpha\circ\gamma =0$. Notice that this is equivalent so say that there exists an $\eta\in\what{\calI}(x,\what{a})$ such that $\alpha\circ\beta =\alpha\circ\eta$. Putting now $Y\leqdef AB\in I^m$, this is saying that there exists a $n\times 1$ matrix $E$ with entries in $I$ (i.e., the one corresponding to $\eta$) such that $AE=Y$. 

All in all, condition~(b.ii) in Proposition~\ref{prop.fp(X)} translates in this context to the fact that if $Y\in\calI^m$ is any column vector with entries in $I$, then the matrix equation $AX=Y$, with unknown a $n\times 1$ matrix $X$, has a solution in $R^n$ if and only if it has a solution in $I^n$.
\end{proof}

\begin{rem}
\label{purity}%
The reader will have noticed that the truth of condition~(ii.2) of last lemma for all matrices $A$ is equivalent to the purity of $I$ on the left (see \cite[Proposition~I.11.2]{St}). It is not hard to see that this extends to the general case, i.e.\ when $\mathcal{A}$ is a small preadditive category and $\calI$ is an idempotent ideal of it. Indeed, we leave to the reader checking that condition~(b.ii) of Proposition~\ref{prop.fp(X)} holds for all morphisms $\alpha$ in $\what{\mathcal{A}}$ if, and only if, $\calI$ is pure on the left, i.e.\ $\rho_M\colon M\otimes_\mathcal{A}\calI\to M$ is a monomorphism for all (finitely presented) right $\mathcal{A}$-modules $M$. This is in turn equivalent to the equality of subcategories $\mathcal{C}_{\calI}=\mathcal{X}_{\calI}$. When $\calI$ is such an ideal, the fp-detecting morphisms (resp., matrices) are exactly the ones that are right invertible module $\calI$.
\end{rem}

As a consequence of Lemma~\ref{lem.fpdetecting-ringcase} and Theorem~\ref{thm.finitepresentability-on-ideal}, we get the following:

\begin{prop}
\label{prop.finitepresentability-ringcase}%
Let $R$ be a ring, $I\leq R$ an idempotent ideal, and let $\ubar{\mathcal{V}}$ be the corresponding set of fp-detecting matrices (see Lemma~\ref{lem.fpdetecting-ringcase}). Then, the following assertions are equivalent:
\begin{enumerate}[(a)]
\item The  Giraud subcategory $\mathcal{G}_I$ of $\rMod R$ is locally finitely presented;
\item $I=\sum_{A\in\ubar{\mathcal{V}}}I_A$, where if $A\in\mathcal{M}_{m\times n}(R)$ then $I_A$ is the two-sided ideal of $R$ that consists of the $r\in R$ that can be expressed as finite sums of matrix products $r=\sum_{k=1}^tX_kY_k$, where $X_k\in\mathcal{M}_{1\times m}(R)$, $Y_k\in\mathcal{M}_{m\times 1}(R)$ and $X_kA=0$, for $k=1,\ldots,t$.
\item $I=\sum_{n>0}\sum_{E\in\mathcal{M}_{n\times n}(I)}\Fix(E)R$, where $\Fix(E)R$ is the two-sided ideal of $R$ that consists of the $r\in R$ that can be expressed as finite sums of matrix products $r=\sum_{k=1}^tX_kY_k$, where $X_k\in\mathcal{M}_{1\times n}(R)$, $Y_k\in\mathcal{M}_{n\times 1}(R)$ and $X_k=X_kE$ for $k=1,\ldots,t$.
\end{enumerate}
\end{prop}

\section{Cuadra-Simson's problem versus the telescope conjecture}%
The Telescope Conjecture (TC) was proposed by Ravenel (\cite{Rav}) for the homotopy category of spectra $\operatorname{Ho}(\mathbf{Sp})$ and has only recently been refuted (see \cite{BHLS}). Transferred to the general context of compactly generated triangulated categories, of which $\operatorname{Ho}(\mathbf{Sp})$ is an example, it asserts the following:

\begin{TC}
If $\mathcal{D}$ is a compactly generated triangulated category and $\mathbf{t}\leqdef(\mathcal{X},\mathcal{Y})$ is a smashing semiorthogonal decomposition of $\mathcal{D}$, then $\mathbf{t}$ is compactly generated (see Subsection~\ref{subs.triangulcats}).
\end{TC}

Krause (\cite{Krause1}) proposed a study of the conjecture via the analysis of certain ideals of the (skeletally small) subcategory $\mathcal{D}^c$. Concretely, for $\mathcal{D}$ compactly generated and $\mathbf{t}\leqdef(\mathcal{X},\mathcal{Y})$ a smashing semiorthogonal decomposition of $\mathcal{D}$, Krause defined the ideal $\calI_\mathbf{t}$ of $\mathcal{D}^c$ consisting of the morphisms $\alpha\colon C\to C'$ in $\mathcal{D}^c$ such that $\mathcal{D}(\alpha,Y)\colon\mathcal{D}(C',Y)\to\mathcal{D}(C,Y)$ is the zero map, for all $Y\in\mathcal{Y}$. It turns out that $\mathbf{t}$ is compactly generated precisely when $\calI_\mathbf{t}$ is generated by identity morphisms of objects in $\mathcal{D}^c$, so that TC can be re-stated as follows:

\begin{TC}[Reformulated]
\label{TC-restatement}%
If $\mathcal{D}$ is a compactly generated triangulated category and $\mathbf{t}\leqdef(\mathcal{X},\mathcal{Y})$ a smashing semiorthogonal decomposition in $\mathcal{D}$, then the ideal $\calI_\mathbf{t}$ is generated by identity morphisms of objects in $\mathcal{D}^c$.
\end{TC}

As shown by Krause, $\mathbf{t}$ is recovered from the ideal $\calI_\mathbf{t}$. Namely $\mathcal{Y}$ consists of the objects $Y\in\mathcal{D}$ such that, for all $C,\, C'\in\mathcal{D}^c$ and all $\alpha\in\calI_\mathbf{t}(C,C')$,  $\mathcal{D}(\alpha ,Y)\colon\mathcal{D}(C',Y)\to\mathcal{D}(C,Y)$ is the zero map. In particular, the smashing semiorthogonal decompositions in $\mathcal{D}$ are precisely the semiorthogonal decompositions with definable co-aisle. This observation led to a generalized version of the conjecture for general t-structures.

\begin{NTC}
Exactly as the stable version (or, equivalently, its reformulation), but with $\mathbf{t}$ a t-structure in $\mathcal{D}$ with definable co-aisle. 
\end{NTC}

We point out that, when $\mathcal{D}$ has a model which allows for a definition of homotopy colimits, e.g., when it is the base of a strong stable derivator or the homotopy category of a stable $\infty$-category, a co-aisle in $\mathcal{D}$ is definable if and only if it is closed under directed homotopy colimits and pure monomorphisms; furthermore, if $\mathcal{D}$ is algebraic and compactly generated, then the definable co-aisles are precisely those that are closed under directed homotopy colimits (see \cite[Remark~8.9]{SS}). With this in mind, the main result of \cite{HN22} provides a positive answer to NTC when $\mathcal{D}$ is the derived category of a commutative noetherian ring. Similarly, in \cite{Sabatini} NTC is proved for the derived category of a path algebra $RQ$, where either $R$ is a commutative noetherian ring and $Q$ is a
Dynkin quiver, or $R$ is a commutative artinian ring and $Q$ is any finite
quiver.

TC (and consequently NTC) has long been known to be false even for derived categories of rings (see \cite{K2}). Then the study of these conjectures has moved to the problem of identifying classes of compactly generated triangulated categories $\mathcal{D}$ for which they hold, or alternatively, identifying some classes of smashing semiorthogonal decompositions (resp., t-structures with definable co-aisles) for which TC (resp., NTC) is true. 

The aim of this subsection is to show a surprising connection between the Cuadra--Simson Problem and these telescope conjectures. The crucial point for the connection is the following.

\begin{prop}[{\cite{Krause1}, \cite{SS}}]%
\label{prop.TC ideal}%
If $\mathbf{t}=(\mathcal{X},\mathcal{Y})$ is a t-structure with definable co-aisle (e.g., a smashing semiorthogonal decomposition) in $\mathcal{D}$, then the ideal $\calI_\mathbf{t}$ of $\mathcal{D}^c$ is idempotent. 
\end{prop}

\begin{rem}
It is proved in \cite{LMV} that the above result is also true for an arbitrary torsion pair with definable co-aisle, thus giving rise to what the authors call the strong telescope property. So far we have not been able to prove our main result in this section, Theorem~\ref{thm.CuadraSimson-Telescope}, in that generality.
\end{rem}

In order to prove the announced connection, we need the following:

\begin{lem}\label{lem.auxiliary for TC}%
Let  $\mathbf{t}=(\mathcal{X},\mathcal{Y})$ be a t-structure  with definable co-aisle (e.g., a smashing semiorthogonal decomposition) in  a compactly generated triangulated category $\mathcal{D}$, let $\calI_\mathbf{t}$ be the associated idempotent ideal of $\mathcal{D}^c$, and denote by $\tau^{\leq0}\colon\mathcal{D}\to\mathcal{X}$ and $\tau^{>0}\colon\mathcal{D}\to\mathcal{Y}$ the associated truncation functors. Then, the following assertions hold:
\begin{enumerate}[(i)]
\item $\calI_\mathbf{t}$ consists of the morphisms $\alpha$ in $\mathcal{D}^c$ such that $\tau^{>0}(\alpha)=0$;
\item If $\alpha$ is morphism in $\mathcal{D}^c$ that is mapped onto an isomorphism in $\mathcal{D}^c\!/\calI_\mathbf{t}$, then $\tau^{>0}(\alpha)$ is an isomorphism;
\item If $\alpha$ is as in assertion~\textup{(ii)}, then $\operatorname{Cone}(\alpha)\in\mathcal{X}\cap\mathcal{D}^c$.
\end{enumerate}
\end{lem}
\begin{proof}
\smallskip \noindent(i) Let $\alpha\colon C\to C'$ be a morphism in $\mathcal{D}^c$. The adjunction isomorphism $\mathcal{Y}(\tau^{>0}C,\tau^{>0}C')\mathrel{\smash[t]{\buildrel\cong\over\to}}\mathcal{D}(C,\tau^{>0}C')$ takes $\tau^{>0}(\alpha)\mapsto \tau^{>0}(\alpha)\circ p_C=p_{C'}\circ\alpha$, where $p_D\colon D\to\tau^{>0}D$ is the canonical morphism, for all $D\in\mathcal{D}$. But if $\alpha\in\calI_\mathbf{t}$ then $p_{C'}\circ\alpha=0$, which implies that $\tau^{>0}(\alpha)=0$, as desired.

\smallskip \noindent(ii) If $\alpha\colon C\to C'$ is a morphism in $\mathcal{D}^c$ such that $\ubar{\alpha}\leqdef p(\alpha)$ is an isomorphism in $\mathcal{D}^c\!/I_\mathbf{t}$,  there is a morphism $\beta\colon C'\to C$ such that $1_{C'}-\alpha\circ\beta\in\calI_\mathbf{t}(C',C')$ and $1_{C}-\beta\circ\alpha\in\calI_\mathbf{t}(C,C)$. Applying $\tau^{>0} $ to these equalities and using assertion~(i), we see that $\tau^{>0}(\alpha)$ and $\tau^{>0}(\beta)$ are mutually inverse isomorphisms.

\smallskip \noindent(iii) Consider the following morphism of triangles:
\[
\xymatrix@R=18pt{%
	\tau^{\le0}C \ar[r]\ar[d]_-{\tau^{\le0}\alpha} & C \ar[r]\ar[d]^-{\alpha} &
		\tau^{>0}C \ar[r]^-{+}\ar[d]_-{\cong}^-{\tau^{>0}\alpha} & {} \cr
	\tau^{\le0}C' \ar[r]
    & C' \ar[r]
    & \tau^{>0}C' \ar[r]^-{+} & {} & {}
}
\]
Since $\tau^{>0}(\alpha)$ is an isomorphism, the commutative square on the left is a homotopy pushout. To see this, apply the dual of \cite[Lemma~1.4.3]{N} and the functoriality of $\tau^{\leq 0}$. Hence, there is an isomorphism $\operatorname{Cone}(\tau^{\leq 0}(\alpha))\cong\operatorname{Cone} (\alpha)$. But $\operatorname{Cone}(\tau^{\leq 0}(\alpha))\in\mathcal{X}$, and so $\operatorname{Cone} (\alpha)\in\mathcal{X}\cap\mathcal{D}^c$. 
\end{proof}

We are now ready to prove the surprising mentioned connection.

\begin{thm}\label{thm.CuadraSimson-Telescope}%
Let $\mathcal{D}$ be a compactly generated triangulated category, let $\mathcal{A}$ be a skeleton of $\mathcal{D}^c$, let $\mathbf{t}=(\mathcal{X},\mathcal{Y})$ be a t-structure with definable co-aisle (e.g., a smashing semiorthogonal decomposition) of $\mathcal{D}$,  and let $\calI_\mathbf{t}$ be the associated idempotent ideal of $\mathcal{A}$. Then, the following assertions are equivalent:
\begin{enumerate}[(a)]
\item $\mathbf{t}$ is compactly generated.
\item The Giraud subcategory $\mathcal{G}_{I_\mathbf{t}}$ of $\rMod\mathcal{A}$ is locally finitely presented.
\end{enumerate}
\end{thm}
\begin{proof}

Recall that assertion~(a) holds if, and only if, the ideal $\calI_\mathbf{t}$ is generated by identity morphisms of objects in $\mathcal{A}$. Then the implication \smallskip \noindent``$\textup{(a)}\Rightarrow\textup{(b)}$'' follows from Theorem~\ref{thm.always enough flats} and Proposition~\ref{prop.traces-versus-idempotents}.

\smallskip \noindent``$\textup{(b)}\Rightarrow\textup{(a)}$'' For each $B\in\Obj(\mathcal{A})$ and each $\eta\in\calI_{\mathbf{t}}(B,B)$, the endomorphism $\alpha\leqdef 1_B-\eta$ clearly satisfies the hypothesis of assertions~(ii) and~(iii) of Lemma~\ref{lem.auxiliary for TC}, so that $C_\eta\leqdef\operatorname{Cone}(1_B-\eta)\in\mathcal{X}\cap\mathcal{D}^c$. Up to replacement by an isomorphic object, we can assume that $C_\eta\in\mathcal{A}$, so that we have a triangle
\[
	B\buildrel 1_B-\eta\over\longrightarrow
	B\buildrel{\gamma_\eta}\over\longrightarrow
	C_\eta\buildrel+\over\longrightarrow{}
\]
and $\gamma_\eta$ is a pseudocokernel of $1_B-\eta$ in $\mathcal{A}$. By Corollary~\ref{cor.I generated by pseudocokernels}, assertion (b) is equivalent to the fact that $\calI_{\mathbf{t}}$ is generated by the morphisms of the form $\gamma_\eta\in \A(B,C_\eta)$, for all $B\in\Obj(\mathcal{A})$ and all $\eta\in\calI_{\mathbf{t}}(B,B)$. On the other hand, it is clear that $\gamma_\eta=1_{C_\eta}\circ\gamma_\eta\in 1_{C_\eta}\A\subseteq \calI_{\mathbf{t}}$. In other words, $\calI_\mathbf{t}$ is generated by identity morphisms of objects in $\mathcal{A}$.
\end{proof}

\end{document}